\documentclass{article}%
\usepackage[left=4cm, right=4cm, top=3.5cm,bottom=4.8cm]{geometry}

\usepackage{amsfonts}
\usepackage{amsmath}
\usepackage{amssymb}
\usepackage{graphicx}
\usepackage{comment}
\usepackage{ntheorem}
\usepackage[citation-order,msc-links]{amsrefs}%
\newtheorem{theorem}{Theorem}

\newtheorem{definition}{Definition}
\newtheorem{lemma}{Lemma}

\theoremstyle{remark}
\newtheorem{remark}{Remark}

\usepackage{amsfonts,mathrsfs}
\usepackage{hyperref}
\usepackage{comment}
\usepackage{amsmath}
\usepackage{amssymb}
\newenvironment{proof}[1][Proof]{\noindent\textbf{#1.} }{\ \rule{0.5em}{0.5em}}
\newcommand{\eps}{\varepsilon}

\renewcommand{\marginpar}[1]{}
\newcommand{\correctiona}[1]{#1}

\begin{document}

\title{Arnold Diffusion, Quantitative Estimates and Stochastic Behavior in the Three-Body Problem}
\author{Maciej J. Capi\'nski \smallskip\\Faculty of Applied Mathematics\\AGH University of Science and Technology \\al. Mickiewicza 30, 30-059 Krak\'ow, Poland \bigskip\\Marian Gidea \smallskip\\Department of Mathematical Sciences\\Yeshiva University\\New York, NY 10016, USA
}
\maketitle

\begin{abstract}

We consider a class of autonomous Hamiltonian systems subject to
small, time-periodic perturbations. When the perturbation parameter is set to
zero, the energy of the system is preserved. This is no longer the case when
the perturbation parameter is non-zero.

We describe a topological method to establish orbits which diffuse in energy for every suitably small perturbation parameter $\varepsilon>0$.
The method yields quantitative
estimates:
(i)~The existence of orbits along which the energy  drifts by an amount independent of $\varepsilon$.
The time required by such orbits to drift is $O(1/\varepsilon)$;
(ii)~The existence of orbits along which the energy makes chaotic excursions;
(iii)~Explicit estimates for the  Hausdorff dimension of the set of such chaotic orbits;
(iv)~\correctiona{The existence of orbits
%, with initial conditions in some suitable sets of positive Lebesgue measure,
along which the time evolution of energy approaches a stopped  diffusion process (Brownian motion with drift)}, as  $\varepsilon$ tends to $0$. \correctiona{For each $\varepsilon$ fixed, the set of initial conditions of the orbits that yield the diffusion process has positive Lebesgue measure, and in the limit the measure of these sets approaches zero. }

Moreover, we can obtain any desired values of the drift and  variance for the limiting Brownian motion, for appropriate sets of initial conditions.

A key feature of our topological method is that it can be implemented in
computer assisted proofs.
%As an application, we show the existence of Arnold diffusion, and provide quantitative estimates, in a concrete model of the planar elliptic restricted three-body problem describing the motion of an infinitesimal body relative to the Neptune-Triton system.
\correctiona{We give an application  to a concrete model of the planar elliptic restricted three-body problem, on the motion of an infinitesimal body relative to the Neptune-Triton system.}

\end{abstract}
\newpage
\tableofcontents

%TCIDATA{Version=5.00.0.2606}
%TCIDATA{LaTeXparent=0,0,MMFedit.tex}

\section{Introduction}

\subsection{Motivation}

%The  Arnold diffusion phenomenon constitutes a universal mechanism of global instability in Hamiltonian dynamics. Through this mechanism,  arbitrarily small perturbations of integrable systems can produce  large effects over time, reflected by orbits that move chaotically in phase space, and explore macroscopic regions of the action-space.  The  actions of some orbits can grow in the long run  to levels  much higher than the size of  the perturbation. While Arnold illustrated this mechanism of instability in a particular example,  he conjectured that  it \emph{``is applicable to the general case (for example, to the problem of three bodies)''} \cite{Arnold64}. Quantitative properties of this mechanism were first explored by Chirikov \cite{Chirikov79}, who conjectured that the energy growth follows a diffusion process. For this reason he   coined   the term  ``Arnold diffusion'' to describe this instability phenomenon. We will refer to the underlying orbits as `diffusing orbits'.
%
% The Arnold diffusion problem asserts that arbitrarily small perturbations of Hamiltonian systems  can produce large effects over time, leading to global instability.
%\correction{The Arnold diffusion problem addresses the impact of arbitrarily small perturbations on Hamiltonian systems, which can produce large effects over time, leading to global instability. }
The Arnold diffusion problem concerns global instability in Hamiltonian systems under arbitrarily small perturbations,
which yield the existence of orbits undergoing large effects over time. While Arnold illustrated such an instability mechanism  in a particular example \cite{Arnold64},  he conjectured that it \emph{``is applicable to the general case (for example, to the problem of three bodies)''}. Quantitative properties of Arnold's mechanism were first explored by Chirikov \cite{Chirikov79}, who conjectured that the energy of perturbed orbits can follow a diffusion process. For this reason he coined  the term  ``Arnold diffusion'' to describe this behaviour.

%\correction{We consider an autonomous Hamiltonian system, in which energy is preserved, and study its behaviour under small time dependent perturbations. Even when arbitrarily small, such perturbations can produce large effects over time on the constants of motion of the unperturbed system, leading to its instability. While Arnold illustrated a mechanism which leads to this in his famous example \cite{Arnold64},  he conjectured that it \emph{``is applicable to the general case (for example, to the problem of three bodies)''}. Quantitative properties of Arnold's mechanism were first explored by Chirikov \cite{Chirikov79}, who conjectured that the energy of perturbed orbits can follow a diffusion process. For this reason he coined  the term  ``Arnold diffusion'' to describe this  phenomenon.}

\correctiona{In this paper, we develop a topological method to study Arnold diffusion in concrete models,
under verifiable conditions.
We consider an autonomous Hamiltonian system, in which the energy is preserved,
and study its dynamics under small, time dependent perturbations.
We show that for every  value of the perturbation parameter within some range,
there exist orbits that drift in energy, as well as orbits that undergo chaotic excursions in energy, over an explicit energy range which is
independent of the size of the perturbation. We will refer to them as `diffusing orbits'}. Further,  we extract quantitative information on the diffusing orbits: an explicit estimate on the diffusion time; \correctiona{an  explicit estimate on the Lebesgue measure  of  orbits that drift in energy for every fixed perturbation parameter};
an explicit estimate on the Hausdorff dimension of chaotic orbits; and  an explicit description of the limiting stochastic process -- as the perturbation parameter tends to zero -- associated to diffusing orbits. %We underscore that our estimates are expressed as explicit formulas depending on the perturbation parameters with explicit bounds on all constants involved.

%Our results  give answers to the two conjectures formulated  by Arnold and Chirikov, mentioned above.

\correctiona{Some }of the existing results on Arnold diffusion are devoted to arbitrarily small perturbations  of  `generic type' in `generic systems'. The novelty of our work is that we obtain a mechanism of diffusion that can also be applied to \correctiona{concrete systems}, for concrete perturbations, and for perturbation parameters  from arbitrarily close to zero up to some physically relevant value.

Besides deriving our results via a traditional mathematical proof, we produce a method that %can lead to
\correctiona{is implementable in }computer assisted proofs which use validated numerical methods.
%This relies on
\correctiona{We develop }a topological method which  allows us to obtain --  by verifying only a finite number of explicit conditions up to finite precision -- asymptotic results for infinitely many parameter values (an interval with one endpoint at zero), and for infinite sets of initial conditions (of positive Lebesgue measure, or of large Hausdorff dimension).

%We apply this method to study diffusion in the elliptic restricted three-body problem (PER3BP) in a concrete example: the motion of a small body (e.g., asteroid or spaceship) relative the Neptune-Triton system.  We regard the eccentricity of the orbit of the moon Triton as a perturbation parameter, where  the unperturbed case corresponds to the planar circular restricted three-body problem (PCR3BP). Thus, both the unperturbed system and the perturbations are concrete. We let the perturbation parameter range between zero and the true value of the eccentricity of the orbit of the moon. There is nothing special about this system, except that the eccentricity of Triton is small. This makes it a suitable example to study when treating the eccentricity as the perturbation parameter. We chose this model to work with realistic parameters, such as mass ratio and eccentricity of the moon orbit.

\correctiona{We apply this method to study diffusion in the elliptic restricted three-body problem (PER3BP), on the motion of a small body (e.g., asteroid or spaceship) under the gravity exerted by two larger bodies moving on Keplerian ellipses.
%We apply this method to study diffusion in the elliptic restricted three-body problem (PER3BP) where the motion of a small body (e.g., asteroid or spaceship) moves in the gravitational pull of two large bodies. 
We regard the eccentricity of the orbits of the} large bodies as a perturbation parameter, where  the unperturbed case corresponds to the planar circular restricted three-body problem (PCR3BP). \correctiona{ We   consider the PER3BP for physically relevant values of the mass ratio of the large bodies and of the eccentricity of their orbits.}
%substantial values of the eccentricity which are of astronomical interest.
In particular, our results %cover
\correctiona{ hold for the observed values in the Neptune-Triton system (which has the smallest eccentricity in the Solar System).
%Even if this is the system with the smallest eccentricity in the Solar System, this shows that the results we obtain here are realistic.
We expect that with extra work we can extend the range of applicability. For simplicity, in the rest of the paper we will describe results for the PER3BP with the same mass ratio and  eccentricity as the Neptune-Triton system. Of course, since our results are
quite robust, it is to be expected that they also apply to models that incorporate more small effects (such as, the unrestricted three body problem with small mass of the small body).}

While  in this work we focus on the PER3BP, %for a very specific system,
the underlying topological method can be applied %to other three-body systems, and indeed
to other models, for instance, to time-dependent, generic perturbations of the geodesic flow, the so called
``Mather acceleration" problem.

\subsection{Brief description of the main results}\label{sec:brief_results}

Consider a class of  Hamiltonian systems of the form \[H_\varepsilon(x,y,I,\theta)=H_0(x,y,I,\theta)+\varepsilon H_1(x,y,I,\theta,t;\eps),\]
where $(x,y,I,\theta)$ takes values in some sub-domain of $\mathbb{R}^{n_u}\times \mathbb{R}^{n_s}\times  \mathbb{R}^{n_c}\times \mathbb{T}^{n_c}$,  $n_u=n_s>0$, $n_c>0$, and $t\in \mathbb{T}^1$. Here  $\mathbb{T}^k$, $k\geq 1$, stands for the $k$-dimensional torus, and the symplectic form is $\omega=dy\wedge dx + dI\wedge d\theta$.

Suppose that, for the unperturbed system, when $\varepsilon=0$, there exists a $(2n_c)$-dimensional normally hyperbolic invariant manifold $\Lambda_0$, corresponding to $x=y=0$, which can be described via action-angle coordinates $(I,\theta)$, where $I$ is an integral of motion; that is, $I=\textrm{const.}$ along  each trajectory in $\Lambda_0$. Suppose that the stable and unstable manifolds $W^s(\Lambda_0)$ and $W^u(\Lambda_0)$ of $\Lambda_0$ intersect transversally. The energy $H_0$ is preserved along trajectories; in particular, each homoclinic orbit is bi-asymptotic to the same action level set $I=\textrm{const.}$ in $\Lambda_0$.  There is no diffusion in the action variable $I$.

When we add the perturbation, i.e., we let $\varepsilon\in(0,\varepsilon_0]$, with $\varepsilon_0$  sufficiently small, we have persistence of the normally hyperbolic invariant manifold  $\Lambda_0$ to some manifold   $\Lambda_\varepsilon$.  The stable and unstable manifolds $W^s(\Lambda_\varepsilon)$ and $W^u(\Lambda_\varepsilon)$ of $\Lambda_\varepsilon$ continue intersecting transversally. Take a neighborhood of  $\Lambda_\varepsilon$  where the action variable $I$ is well defined,  and a family of return maps to that neighborhood. To show the existence of diffusing orbits, that is, orbits along which $I$ changes by $O(1)$, as well as symbolic dynamics relative to $I$, it is sufficient to  show that these properties can be achieved by iterating the return maps in suitable ways.

The class of systems considered above, when the unperturbed Hamiltonian has periodic/quasiperiodic orbits whose stable and unstable manifolds intersect transversally, is sometimes referred to as \emph{a priori chaotic}.

When we let $\varepsilon>0$ small, 
\correctiona{the system is not autonomous, so the energy is no longer preserved.}
There exist diffusing orbits that follow $W^u(\Lambda_\eps)$, $W^s(\Lambda_\eps)$  and return to a neighborhood of $\Lambda_\eps$ with an $O(\varepsilon)$-change in the action coordinate $I$, and implicitly an $O(\varepsilon)$-change in their energy. The effects of the small perturbation can accumulate over time. The existence of orbits whose $I$-coordinate explores a $O(1)$-region of the action-space  amounts to the Arnold diffusion phenomenon in this setting.

%We consider the PER3BP in the  concrete case of the Neptune-Triton system, where we denote by $\varepsilon_0$ the true value of the eccentricity. We prove the following results:
Our methods allow us to prove the following results:
\begin{itemize}
\item  There exists  an explicit  constant $C_{h_0}>0$, independent of the perturbation, such that, for each $\varepsilon\in (0,\varepsilon_0]$, there exist orbits for which $\|I(t(\varepsilon))-I(0)\|\geq C_{h_0}$, for some $t(\varepsilon)>0$. We provide estimates on the Lebesgue measure of such orbits.
    Moreover, we show that the corresponding diffusion time $t(\varepsilon)$ satisfies $t(\varepsilon)\leq T/\varepsilon$ for some explicit constant $T>0$. This order of  time for diffusion is optimal for the class of a priori chaotic systems considered in this paper.
\item Given a sequence of action level sets $(I^\sigma)_{\sigma\geq 0}$, with $\|I^{\sigma+1}-I^\sigma\|>2\eta\varepsilon$, for some suitable $\eta>0$, there exists an orbit with $\|I(t^\sigma)-I^\sigma\|<\eta\varepsilon$, for some increasing sequence of times  $t^\sigma>0$ and all $\sigma\geq 0$;
\item The set of all initial points whose orbits $(\eta\varepsilon)$-shadow, in terms of the action variable $I$, any prescribed sequence of action level sets  $(I^\sigma)_{\sigma\geq 0}$ as above, has Hausdorff dimension  strictly greater than $4$.
\item Consider the stochastic process
$X_{t}^{\varepsilon }(z)$ given by the  energy path along an orbit starting from $z$, with appropriately rescaled time $t$.
Then, for every choice of $\mu\in \mathbb{R}$, $\sigma\in\mathbb{R}^+$, there exists a set of points $z\in\Omega_\eps$, of positive Lebesgue measure, such that $X_{t}^{\varepsilon}$ converges  to $\mu t+\sigma W_{t}$ as $\eps\to 0$, where $W_{t}$ is the standard Brownian motion. The convergence is in the sense of the functional central limit theorem (see Section \ref{sec:stochastic}).
\end{itemize}

The situation described above is present in the PER3BP.
In this case $n_u=n_s=n_c=1$. The PER3BP can be described as a time-periodic Hamiltonian  perturbation of the autonomous  Hamiltonian associated to the PCR3BP, with the perturbation parameter $\varepsilon$ equal to the eccentricity of the orbits of the primaries. In the PCR3BP,
we choose one of the equilibrium points of center-saddle type.
The typical geometric picture is the following.
There is  a family of Lyapunov periodic orbits around this point that  forms a $2$-dimensional   normally hyperbolic invariant manifold (NHIM) $\Lambda_0$.
The dynamics restricted to $\Lambda_0$ is integrable and can be described in action-angle coordinates $(I,\theta)$ on $\Lambda_0$; each Lyapunov orbit corresponds to a unique value of $I$. The action coordinate $I$ is uniquely determined by the energy of the PCR3BP.
The stable and unstable manifolds $W^s(\Lambda_0)$ and $W^u(\Lambda_0)$ turn around the main bodies and intersect
along transverse homoclinic orbits. After the perturbation $\Lambda_0$ is perturbed to $\Lambda_{\varepsilon}$, with a Cantor set of Lyapunov orbits surviving as KAM tori in the extended phase space \cite{CapinskiGL14}.

\subsection{Methodology}

\correctiona{Our argument has several components  that involve different tools.
We first use use topological methods
to identify a mechanism that, given some conditions verifiable by finite computations
in a concrete system, leads to the existence of orbits which diffuse and for which we have symbolic dynamics in energy. We also use some stochastic analysis and geometry to
show that the topological description leads to a set of
orbits -- with initial conditions of positive Lebesgue measure for finite time approximation and a fixed perturbation parameter, 
and of positive Hausdorff dimension for arbitrarily long time -- whose dynamics can be described as a stochastic process. In other words, we use traditional
mathematical methods to show that some structures verifiable with finite precision,
lead to dynamical, stochastic and geometric consequences. 

%In a second part, we use computer assisted proofs to verify these conditions in concrete problems of practical interest including the Neptune-Triton system. 

In a second part, we use
computer assisted proofs to verify these conditions in a concrete model of
%practical
astronomical interest, namely the Neptune-Triton system.
The code and the documentation for the computer assisted proof is available on the web page of the first author.}
%\correctionr{Since the geometric structures we consider are very robust, we expect that the conditions can be verified by other methods (e.g., perturbative expansions, genericity arguments, etc.).}

\subsection{Overview}
Our argument relies on topological methods and their implementation into computer assisted proofs. The main topological tool is \emph{correctly aligned windows } with \emph{cone conditions}; see \cite{GideaZ04a,Zcc}.

A window is  a multidimensional rectangle with some distinguished `topologically unstable' and `topologically stable' directions. Two windows are correctly aligned under some mapping if the image of the first window stretches across the second window along its `topologically unstable' directions. Given a family of cones, a mapping is said to satisfy a cone condition if it maps  cones in the family inside  cones in the family. The correct alignment of windows with cone conditions can be validated using rigorous numerics with interval arithmetic.

We construct windows inside some surfaces of sections  along the  NHIM and along its hyperbolic stable and unstable manifolds mentioned in Section \ref{sec:brief_results}. The NHIM and its hyperbolic stable and unstable  manifolds  induce stable and unstable hyperbolic coordinates as well as centre coordinates on the aforementioned surfaces of sections. The constructed windows have their `topologically unstable' directions aligned with the hyperbolic unstable directions, and their `topologically stable' directions aligned with the hyperbolic stable directions and centre directions combined. The correct alignment of windows is achieved by the appropriate section-to-section mappings.

To control the behavior of the windows in the centre directions,  in our theoretical arguments we consider topological discs inside the windows, which are aligned with their `topologically unstable' directions. These discs also satisfy cone conditions, they  are uniformly bounded in the centre directions. Given a finite sequence of correctly aligned windows with cone conditions, for every such disc inside the first window there exists a disc inside the last window, which is the image of a subset of the first disc. Consequently, there exist orbits that `shadow' the given sequence of windows. If the resulting disc in the last window of a sequence can be fitted inside the first window of another  sequence of windows, the `shadowing' can be continued through this new sequence of windows. In our constructions, we continue this `shadowing' process indefinitely by concatenating only a finite collection of finite sequences of correctly aligned windows. Using these discs is a novel contribution to the method of correctly aligned windows, which allows us  to obtain infinite `shadowing' for systems with center coordinates by using only finitely many windows.

In our construction for our diffusion results, certain sequences of correctly aligned windows from the collection yield a  growth in the action variable by $O(\varepsilon)$. Other sequences yield a  decay in the action by $O(\varepsilon)$. Checking the growth (resp. decay) in action by $O(\varepsilon)$ is done by computing the derivative with respect to $\varepsilon$ of the action along the underlying composition of section-to-section mappings, and showing that this derivative is larger than some positive constant (rep. smaller than some negative constant). The verification can be done via rigorous numerics. This is the key step that allows us to quantify the change in action along sequences of correctly aligned windows, for all $\varepsilon$ (including values arbitrarily close to $0$).

We can follow repeatedly the sequences of  windows that yield   growth in action until we obtain a growth of order $O(1)$. We can also alternate sequences of  windows that yield   growth in  action with sequences that achieve decay in   action to obtain symbolic dynamics, hence chaotic orbits.  The conclusions is the existence of diffusing orbits and symbolic dynamics, for every value of the perturbation parameter $\varepsilon \in(0,\varepsilon_0]$. We emphasize that our procedure involves the construction of only finitely many windows, and the rigorous verification of only finitely many correct alignments with cone conditions.

The orbits that achieve diffusion and symbolic dynamics have initial points lying on  $C^0$-families of discs. This allows us to obtain information about the Lebesgue measure and the Hausdorff dimension of such orbits.

Moreover, we use the symbolic dynamics to show the existence of orbits whose time-evolution of the action variable follows a random walk. For $\varepsilon>0$ fixed, the set of initial points that follow such a random walk for finitely many steps has positive Lebesgue measure. When we let $\varepsilon\to 0$, we obtain in the limit a stopped  diffusion process (Brownian motion with drift). Moreover, we can obtain any Brownian motion with drift that we wish, for suitable choices of sets of initial points.

\subsection{Related works}\label{sec:related_works}

In the recent years, the Arnold diffusion problem has taken a central role in the study of Hamiltonian dynamics.
Significant works,  using variational methods or geometric methods,  include, e.g.,
\cite{ChierchiaG94,Bessi96,BertiB02a,BolotinTreschev1999,Treschev02a,Treschev04,Treschev12,
DelshamsLS00,FontichM01,DelshamsLS2006,DelshamsLS2006b,GelfreichT2008,Piftankin2006,Kaloshin2003,
KaloshinBZ11,KaloshinZ12a,KaloshinZ12b,Zhang11,Bounemoura12,ChengY2004,ChengY2009,Cheng12,ChengX2015,DelshamsLS16,
GideaR13,gidea2019general,marco2016arnold,marco2016chains,gidea_marco_2017}.
There are also numerical approaches, including, e.g.,  \cite{tennyson1980diffusion,Tennyson82,Lieberman83,Lichtenberg92,Laskar93,
zaslavsky2005hamiltonian,GuzzoLF05,FroeschleLG06,guzzo2009numerical}.
Part of the interest in Arnold diffusion is owed to  the seminal  work of John Mather in the field \cite{Mather04,Mather12}, as well as to  possible applications to celestial mechanics, dynamical astronomy, particle accelerators, and plasma confinement; see, e.g., \cite{cincotta2002arnold,laskar2004long,LuqueP17}.

In particular, various mechanisms for Arnold diffusion in the $N$-body problem have been examined in several papers, including, e.g.,   \cite{xia1993arnold,moeckel1996transition,galante2011destruction,FejozGKR11,capinski2011transition,DelshamsGR13,xue2014arnold,CapinskiGL14,Delshams2018}.
%{\color{blue} Notably, \cite{Delshams2018} established the proof of Arnold diffusion in the PERTBP under the assumption that the eccentricity of the small body is sufficiently small and that mass parameter of one of the primaries is very small compared to the eccentricity. }
%\marginpar{MG: Commented out reference to Delshams et al. It's a different model. If we start to describe this model, we should describe other models as well.}
The present work is closely related to \cite{CapinskiGL14}, in which Arnold diffusion is shown in the PERTBP, using a shadowing lemma for NHIM's (\cite{gidea2019general}) and numerical arguments. The paper \cite{CapinskiGL14} assumes the existence of a NHIM, and does not provide quantitative estimates on diffusion; also, the numerical experiments are non-rigorous. In the present paper we develop  an entirely new methodology, which  does not need to assume or to verify the existence of a NHIM. Moreover, we provide quantitative estimates via computer assisted proofs.
In particular, we show that  diffusing orbits exist in the concrete system under consideration, with the given mass parameter and  up to the true value of the eccentricity.
%In particular, we show that  diffusing orbits exist up to the true value of the eccentricity of  the orbits in the concrete system under consideration.

Several works, including e.g. \cite{BertiB02b,BertiBB03,BessiCV01,Treschev04,GideaL06,LochakM05,Piftankin2006,GelfreichT2008,Zhang11,GideaL13,Guzzo2018}, obtain estimates on the diffusion time.
A novelty of our results is that we provide optimal estimates on the diffusing time
with explicit constants,  and for the full range of parameters under consideration.

Estimates on the Hausdorff dimension of the set of initial conditions for unstable orbits appear for example in \cite{GorodetskiK11}, where they study oscillatory motions in the Sitnikov problem and in the PCR3BP.
In this paper we provide estimates on the Lebesgue measure of  the set of initial conditions for orbits  that drift on energy, and on the Hausdorff dimension of  the set of initial conditions for orbits undergoing symbolic dynamics.

Related works that provide analytic results on the stochastic process followed by diffusing orbits in random iterations of the maps appear in, e.g., \cite{sauzin2006ergodicity,sauzin2006exemples,castejon2017random,kaloshin2015normally}.
There are also heuristic arguments and numerical work,   e.g., \cite{Chirikov79,lichtenberg2013regular,KaloshinRoldan}.
%In \cite{kaloshin2015normally} the authors consider a nearly integrable system given by the product of a pendulum and a rotator   with a small perturbation from some open set of trigonometric functions, and show the existence of a Normally Hyperbolic Invariant Lamination (NHIL) and of a probability measure supported on this NHIL that weakly converges to a   diffusion process,  as the perturbation parameter tends to zero. The result in \cite{kaloshin2015normally} applies to initial points whose action variable is outside certain non-resonant domains.
A novelty of this paper is that we provide rigorous results, via a computer assisted proof,  on the stochastic process underlying diffusing orbits,  as the perturbation parameter tends to zero.
%for a concrete system with a given, physical perturbation.
%We do not require any non-resonance conditions on the initial conditions on diffusing orbits.
Moreover,  we can obtain in the limit \emph{any Brownian motion with drift}, that is, any value of the drift and variance.

%TCIDATA{Version=5.00.0.2606}
%TCIDATA{LaTeXparent=0,0,MMFedit.tex}

\section{Statement of the main theorem}
\label{section:main_theorem}
%The particular system of interest in this paper is the Planar Elliptic Restricted Three Body Problem (PER3BP). This can be written as a perturbation of the Planar Circular Restricted Three Body Problem (PCR3BP).
%For the PER3BP, we show that there is Arnold diffusion in the energy, and obtain a detailed quantitative description of this phenomenon.
Our results on Arnold diffusion in the PER3BP are stated in Theorem \ref%
{th:main-3bp}. They are derived from the general results obtained in this paper
-- Theorems \ref{thm:diffusion}, \ref{th:symbolic-dynamics-from-covering}, %
\ref{th:Hausdorff-dim} and \ref{th:diffusion-process}. These general results
and the underlying methodology can be applied to other models, for instance
to time-dependent, generic perturbations of the geodesic flow \cite%
{DelshamsLS00,DelshamsLS2006b,GideaL13}.

\subsection{Model}

We first briefly introduce the model and then state the main result.
%More details on the PRE3BP model and on the implementation of our method  will be given in Section \ref{sec:application}.
Our model describes the motion of a massless particle (e.g., an asteroid or
a spaceship), under the gravitational pull of two large bodies, which we
call primaries. The primaries rotate in a plane along Keplerian elliptical
orbits with eccentricity $\varepsilon $, while the massless particle moves
in the same plane and has no influence on the orbits of the primaries. We
use normalized units, in which the masses of the primaries are $\mu $ and $%
1-\mu $. We consider a frame of `pulsating' coordinates that rotates
together with the primaries, making their position fixed on the horizontal
axis \cite{S}. The motion of the massless particle is described via the
Hamiltonian $H_{\varepsilon }:\mathbb{R}^{4}\times \mathbb{T\rightarrow R}$
\begin{equation}\label{eq:H-pre3bp}\begin{split}
H_{\varepsilon }\left( X,Y,P_{X},P_{Y},\theta \right) &=\frac{\left(
P_{X}+Y\right) ^{2}+\left( P_{Y}-X\right) ^{2}}{2}-\frac{\Omega \left(
X,Y\right) }{1+\varepsilon \cos (\theta )}, \\
\Omega \left( X,Y\right) & =\frac{1}{2}\left( X^{2}+Y^{2}\right) +\frac{%
\left( 1-\mu \right) }{r_{1}}+\frac{\mu }{r_{2}}, \\
r_{1}^{2}& =\left( X-\mu \right) ^{2}+Y^{2}, \\
r_{2}^{2}& =\left( X-\mu +1\right) ^{2}+Y^{2}.
\end{split}
\end{equation}

The corresponding Hamilton equations are:
\begin{equation}
\begin{array}{lll}
\displaystyle\frac{dX}{d\theta }=\frac{\partial H_{\varepsilon }}{\partial
P_{X}},\smallskip & \qquad \qquad & \displaystyle\frac{dP_{X}}{d\theta }=-%
\frac{\partial H_{\varepsilon }}{\partial X}, \\
\displaystyle\frac{dY}{d\theta }=\frac{\partial H_{\varepsilon }}{\partial
P_{Y}}, & \qquad \qquad & \displaystyle\frac{dP_{Y}}{d\theta }=-\frac{%
\partial H_{\varepsilon }}{\partial Y},%
\end{array}
\label{eq:PER3BP-ode}
\end{equation}%
where $X,Y\in \mathbb{R}$ are the position coordinates of the massless
particle, and $P_{X},P_{Y}\in \mathbb{R} $ are the associated linear
momenta. The variable $\theta \in \mathbb{T}$ is the true anomaly of the
Keplerian orbits of the primaries, where $\mathbb{T}$ denotes the $1$%
-dimensional torus. The system is non-autonomous, thus we consider it in the
extended phase space, of dimension $5$, which includes $\theta $ as an
independent variable. We use the notation $\Phi _{t}^{\varepsilon }$ to
denote the flow of (\ref{eq:PER3BP-ode}) in the extended phase space, which
includes $\theta \in \mathbb{T}$.

When $\varepsilon =0$ the corresponding Hamiltonian $H_{0}$ describes the
PCR3BP, and the variable $\theta $ represents the physical time. An
important feature of our model is that the Hamiltonian $H_{0}$ is
autonomous, hence the energy is preserved. The energy $H_{0}$ is no longer
preserved when $\varepsilon >0$. The main objective of this paper is to
investigate the changes in the energy when $\varepsilon >0$.

There are certain geometric structures that organize the dynamics. The
Hamiltonian vector field of $H_{0}$ has 5 equilibrium points. One of these
points, denoted by $L_{1}$, is located between the primaries, and is of
center-stable linear stability type. There exists a family of Lyapunov
periodic orbits about $L_{1}$. Each Lyapunov periodic orbit is uniquely
characterized by some fixed value of the energy $H_{0}$. The family of
Lyapunov orbits form a normally hyperbolic invariant manifold (NHIM) with
boundary in the phase space, which we denote as $\Lambda _{0}$. The NHIM has
associated stable and unstable manifolds, $W^{s}(\Lambda _{0}) $, $%
W^{u}(\Lambda _{0})$, respectively. For certain values of the mass ratio and
energy, one can observe numerically that $W^{s}(\Lambda _{0})$ and $%
W^{u}(\Lambda _{0})$ intersect transversally.

The NHIM $\Lambda _{0}$ and the homoclinic orbits to $\Lambda _{0}$ are
the geometric objects on which the computer assisted proof is based. We
emphasize that in the computer assisted proof we do not need an explicit
knowledge of these objects. We only use numerical approximations of these
objects to construct the tools -- correctly aligned windows with cone
condition -- which are used in the computer assisted proof.

\subsection{Main theorem}

Below we state our main theorem.

\begin{theorem}
\label{th:main-3bp} Consider the Neptune-Triton PER3BP, with mass ratio $\mu
=2.089\cdot 10^{-4}$, and orbit eccentricity $\eps_{0}=1.6\cdot 10^{-5}$.
Let $h_{0}$ be a fixed energy level, specified below, and $I=H_{0}-h_{0}$ be
the rescaled energy. We have the following results:

\begin{enumerate}
\item[(1)] (Diffusing orbits)\label{pt:main-diffusion} There exist $%
C_{h_{0}}>0$ and $T>0$ such that for every $\varepsilon \in (0,\eps_{0}]$,
there exists a point $z\left( \varepsilon \right) $ and a time $%
t(\varepsilon )\in \left( 0,T/\varepsilon \right) $, such that
\begin{equation}
I\left( \Phi _{t\left( \varepsilon \right) }^{\varepsilon }\left( z(\eps%
)\right) \right) -I\left( z(\eps)\right) >C_{h_{0}}.
\label{eq:diffusion-3bp}
\end{equation}%
The Lebesgue measure of the set of points $z$ satisfying %
\eqref{eq:diffusion-3bp} has a lower bound given by Theorem~\ref%
{thm:diffusion}, where the respective constants are defined in Theorem~\ref%
{th:3bp-C-bounds} and (\ref{eq:Su-choice}).

\item[(2)] (Symbolic dynamics)\label{pt:main-symbolic} There exists $\eta >0$
such that for every $\varepsilon \in (0,\eps_{0}]$ and every sequence $%
\left\{ I^{n}\right\} _{n\in \mathbb{N}}$, $I^{n}\in \left[ 2\eta
\varepsilon ,C_{h_{0}}-2\eta \varepsilon \right] $ with $\left\vert
I^{n+1}-I^{n}\right\vert >2\eta \varepsilon $ there exists a point $z$ and
an increasing sequence of times $t^{n}>0$, for $n\in \mathbb{N}$, such that
\begin{equation*}
\left\vert I\left( \Phi _{t^{n}}^{\varepsilon }\left( z\right) \right)
-I^{n}\right\vert <\eta \varepsilon \qquad \text{for all }n\in \mathbb{N}.
\end{equation*}

\item[(3)] (Hausdorff dimension)\label{pt:main-Hausdorff} Given $\left\{
I^{n}\right\} _{n\in \mathbb{N}}$ as in (2), the Hausdorff dimension of the
set
\begin{equation*}
\{z : \exists (t^{n})_{n\in\mathbb{N}}\text{ positive and increasing, s.t. }%
\forall n\in \mathbb{N},\, \left\vert I\left( \Phi _{t^{n}}^{\varepsilon
}\left( z\right) \right) -I^{n}\right\vert <\eta \varepsilon\}
\end{equation*}
is strictly greater than $4$ (in the $5$ dimensional extended phase space).

\item[(4)] (Stochastic behavior)\label{pt:main-stochastic} Let $\gamma >%
\frac{3}{2}$. For each ${X_{0}}\in \left( 0,C_{h_{0}}\right) $, $\mu \in
\mathbb{R}$, $\sigma >0$, consider the stochastic processes
\begin{equation*}
X_{t}^{0}:=X_{0}+\mu t+\sigma W_{t},\qquad \text{for }t\in \left[ 0,1\right]
.
\end{equation*}%
For $\varepsilon >0$ and a given point $z$ define the energy path \footnote{%
Alternatively we could define the energy paths as $X_{t}^{\varepsilon
}(z):=I(\Phi _{\gamma \varepsilon ^{-3/2}t}^{\varepsilon }(z))$, by taking
sufficiently large $\gamma >0$. The explicit size of such $\gamma $ and the
related details are outlined in the footnotes in Section \ref{sec:stochastic}%
.}
\begin{equation*}
X_{t}^{\varepsilon }(z):=I\left( \Phi _{t\varepsilon ^{-\gamma
}}^{\varepsilon }\left( z\right) \right) ,\qquad \text{for }t\in \left[ 0,1%
\right] .
\end{equation*}%
Define the stopping time
\begin{equation*}
\tau =\tau (X^{\varepsilon }):=\inf \left\{ t:X_{t}^{\varepsilon }\geq
C_{h_{0}}\text{ or }X_{t}^{\varepsilon }\leq 0\right\} .
\end{equation*}%
Then for every $0<\varepsilon <\eps_{0}$ there exits a set $\Omega
_{\varepsilon }$ of positive Lebesgue measure, such that the process $X_{t}^{%
\eps}:\Omega _{\eps}\rightarrow \mathbb{R}$, for $t\in \left[ 0,1\right] $,
has the following limit
\begin{equation*}
\lim_{\varepsilon \rightarrow 0}X_{t\wedge \tau }^{\varepsilon }=X_{t\wedge
\tau }^{0},
\end{equation*}
where $t\wedge \tau =\min (t,\tau )$. Above, $\Omega _{\varepsilon }$ is
endowed with with the sigma field of Borel sets and the normalized Lebsegue
measure (i.e., $\mathbb{P}_{\varepsilon }\left( \Omega _{\varepsilon
}\right) =1$), and the limit is in the sense of the functional central limit
theorem (see Section \ref{sec:stochastic}).
\end{enumerate}

The constants in the statements (1)-(4) above can be chosen as
\begin{equation*}
h_{0}=-1.5050906397016,\quad C_{h_{0}}=10^{-6},\quad T=\correctiona{\frac{1}{4}},\quad
\eta =10^{-2}.
\end{equation*}
\end{theorem}

\begin{remark}
In the statement of Theorem \ref{th:main-3bp}, $\eps_{0}=1.6\cdot 10^{-5}$
represents the true value of the eccentricity of the orbit of Triton. This
is among the lowest values of orbital eccentricities among the planets and
moons in our solar system. Triton is believed to had been captured by
Neptune about 1 billion years ago in a highly eccentric orbit, which decayed
in time to the currently observed value \cite{agnor2006neptune}.

%The PRE3BP is an idealisation 
\correctiona{The PRE3BP is a simplified model, }which neglects the effects of the other bodies in the solar system on the small particle. In the case of the Neptune-Triton system such effects are %strong. 
\correctiona{non-negligible. }
We therefore do not claim that our results are fully physical. They should be viewed as a mathematical model of some planet-moon or star-planet system, which is isolated from other effects. We have chosen the Neptune-Triton system due to its small eccentricity.
When eccentricity is small, it is hard to observe diffusion directly by integrating the system and measuring the change in energy along the trajectory. 
\end{remark}

\begin{remark}
The constant $h_{0}$ in the statement of Theorem \ref{th:main-3bp} is the
energy of some Lyapunov orbit, depicted in Figure \ref{fig:LyapOrb}.
This is merely a choice of convenience.

Statement (1) of Theorem \ref{th:main-3bp} shows the existence of orbits
that diffuse in energy for the whole range of eccentricities from zero up to
the current value of the eccentricity of the orbit of Triton. In particular,
there exist orbits whose energy changes by at least $C_{h_{0}}=10^{-6}$, in
a time at most $T/\eps_{0}=\correctiona{\frac{5}{32}\cdot10^{5}}$, in normalized units. In real units these
values are physically significant. Indeed, the energy drift $C_{h_{0}}$
corresponds to a distance of order $1$ km between two Lyapunov orbits whose
energy differs by $C_{h_{0}}$. Since the orbital period of Triton is $6$
days, which corresponds to $2\pi $ in our model, the time $T/\varepsilon _{0}
$ needed to achieve this change is under \correctiona{$42$ }years. Yet the Lebesgue
measure of the set of initial points of diffusing orbits which we establish
is exponentially small in $\varepsilon $.

Statement (2) of Theorem \ref{th:main-3bp} shows that any sequence of energy
levels within the range $(0,C_{h_{0}})$ can be $O(\eps)$-shadowed by true
orbits. There is a mild condition that these level sets should be chosen $O(%
\eps)$ apart from one another as well as from the endpoints of the energy
range.

Statement (4) of Theorem \ref{th:main-3bp} shows that for any prescribed
Brownian motion with drift, there exist sets of initial conditions $\Omega _{%
\eps}$ for which the corresponding energy paths approach asymptotically, as $%
\eps\rightarrow 0$, the given Brownian motion with drift. This fact is
consistent with Chirikov's empirical observations in \cite{Chirikov79} that
different sets of initial conditions yield different stochastic processes.
The obtained sets $\Omega _{\eps}$ have positive Lebesgue measure which goes
to $0$ as $\eps\rightarrow 0$. The broader question on characterizing the
limiting stochastic process for a fixed set of initial points of positive
measure is an open question.

We emphasize that the
methodology in this paper  can be used to obtain results similar to those in
Theorem~\ref{th:main-3bp} for other concrete three-body problems, with
higher values of the eccentricity.
\end{remark}

\begin{remark}
The diffusion time of order $O(1/\eps)$ in (1) is optimal for a priori chaotic
Hamiltonian systems, in the sense that  
the energy $H_{0}(z(t))$, hence the action $I$, cannot grow in time faster than
$O(\eps)$ (see, e.g., \cite{GelfreichT2008,GideaL13}). 
Indeed,  
if $z(t)$ is a solution of a Hamiltonian of the form 
\[H_\varepsilon(z,t)=H_0(z)+\varepsilon H_1(z,t; \varepsilon),\] then we have  
\begin{equation*}
\frac{d}{dt} H_{0}(z(t))= \varepsilon [H_0,H_1] (z(t),t;\varepsilon),
\end{equation*}
where $[\cdot,\cdot]$ denotes the Poisson bracket. 
In our setting, both  $H_0$ and $H_1$  and  their derivatives are bounded  since we restrict to $z$ in a compact neighborhood of some homoclinic orbit, $t\in\mathbb{T}^1$, and $\varepsilon\in[0,\varepsilon_0]$; see Section~\ref{sec:proof_main_theorem}\correctiona{, and in particular Figure \ref{fig:LyapOrb}}.
This implies that there is a constant $A>0$ such that
\[ |H_0(z(t)) - H_0(z(0)) |\leq \varepsilon At,\quad\textrm{for all $t\geq 0$}, \]
hence the energy  $H_0(z(t))$ cannot grow faster than linearly in time,
with  an $O(\varepsilon)$ growth rate.
\end{remark}

%TCIDATA{Version=5.00.0.2606}
%TCIDATA{LaTeXparent=0,0,MMFedit.tex}

\section{Tools}

%In this section we state some abstract topological results 
\correctiona{Our diffusion mechanism will be based on some abstract topological results} -- Theorems \ref{thm:diffusion}, \ref{th:symbolic-dynamics-from-covering}, \ref{th:Hausdorff-dim} and \ref{th:diffusion-process}, -- from which the proof of Theorem \ref{th:main-3bp} will follow.
These results are formulated in the general context of a non-autonomous,
parameter dependent $C^{r}$-smooth Hamiltonian $H_{\varepsilon }:\mathbb{R}%
^{n_{u}+n_{s}+2n_{c}}\times \mathbb{T\rightarrow R}$ of the form
\begin{equation}
H_{\varepsilon }\left( z,t\right) =H_{0}\left( z\right) +\varepsilon
H_{1}\left( z,t;\varepsilon \right) ,  \label{eqn:Hamiltonian_eps}
\end{equation}%
where $r\geq 3$, $n_{u}=n_{s}\geq 1$, $n_{c}\geq 1$, and $\varepsilon \in
\lbrack 0,\varepsilon _{0}]$. The associated ODE is
\begin{equation}
z^{\prime }(t)=J\nabla H_{\varepsilon }\left( z(t),t\right) ,
\label{eq:general-ode}
\end{equation}%
where $J$ is the standard almost complex structure on $\mathbb{R}%
^{n_{u}+n_{s}+2n_{c}}$. Although $n_{u}=n_{s}$, we will keep the two
notations separate for convenience. (In the sequel, $n_{u}$ will play the
role of the dimension of the `unstable' variables, $n_{s}$ of the `stable'
variables, and $\correctiona{2}n_{c}$ of the `center' ones).

We denote by $\Phi _{t}^{\varepsilon }\left( z,\theta \right) $ the flow
induced by (\ref{eq:general-ode}) in the extended phase space, where $\left(
z,\theta \right) \in \mathbb{R}^{n_{u}+n_{s}+2n_{c}}\times \mathbb{T}$, and $%
\theta^{\prime}(t)=1$. Note that for $\varepsilon =0$ the system %
\eqref{eqn:Hamiltonian_eps} is autonomous, so $H_{0}$ is preserved along the
solutions.

In the context of the PER3BP it is enough to consider $n_{u}=n_{s}=n_{c}=1$,
but we formulate the general results in the higher dimensional setting when $%
n_{u}=n_{s}\geq 1$ and $n_{c}=1$, since this can be done at no expense. The
results can also be generalized to $n_{c}\geq 1$. We refrain from doing so
here since the statements become too technical. We comment on such
generalization in Sections \ref{sec:diffusion}, \ref{sec:symbolic}, \ref%
{sec:Hausdorff} and \ref{sec:stochastic} where we prove the four general
results.

\subsection{Iterated function systems\label{sec:IFS}}

In our general results which lead to the proof of the Main Theorem \correctiona{\ref{th:main-3bp}}, we use a
system of maps as defined below, rather \correctiona{than }the flow.

Let
\begin{equation}
\{(\Sigma _{\ell ,0}\ldots ,\Sigma _{\ell ,k_{\ell }})\}_{\ell \in L}
\label{eqn:section_ell}
\end{equation}
be a finite collection of sections in the extended phase space $\mathbb{R}%
^{n_{u}+n_{s}+2}\times \mathbb{T}$, where $L$ is a finite set, and $%
k_{\ell}\geq 0$ for each $\ell \in L$. We assume that each section $\Sigma
_{\ell,i}$ is (locally) transverse to the flow $\Phi _{t}^{\varepsilon }$,
and is
%locally given by
\correctiona{homeomorphic to $\mathbb{R}^{n_{u}}\times \mathbb{R}^{n_{s}}\times \mathbb{R}\times \mathbb{T}$,} \marginpar{\color{red} \hfill 14.}
\begin{equation}
\Sigma _{\ell ,i}\simeq \mathbb{R}^{n_{u}}\times \mathbb{R}^{n_{s}}\times
\mathbb{R}\times \mathbb{T},  \label{eqn:sections}
\end{equation}
with local coordinates
\begin{equation}
v=(x,y,I,\theta ),\, x\in \mathbb{R}^{n_{u}},\, y\in \mathbb{R}^{n_{s}},\,
(I,\theta )\in \mathbb{R}\times \mathbb{T}^{1}.  \label{eqn:v_coordinates}
\end{equation}

The coordinate $I$ on each section $\Sigma _{\ell ,i}$ is defined by $%
I=H_{0}-h_{0}$, where $H_{0}$ is the energy of the unperturbed Hamiltonian
when $\eps=0$ in \eqref{eqn:Hamiltonian_eps}, and $h_{0}$ is some initial
level of the energy $H_{0}$.

Further, we assume that
\begin{equation*}
\Sigma _{\ell ,0}=\Sigma _{\ell ,k_{\ell }}=\Sigma _{0},
\end{equation*}%
where $\Sigma _{0}$ is a fixed section for all $\ell \in L$. The sections
are assumed to be independent of $\varepsilon $.

For a given $\ell \in L$ and $i\in \{1,\ldots ,k_{\ell }\}$ we define $\tau
_{\ell ,i}:\Sigma _{\ell ,i-1}\rightarrow \mathbb{R}$ by
\begin{equation*}
\tau _{\ell ,i}(z,\theta ):=\inf \left\{ t>0:(z,\theta )\in \Sigma _{\ell
,i-1}\text{ and }\Phi _{t}^{\varepsilon }\left( z,\theta \right) \in \Sigma
_{\ell ,i}\right\} .
\end{equation*}

Then we define the family of maps $f_{\ell ,i,\varepsilon }:\Sigma _{\ell
,i-1}\rightarrow \Sigma _{\ell ,i}$, for $i\in \{1,\ldots ,k_{\ell }\}$, to
be the section-to-section mappings along the flow, i.e., $f_{\ell
,i,\varepsilon }=\Phi _{\tau _{\ell ,i}}^{\varepsilon }$.

Note that $\tau_{\ell ,i}$ and $f_{\ell,i,\varepsilon }$ are only locally
defined.

We use the section-to-section mappings to define return maps $%
F_{\ell,\varepsilon}:\Sigma_0\to\Sigma_0$ associated to each family of
sections associated to $\ell\in L$, i.e.,
\begin{equation}  \label{eq:Fl-eps-map}
F_{\ell,\varepsilon}=f_{\ell,k_\ell,\varepsilon}\circ\ldots \circ
f_{\ell,1,\varepsilon}.
\end{equation}

The dynamics of primary interest is that of the iterated function system
(IFS)
\begin{equation}
\mathscr{F}_{\varepsilon }=\{F_{\ell ,\varepsilon }\}_{\ell \in L},
\label{eq:IFS}
\end{equation}%
that depends on $\varepsilon \in \lbrack 0,\varepsilon _{0}]$. For a fixed $%
\varepsilon $, an orbit of a point $z_{0}$ under the IFS is given by
\begin{equation}  \label{eq:IFS-orbit}
z_{n}=F_{\ell _{n},\varepsilon }\circ \ldots \circ F_{\ell _{1},\varepsilon
}(z_{0}),
\end{equation}
for some choice of $\ell _{1},\ldots ,\ell _{n}\in L$. Note that the same
point $z_{0}$ can yield different orbits depending on the choice of
successive maps that are applied.

We express the maps $f_{\ell ,i,\varepsilon }$ as well as $F_{\ell
,\varepsilon }$ in the local coordinates of the sections, i.e.,
\begin{equation*}
\begin{split}
f_{\ell ,i,\varepsilon }& :\mathbb{R}^{n_{u}}\times \mathbb{R}^{n_{s}}\times
\mathbb{R}\times \mathbb{T\rightarrow R}^{n_{u}}\times \mathbb{R}%
^{n_{s}}\times \mathbb{R}\times \mathbb{T},\qquad \text{ for }i\in
\{1,\ldots ,k_{\ell }\}, \\
F_{\ell ,\varepsilon }& :\mathbb{R}^{n_{u}}\times \mathbb{R}^{n_{s}}\times
\mathbb{R}\times \mathbb{T\rightarrow R}^{n_{u}}\times \mathbb{R}%
^{n_{s}}\times \mathbb{R}\times \mathbb{T},\qquad \text{ for }\ell \in L.
\end{split}%
\end{equation*}

Since $I$ is defined by $I=H_0-h_0$, it is a first integral of the
Hamiltonian vector field of $H_0$. A change in $I$ is equivalent to a
change in the energy $H_{0}$. We will refer to $I$ as the `action variable'
and to the $\theta$ as the `angle' variable. We point out that we do not
require the coordinate systems $(x,y,I,\theta )$ to be symplectic though.

For the unperturbed system, with $\varepsilon =0$, the ODE which drives our
IFS is autonomous, and hence the energy $H_0$, and implicitly the action $I$%
, are both preserved by the maps, i.e.,
\begin{equation*}
\pi _{I}f_{\ell,i,\varepsilon =0}\left( z\right) =\pi _{I}z.
\end{equation*}

\begin{remark}
{\ \phantom{.} }

\begin{itemize}
\item[1.] In the case of the PER3BP, in the definition of the action $%
I=H_{0}-h_{0}$, the energy $h_{0}$ is chosen as the energy level of some
Lyapunov orbit.

\item[2.] In the case of the PER3BP, for each index $\ell \in L$, we
construct a sequence of sections $(\Sigma _{\ell ,0}\ldots ,\Sigma _{\ell
,k_{\ell }})$ so that some of the sections are positioned along a homoclinic
orbit to $\Lambda _{0}$, and some other sections are positioned along the
NHIM $\Lambda _{0}$ itself. In our constructions we will always use the same
number of sections along the homoclinic orbit, but we will use a varying
number of sections along the NHIM. 
%(For example, during a homoclinic excursion we will cross a certain number of sections around the normally hyperbolic manifold in order to increase the energy, and a different number of sections around the normally hyperbolic manifold in order to decrease the energy.) 
(For example, after a homoclinic excursion we can cross different numbers of sections around the normally hyperbolic invariant manifold in order to return to the manifold at different angles.)
Different sequences of sections, corresponding to different indices
$\ell$, may share some common sections. Overall there is a finite number of
sections, therefore a finite number of sequences of sections that we use.

The index $\ell$ plays the role of a selector for which particular sequence
of sections we use to achieve a desired effect on the dynamics.

More generally, if we use more than one homoclinic orbit, we can associate
different sets of indices $\ell$ to sequences of sections associated to
different homoclinic orbits.

For the general results, we consider the case of using more than one
homoclinic orbit, therefore we allow multiple sets of sections,
corresponding to different homoclinics.

%The index $\ell$ plays the role of a selector of whether we follow one of the homoclinic orbits or the dynamics near the normally hyperbolic invariant manifold. Thus, the maps $f_{\ell,i,\varepsilon}$ are the section-to-section maps along one homoclinic orbit, or near $\Lambda_0$. The maps $% F_{\ell,\varepsilon}$ are the return maps to $\Sigma_0$ as we follow one homoclinic orbit, or we move near $\Lambda_0$.

\item[3.] The dynamics of interest for establishing diffusion is that of the
maps $F_{\ell,\varepsilon}$ in the IFS. The section-to-section maps $f_{\ell
,i,\varepsilon }$ play a technical role in the computer assisted proof. They
allow for shorter integration times between respective sections. This helps
in the setting of strong expansion along the unstable coordinate and
improves the accuracy of the estimates.
%In particular, they allow us to control better the topological expansion and contraction in the construction of correctly aligned windows.
\end{itemize}
\end{remark}

\subsection{Cone conditions}

Let $\mathbb{R}^{n}=\mathbb{R}^{n_{1}}\times \mathbb{R}^{n_{2}}$, and $Q:%
\mathbb{R}^{n_{1}}\times \mathbb{R}^{n_{2}}\rightarrow \mathbb{R}$ be a
\correctiona{function }given by
\begin{equation}
Q\left( z_1,z_2\right) =\left\Vert z_1\right\Vert _{n_1}^{2}-\left\Vert
z_2\right\Vert _{n_2}^{2},  \label{eq:cone-Q-0}
\end{equation}
where $\left\Vert \cdot \right\Vert _{n_i}$, are some norms on $\mathbb{R}%
^{n_{i}}$, for $i=1,2$. For a point $z\in \mathbb{R}^{n}$ we define the $Q$%
-cone at $z$ as the set $\left\{ z^{\prime}\in\mathbb{R}^{n} :Q\left(
z-z^{\prime}\right) >0\right\}$.

\begin{definition}
Let $Q_{1},Q_{2}:\mathbb{R}^{n_{1}}\times \mathbb{R}^{n_{2}}\rightarrow
\mathbb{R}$ be cones of the form (\ref{eq:cone-Q-0}). We say that a
continuous map $f:\mathbb{R}^{n}\rightarrow \mathbb{R}^{n}$ satisfies a $%
(Q_{1},Q_{2})$-cone condition if
\begin{equation}
Q_{1}(z-z^{\prime })>0\qquad \text{implies}\qquad Q_{2}(f(z)-f(z^{\prime
}))>0  \label{eq:cc}
\end{equation}%
for all $z,z^{\prime }$.

When $Q_{1}=Q_{2}=Q$, then we will simply say that $f$ satisfies a $Q$-cone
condition.
\end{definition}

Condition (\ref{eq:cc}) means that if $z^{\prime }$ is inside the $Q_{1}$%
-cone based at $z$, then $f(z^{\prime })$ is inside the $Q_{2}$-cone based
at $f(z)$.%
%\marginpar{changed order to $(x,y,I,\theta )$ to unify with other sections}

For the general results, in \eqref{eq:cone-Q-0} we will take $n_{1}=n_{u}$, $%
n_{2}=n_{s}+2$, $z=(z_{1},z_{2})\in \mathbb{R}^{n_{u}+n_{s}+2}$ with $%
z_{1}=x\in \mathbb{R}^{n_{u}}$, $z_{2}=(y,I,\theta )\in \mathbb{R}^{n_{s}+2}$%
, and the norms
\begin{equation*}
\Vert z_{1}\Vert _{n_{1}}=\Vert x\Vert \text{ and }\Vert z_{2}\Vert
_{n_{2}}=\max \left\{ \frac{1}{a_{y}}\left\Vert y\right\Vert ,\frac{1}{%
\varepsilon a_{I}}\left\Vert I\right\Vert ,\frac{1}{a_{\theta }}\left\Vert
\theta \right\Vert \right\} ,
\end{equation*}%
where $\Vert \cdot \Vert $ denotes the Euclidean norm and $a_{y}$, $%
a_{\theta }$, $a_{I}>0$ are constants independent of $\varepsilon $. Then
the mapping defining the corresponding cone is given by
\begin{equation}
Q^{\varepsilon }\left( x,y,I,\theta \right) =\left\Vert x\right\Vert
^{2}-\left( \max \left\{ \frac{1}{a_{y}}\left\Vert y\right\Vert ,\frac{1}{%
\varepsilon a_{I}}\left\Vert I\right\Vert ,\frac{1}{a_{\theta }}\left\Vert
\theta \right\Vert \right\} \right) ^{2}.  \label{eq:cone-Q}
\end{equation}%
Note that $Q^{\varepsilon }$ represents a family of functions, parameterized
by $\varepsilon >0$. If we also want to emphasize the dependence on $%
a=\left( a_{y},a_{I},a_{\theta }\right) \in \mathbb{R}_{+}^{3}$, we write $%
Q_{a}^{\varepsilon }$ instead of $Q^{\varepsilon }$.

We shall keep in mind that $Q^{\varepsilon }\left( x,y,I,\theta \right) >0$
implies%
\begin{equation}
a_{y}\left\Vert x\right\Vert \geq \left\Vert y\right\Vert ,\qquad a_{\theta
}\left\Vert x\right\Vert \geq \left\Vert \theta \right\Vert ,\qquad
\varepsilon a_{I}\left\Vert x\right\Vert \geq \left\Vert I\right\Vert .
\label{eq:cone-slopes-assumption}
\end{equation}

\subsection{Correctly aligned windows}

We shall write $B^{n}\left( z,r\right) $ to denote a ball in $\mathbb{R}^{n}$
of radius $r$ centered at $z$, $\bar{B}^{n}\left( z,r\right) $ for its
closure, $\partial B^{n}\left( z,r\right) $ for its boundary, and $B^{n}$
for a unit ball in $\mathbb{R}^{n}$ centered at $0$. Here the balls are
considered under some norms on $\mathbb{R}^{n}$.

A window in $\mathbb{R}^{n}=\mathbb{R}^{n_{1}}\times \mathbb{R}^{n_{2}}$ is
a set of the form $N=\bar{B}^{n_{1}}\left( z_{1},r_{1}\right) \times \bar{B}%
^{n_{2}}\left( z_{2},r_{2}\right) \subseteq \mathbb{R}^{n}$ with a choice of
an `exit set' and `entry set' respectively given by
\begin{align*}
N^{-}& :=\partial \bar{B}^{n_{1}}\left( z_{1},r_{1}\right) \times \bar{B}%
^{n_{2}}\left( z_{2},r_{2}\right) , \\
N^{+}& :=\bar{B}^{n_{1}}\left( z_{1},r_{1}\right) \times \partial \bar{B}%
^{n_{2}}\left( z_{2},r_{2}\right) .
\end{align*}

\begin{definition}
\label{def:covering}(\cite[Definition 6]{GideaZ04a}) Assume that $N$ and $M$
are windows in $\mathbb{R}^n$, and let $f:N\rightarrow \mathbb{R}^{n}$ be a
continuous mapping.

We say that $N$ is correctly aligned with $M$, and write
\begin{equation*}
N\overset{f}{\Longrightarrow }M
\end{equation*}
if the following conditions are satisfied:

\begin{enumerate}
\item There exists a continuous homotopy $\chi:[0,1]\times N\rightarrow
\mathbb{R}^{n_{1}}\times \mathbb{R}^{n_{2}}$, such that the following
conditions hold true%
\begin{align*}
\chi_{0}& =f, \\
\chi\left( \left[ 0,1\right] ,N^{-}\right) \cap M& =\emptyset , \\
\chi\left( \left[ 0,1\right] ,N\right) \cap M^{+}& =\emptyset .
\end{align*}

\item There exists a linear map $A:\mathbb{R}^{n_{1}}\rightarrow \mathbb{R}%
^{n_{1}}$ such that
\begin{align*}
\chi_{1}\left( z_1,z_2\right) & =\left( A z_1,0\right) ,\qquad \text{for all
}(z_1,z_2)\in N \subset \mathbb{R}^{n_{1}}\times \mathbb{R}^{n_{2}}, \\
A\left( \partial B^{n_{1}}\right) & \subset \mathbb{R}^{n_{1}}\setminus \bar{%
B}^{n_{1}}.
\end{align*}
\end{enumerate}
\end{definition}

Intuitively, Definition \ref{def:covering} states that $f(N)$ is
topologically aligned with $M$ as in Figure \ref{fig:covering}. The
coordinate $z_1$ corresponds to the `topologically unstable' directions, and
$z_2$ to the `topologically stable' directions.

The terminology of a `window' and `correctly aligned windows' was introduced
by Easton in \cite{easton1981orbit}. An alternative terminology which refers
to `windows' as `h-sets', and to `correct alignment' as `covering relation',
has been introduced in \cite{GideaZ04a}, where the method of \cite%
{easton1981orbit} has been generalized.

\begin{definition}
If $N$ is correctly aligned with $M$, and the function $f$ also satisfies a $%
(Q_{1},Q_{2})$-cone condition, then we say that we have a correct alignment
with cone conditions, and denote this by
\begin{equation*}
\left( N,Q_{1}\right) \overset{f}{\Longrightarrow }(M,Q_{2}).
\end{equation*}
\end{definition}

\correctiona{
\begin{remark} Correct alignment of windows and cone conditions are robust. When they hold for $f$, then they will also hold for functions that are sufficiently $C^1$-close to $f$. \end{remark}
}\marginpar{\color{red} 16.}

For the general results we will consider windows of the form:
\begin{equation*}
N=\bar{B}^{n_{u}}\times \bar{B}^{n_{s}}\times \lbrack I^{1},I^{2}]\times
[\theta ^{1},\theta ^{2}]\subset \mathbb{R}^{n_{u}}\times \mathbb{R}%
^{n_{s}}\times \mathbb{R\times T}^{1},
\end{equation*}%
where $I^{1},I^{2}\in \mathbb{R} ,$ $I^{1}<I^{2}$, $\theta ^{1},\theta
^{2}\in \lbrack 0,2\pi ),$ $\theta ^{1}<\theta ^{2}$.

In terms of coordinates, $z=(z_{1},z_{2})\in N$ if $z_{1}=x\in \bar{B}%
^{n_{u}}$, $z_{2}=(y,I,\theta )$, where $y\in \bar{B}^{n_{s}}$, $I\in
\lbrack I^{1},I^{2}]$, and $\theta \in \lbrack \theta ^{1},\theta ^{2}].$
The coordinate $x$ is the topologically unstable coordinate, and $y,I,\theta
$ the topologically stable coordinates.

The exit set of $N$ is defined by
\begin{equation*}
N^{-}=\partial (\bar{B}^{n_u})\times \bar{B}^{n_s}\times \lbrack
I^{1},I^{2}]\times \lbrack \theta ^{1},\theta ^{2}],
\end{equation*}%
and the entry set by
\begin{equation*}
N^{+}=\bar{B}^{n_u}\times \partial \left( \bar{B}^{n_s}\times \lbrack
I^{1},I^{2}]\times \lbrack \theta ^{1},\theta ^{2}]\right) .
\end{equation*}

\begin{figure}[t]
\begin{center}
\includegraphics[width=8.5cm]{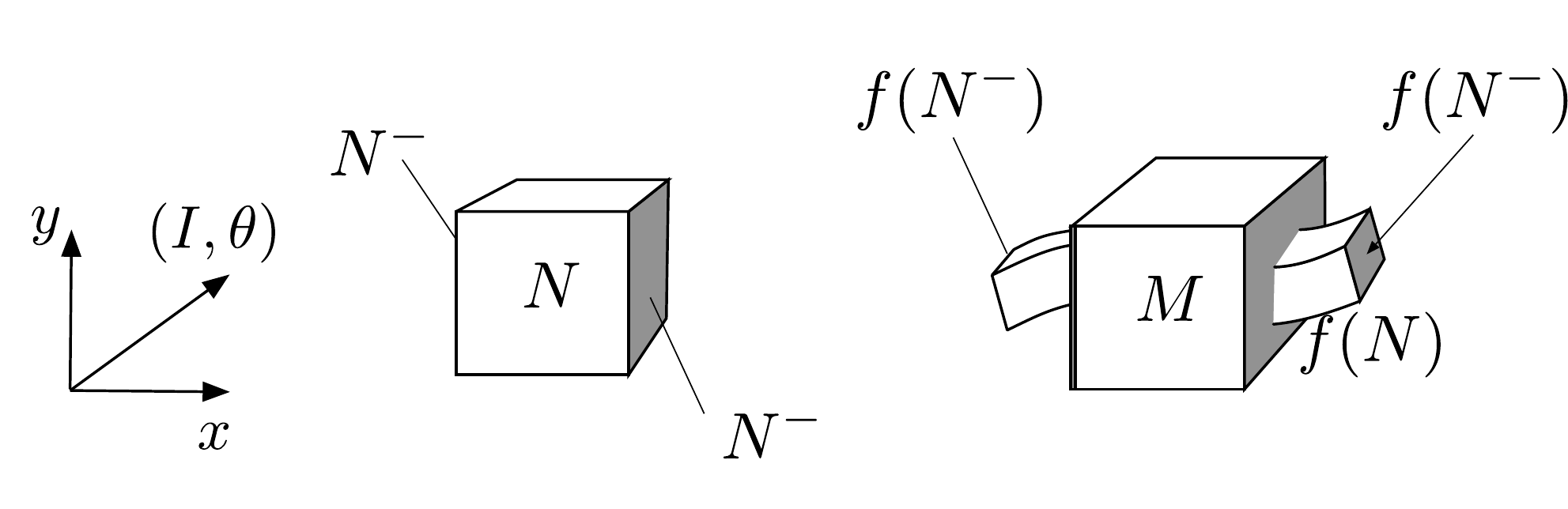}
\end{center}
\caption{The correct alignment of windows $N\protect\overset{f}{%
\Longrightarrow }M$.}
\label{fig:covering}
\end{figure}

\begin{remark}
In the context of the PER3BP, we will construct windows contained in the
Poincar\'e sections $\Sigma_{\ell,i}$, which are correctly aligned under the
section-to-section mappings $f_{\ell,i,\varepsilon}:
\Sigma_{\ell,i-1}\to\Sigma_{\ell,i}$. The sections and the
section-to-section mappings are as in Section~\ref{sec:IFS}. Via the local
coordinates of the sections, the windows are given as subsets of $\mathbb{R}%
^{n_u+n_s+2}$ and the section-to-section mappings as maps $%
f_{\ell,i,\varepsilon}: \mathbb{R}^{n_u+n_s+2}\to \mathbb{R}^{n_u+n_s+2}$.
Relative to these local coordinates, we will verify that the coordinate $x$
is expanded and the coordinate $y$ is contracted by the maps $%
f_{\ell,i,\varepsilon}$. The coordinates $I,\theta$ play the role of centre
coordinates, as they remain `neutral' under the maps $f_{\ell,i,\varepsilon}$%
, which means that the expansion and contraction rates in $I$, $\theta$ are
dominated by those in $x$, $y$. In the context of the windowing method, we
treat the $x$-coordinate as topologically unstable, and the $(y,I,\theta)$
as topologically stable. In order to achieve topological stability in $%
(I,\theta)$, at each step of the correct alignment we choose the successive
windows $N$ and $M$ such that the $(I,\theta)$-component of $M $ contains
inside it the projection of the image of $N$ onto the $(I,\theta)$%
-coordinates.
\end{remark}

\subsection{Connecting sequences\label{sec:connecting-sequences}}

We work in the framework of the IFS introduced in Section \ref{sec:IFS}.

Consider two sets in $\Sigma_0$, which we shall refer to as energy strips:
\begin{equation}
\mathbf{S}^{u}=\bar{B}^{n_u}\times \bar{B}^{n_s}\times \mathbb{R}\times
S_{\theta }^{u}\text{ \qquad and\qquad\ }\mathbf{S}^{d}=\bar{B}^{n_u}\times
\bar{B}^{n_s}\times \mathbb{R}\times S_{\theta }^{d},
\label{eq:strips-Sa-Sb}
\end{equation}%
where
\begin{equation*}
S_{\theta }^{u}=[\theta^{u}_{1},\theta^{u}_{2}]\text{ and }S_{\theta
}^{d}=[\theta^{d}_{1},\theta^{d}_{2}].
\end{equation*}

The indices $u,d$ are meant to suggest that when an orbit starts from $%
\mathbf{S}^{u}$ and returns to $\mathbf{S}^{u}$, the action variable $I$
goes \underline{u}p, and when it starts from $\mathbf{S}^{d}$ and returns to
$\mathbf{S}^{d}$, the action variable $I$ goes \underline{d}own.

%We will assume that we also have a family of cones $Q^{\varepsilon }:\mathbb{% R}^{n_u}\times \mathbb{R}^{n_s}\times \mathbb{R}\times \mathbb{R}\rightarrow \mathbb{R}$, parameterised by $\varepsilon $, satisfying (\ref{eq:cone-slopes-assumption})%

\begin{definition}
(Connecting sequence)\label{def:connecting-seq}

\begin{itemize}
\item[1.] A \emph{connecting sequence} consists of a sequence of windows
\begin{equation*}
N_{\ell ,0},N_{\ell ,1},\ldots ,N_{\ell ,k_{\ell }-1},N_{\ell ,k_{\ell }},
\end{equation*}%
a sequence of cones
\begin{equation*}
Q_{\ell ,0}^{\varepsilon },Q_{\ell ,1}^{\varepsilon },\ldots ,Q_{\ell
,k_{\ell }-1}^{\varepsilon },Q_{\ell ,k_{\ell }}^{\varepsilon },
\end{equation*}%
and a sequence of maps
\begin{equation*}
f_{\ell ,i,\varepsilon },\quad i=1,\ldots ,k_{\ell },
\end{equation*}%
for some $\ell \in L$, such that the following correct alignments with cone
conditions hold:
\begin{equation}
\left( N_{\ell ,0},Q_{\ell ,0}^{\varepsilon }\right) \overset{f_{\ell
,1,\varepsilon }}{\Longrightarrow }\left( N_{\ell ,1},Q_{\ell
,1}^{\varepsilon }\right) \overset{f_{\ell ,2,\varepsilon }}{\Longrightarrow
}\ldots \overset{f_{\ell ,k_{\ell },\varepsilon }}{\Longrightarrow }\left(
N_{\ell ,k_{\ell }},Q_{\ell ,k_{\ell }}^{\varepsilon }\right) .
\label{eq:connecting-sequence}
\end{equation}

Above, we assume that the cone $Q_{\ell ,0}^{\varepsilon }$ corresponding to
the initial window $N_{\ell ,0}$ and the cone $Q_{\ell ,k}^{\varepsilon }$
corresponding to the final window $N_{\ell ,k}$ are the same cone, i.e.,
\begin{equation}
Q_{\ell ,0}^{\varepsilon }=Q_{\ell ,k}^{\varepsilon }=Q^{\varepsilon }
\label{eq:cones-same}
\end{equation}%
with $Q^{\varepsilon }$ independent of $\ell \in L$.

%(The $Q_{1}^{\varepsilon},\ldots,Q_{k-1}^{\varepsilon}$ can depend on the choice of the sets $N_{1}^{\varepsilon},\ldots,N_{k-1}^{\varepsilon}$.)

\item[2.]
%\marginpar{{\small I added conditions }$\pi_{x,y}N_{0}=\pi_{x,y}S^{\kappa_{1}}${\small , }$\pi_{x,y}N_{k}=\pi_{x,y}S^{\kappa_{2}}$% {\small . They are implied from definitions of }$N_{i}${\small and }$% S^{\kappa_{i}}${\small but it is good to emphasize this.}}
Let $\kappa _{1},\kappa _{2}\in \{u,d\}$. A \emph{connecting sequence from }$%
\mathbf{S}^{\kappa _{1}}$\emph{\ to }$\mathbf{S}^{\kappa _{2}}$ is a
connecting sequence as above, such that
\begin{equation*}
N_{\ell ,0}\subseteq \mathbf{S}^{\kappa _{1}},\,N_{\ell ,k_{\ell }}\subseteq
\mathbf{S}^{\kappa _{2}},\,\pi _{x,y}N_{\ell ,0}=\pi _{x,y}\mathbf{S}%
^{\kappa _{1}},\,\text{ and }\pi _{x,y}N_{\ell ,k_{\ell }}=\pi _{x,y}\mathbf{%
S}^{\kappa _{2}}.
\end{equation*}
\end{itemize}

To simplify notation we refer to a connecting sequence (\ref%
{eq:connecting-sequence}) by writing out the sequence of sets $%
(N_{\ell,0},\ldots,N_{\ell,k_\ell})$.
\end{definition}

%\begin{remark}
%In the above, note that windows $N_{\ell ,i}$ are independent of $%
%\varepsilon $, while the maps $f_{\ell ,i,\varepsilon }$ do depend on $%
%\varepsilon $.
%\end{remark}

\begin{definition}
We will say that the orbit of a point $z\in N_{\ell,0}$ \emph{passes through
a connecting sequence} $(N_{\ell,0},\ldots,N_{\ell,k_\ell})$, for $\ell\in L$%
, if
\begin{align*}
f_{\ell,i,\varepsilon}\circ\ldots\circ f_{\ell,1,\varepsilon}\left( z\right)
& \in N_{\ell,i} ,\qquad\text{for }i=1,\ldots,k_\ell.
\end{align*}
\end{definition}

For each connecting sequence as in (\ref{eq:connecting-sequence}), we have a
map $F_{\ell ,\varepsilon }$ in the IFS defined as in \eqref{eq:Fl-eps-map}
Each $F_{\ell,\varepsilon }$ is associated with a connecting sequence from $%
\mathbf{S}^{\kappa _{1}}$ to $\mathbf{S}^{\kappa _{2}}$ for some $\kappa
_{1},\kappa _{2}\in \left\{ u,d\right\} $. Note that $F_{\ell,\varepsilon
}\left( z\right) $ is well defined for any $z$ which passes through the
connecting sequence. For a fixed $\varepsilon$, an orbit of the IFS starting
from a point $z_{0}$ can be expressed using the functions $%
F_{\ell,\varepsilon }$ as in \eqref{eq:IFS-orbit}.

%TCIDATA{Version=5.00.0.2606}
%TCIDATA{LaTeXparent=0,0,MMFedit.tex}

\section{Master theorems}

\subsection{Master theorem for establishing diffusion}

Consider an energy strip $\mathbf{S}^{u}$ as in (\ref{eq:strips-Sa-Sb}), and
assume that we have an IFS with a finite set $L$ of connecting sequences of
the form (\ref{eq:connecting-sequence}), satisfying (\ref{eq:cones-same}),
with $Q^{\varepsilon }$ independent of $\ell \in L$.
%{\color{red} We assume that there is a finite collection of connecting sequences  as in \eqref{eq:connecting-sequence}, with $\ell $ from a set of labels $L$.}

We will assume below that the following condition holds:

\medskip

\textbf{Condition \hypertarget{cond:C1}{C1}.}%\label{cond:C1}

\begin{itemize}
\item[(C1.i)] \hypertarget{cond:C1.i}{F}or each $\ell \in L$ and $%
\varepsilon \in (0,\varepsilon _{0}]$ there is a connecting sequence $%
(N_{\ell ,0},\ldots ,N_{\ell ,k_{\ell }})$ from $\mathbf{S}^{u}$ to $\mathbf{%
S}^{u}$;

\item[(C1.ii)] \hypertarget{cond:C1.ii}{T}he projection of the initial
windows $N_{\ell,0}$ onto the $(I,\theta)$-coordinates covers $\left[0,1%
\right] \times S_{\theta }^{u}$, i.e.,
\begin{equation}
\bigcup \limits_{\ell\in L}\pi_{I,\theta}\left( N_{\ell,0}\right) = \left[0,1%
\right] \times S_{\theta }^{u};  \label{eqn:C1_ii}
\end{equation}

\item[(C1.iii)] \hypertarget{cond:C1.iii}{W}henever $N_{\ell,0}\cap
N_{\ell^\prime,0}\neq \emptyset $, for every $\left( I^{\ast },\theta ^{\ast
}\right) \in \pi _{I,\theta }(N_{\ell,0}\cap N_{\ell^{\prime },0})$ the
multi-dimensional rectangle
\begin{equation*}
\bar{B}^{n_u}\times \bar{B}^{n_s}\times \left (\bar{B} \left( I^{\ast
},\varepsilon _{0}a_{I}\right) \cap \left[ 0,1\right]\right) \times \left(%
\bar{B}\left( \theta ^{\ast },a_{\theta }\right) \cap S_{\theta }^{u}\right)
\end{equation*}
is contained in $N_{\ell,0}$ or in $N_{\ell^{\prime },0}$, where $%
a_I,a_\theta$ are associated to $Q^\varepsilon=Q^\varepsilon_a$; 

\item[(C1.iv)] \hypertarget{cond:C1.iv}{T}here exists a fixed constant $c>0$
such that, for each each $z\in N_{\ell ,0}$ which passes through the
connecting sequence we have
\begin{equation}
\varepsilon \cdot c<\pi _{I}F_{\ell ,\varepsilon }(z)-\pi _{I}(z).
\label{eqn:C1_iv}
\end{equation}
\end{itemize}

\correctiona{\begin{remark}\label{rem:r-a-issue}
In the above, we require that the constant $c$ from condition (\ref{eqn:C1_iv}) is
independent from the choice of a connecting sequence. 

Also, in condition \correctiona{\hyperlink{cond:C1.iii}{(C1.iii)}}, 
when $\bar{B}^{n_u}$ and $\bar{B}^{n_s}$ are balls of radius $r$ instead of
unit balls, then we require that $\bar{B}^{n_u}\times \bar{B}^{n_s}\times
\left(\bar{B}\left( I^{\ast },\varepsilon _{0}a_{I}r\right) \cap \left[ 0,1%
\right]\right) \times \left(\bar{B}\left( \theta ^{\ast },a_{\theta
}r\right) \cap S^u_{\theta}\right) $ is contained in $N_{\ell,0}$ or $%
N_{\ell^{\prime },0}$.
\end{remark}}

\begin{theorem}[Existence of diffusing orbits]
\label{thm:diffusion} Assume that condition \hyperlink{cond:C1}{\textbf{C1}} holds. Then for each
$\varepsilon \in (0,\varepsilon _{0}]$, there exist $z\in \mathbf{S}^{u}$
and a sequence of functions $F_{\ell _{1},\varepsilon },\ldots ,$ $F_{\ell
_{m},\varepsilon }$ of the form (\ref{eq:Fl-eps-map}), such that $\tilde{z}%
=(F_{\ell _{m},\varepsilon }\circ \ldots \circ F_{\ell _{1},\varepsilon
})(z)\in \mathbf{S}^{u}$ satisfies
\begin{equation}
\left\Vert \pi _{I}\left( \tilde{z}\right) -\pi _{I}\left( z\right)
\right\Vert \geq 1,  \label{eqn:diffusionC}
\end{equation}%
for some
\begin{equation*}
m\leq \frac{1}{\varepsilon c}.
\end{equation*}%
In other words, there exists an orbit of the IFS with a change in $I$ of
order $O(1)$.

Moreover, if $\alpha \in \mathbb{R}$ is a number such that
\begin{equation}
\alpha >\sup_{\ell\in L,\varepsilon \in \left[ 0,\varepsilon _{0}\right] }
\left\{ \left\Vert DF_{\ell ,\varepsilon}(z)\right\Vert :z\in N_{\ell
,0}\right\} ,  \label{eq:expansion-bound}
\end{equation}
then for every $\varepsilon \in (0,\varepsilon _{0}]$ the set
\begin{equation*}
\Omega :=\left\{ z\,:\,\pi _{I}z\in \lbrack 0,1/3]\text{ and }\left\Vert \pi
_{I}(F_{\ell _{m},\varepsilon }\circ \ldots \circ F_{\ell _{1},\varepsilon
})(z)-\pi _{I}\left( z\right) \right\Vert \geq \frac{1}{3}\right\}
\end{equation*}%
has positive Lebesgue measure $\mu (\Omega )$ lower bounded by
\begin{equation*}
\mu (\Omega )\geq \alpha ^{-2n_{u}/\left( 3\varepsilon c\right) }\mu \left(
\mathbf{S}^{u}\cap \left\{ I\in \left[ 0,1/3\right] \right\} \right) .
\end{equation*}
\end{theorem}

The proof is given in section \ref{sec:diffusion}.

\begin{remark}
The second part of Theorem~\ref{thm:diffusion} provides an explicit lower
bound estimate for the Lebesgue measure of a set of points whose action
changes by $O(1)$. This lower bound is clearly positive but exponentially
small in $\varepsilon $.
\end{remark}

\subsection{Master theorem for establishing symbolic dynamics}

Consider two energy strips $\mathbf{S}^{u}$ and $\mathbf{S}^{d}$ as in (\ref%
{eq:strips-Sa-Sb}).

We assume that there is a finite collection of connecting sequences as in %
\eqref{eq:connecting-sequence},
%\begin{equation*}(N_{\ell ,0},Q_{\ell ,0}^{\varepsilon })\overset{f_{\ell ,1,\varepsilon }}{\Longrightarrow }(N_{\ell ,1},Q_{\ell ,1}^{\varepsilon })\overset{f_{\ell ,2,\varepsilon }}{\Longrightarrow }\ldots\overset{f_{\ell ,k_{\ell },\varepsilon }}{\Longrightarrow }(N_{\ell ,k_{\ell }},Q_{\ell ,k_{\ell}}^{\varepsilon }), \end{equation*}%
with $\ell $ from a set of labels $L=L^{uu}\cup L^{ud}\cup L^{du}\cup L^{dd}$%
, with $L^{\alpha \beta }$, $\alpha ,\beta \in \{u,d\}$, mutually disjoint.
We assume that the connecting sequences for $\ell \in L^{\alpha \beta }$ are
from $\mathbf{S}^{\alpha }$ to $\mathbf{S}^{\beta }$ for $\alpha ,\beta \in
\{u,d\}$, and satisfy (\ref{eq:cones-same}), with $Q^{\varepsilon }$
independent of $\ell \in L$.
% that each connecting sequence for $\ell \in L^{\alpha \beta }$ for $\alpha ,\beta \in \{u,d\},$ starts with a cone $Q^{\varepsilon }$ and finishes also with the same cone $Q^{\varepsilon }$,i.e. $Q_{\ell ,0}^{\varepsilon }=Q_{\ell ,k_{\ell }}^{\varepsilon}=Q^{\varepsilon }.$

We assume the following condition:

\medskip

\textbf{Condition \hypertarget{cond:C2}{C2}.}

\begin{itemize}
\item[(C2.i)] \hypertarget{cond:C2.i}{F}or each $\ell \in L^{\alpha \beta }$%
, $\alpha ,\beta \in \{u,d\},$ and each $\varepsilon \in (0,\varepsilon
_{0}] $, there is a connecting sequence $(N_{\ell ,0},\ldots ,N_{\ell
,k_{\ell }})$ from $\mathbf{S}^{\alpha }$ to $\mathbf{S}^{\beta }$;

\item[(C2.ii)] \hypertarget{cond:C2.ii}{W}e have
\begin{equation*}
\bigcup \limits_{\ell\in L^{\alpha\beta}}\pi_{I,\theta}\left(
N_{\ell,0}\right) = \left[0,1\right] \times S_{\theta }^{\alpha},\qquad\text{%
for }\alpha,\beta\in\left\{ u,d\right\} ;
\end{equation*}

\item[(C2.iii)] \hypertarget{cond:C2.iii}{W}henever $N_{\ell,0}\cap
N_{\ell^{\prime },0}\neq \emptyset $ for $\ell,\ell^{\prime }\in
L^{\alpha\beta},$ for every $\left( I^{\ast },\theta ^{\ast }\right) \in \pi
_{I,\theta }(N_{\ell,0}\cap N_{\ell^{\prime },0})$ the multi-dimensional
rectangle
\begin{equation*}
\bar{B}^{n_u}\times \bar{B}^{n_s}\times \left(\bar{B}\left( I^{\ast
},\varepsilon _{0}a_{I}\right) \cap \left[0,1\right] \right) \times \left(%
\bar{B}\left(\theta ^{\ast },a_{\theta }\right) \cap S_{\theta
}^{\alpha}\right)
\end{equation*}
is contained in $N_{\ell,0}$ or $N_{\ell^{\prime },0}$;
\item[(C2.iv)] \hypertarget{cond:C2.iv}{T}here exists a constant $C>0$, such
that, if the orbit of $z$ passes through a connecting sequence then
\begin{equation*}
\left\vert \pi _{I}F_{\ell ,\varepsilon }(z)-\pi _{I}(z)\right\vert
<\varepsilon \cdot C\qquad \text{for }\ell \in L^{uu},L^{ud},L^{du},L^{dd};
\end{equation*}

\item[(C2.v)] \hypertarget{cond:C2.v}{T}here exists a $c>0$ such that, if
the orbit of $z$ passes through a connecting sequence then
\begin{align*}
\varepsilon \cdot c& <\pi _{I}F_{\ell ,\varepsilon }(z)-\pi _{I}(z)\qquad
\text{if }\ell \in L^{uu}, \\
\varepsilon \cdot c& <\pi _{I}(z)-\pi _{I}F_{\ell ,\varepsilon }(z)\qquad
\text{if }\ell \in L^{dd}.
\end{align*}
\end{itemize}

\correctiona{\begin{remark}\label{rem:r-a-issue-2}
In the above, we require that the constant $c$ and $C$ are independent of the choice of a
connecting sequence.

Also, in condition \correctiona{\hyperlink{cond:C2.iii}{(C2.iii)}}, 
when $\bar{B}^{n_u}$ and $\bar{B}^{n_s}$ are balls of radius $r$ instead of
unit balls, then we require that $\bar{B}^{n_u}\times \bar{B}^{n_s}\times
\left(\bar{B}\left( I^{\ast },\varepsilon _{0}a_{I}r\right) \cap \left[ 0,1%
\right]\right) \times \left( \bar{B}\left( \theta ^{\ast },a_{\theta
}r\right) \cap S^{\alpha}_{\theta}\right) $ is contained in $N_{\ell,0}$ or $%
N_{\ell^{\prime },0}$.
\end{remark}
}

\begin{theorem}[Symbolic dynamics]
\label{th:symbolic-dynamics-from-covering} Let $\eta >0$ be a constant
satisfying $\eta \geq 2a_{I}+C$. Assume that
condition \hyperlink{cond:C2}{\textbf{C2}} holds. Then, for every $\varepsilon \in (0,\varepsilon
_{0}]$ and every infinite sequence of $I$-level sets $(I^{n})_{n\in \mathbb{N%
}}$ with $2\eta \varepsilon \leq I^{n}\leq 1-2\eta \varepsilon $ and $%
|I^{n+1}-I^{n}|>2\eta \varepsilon $, there exist an orbit $(z_{n})_{n\in
\mathbb{N}}$ of the IFS (\ref{eq:IFS-orbit}),
\begin{equation}  \label{eqn:z_n}
z_{n}=F_{\ell _{n},\varepsilon }\circ \ldots \circ F_{\ell _{1},\varepsilon
}\left( z_{0}\right) ,
\end{equation}%
where $\{\ell _{i}\}_{i\in \mathbb{N}}\subset L$, and an increasing sequence
$k_{n}$ such that, for every $n\in \mathbb{N}$ we have
\begin{equation}
|\pi _{I}z_{k_{n}}-I^{n}|<\eta \varepsilon .  \label{eq:epsilon-shadowing}
\end{equation}
\end{theorem}

The proof is given in section \ref{sec:symbolic}.
\correctiona{\begin{remark}\label{rem:eta-r} In Theorem \ref{th:symbolic-dynamics-from-covering}, when $\bar{B}^{n_{u}}$ and $\bar{B}^{n_{s}}$ are balls of radius $r$ instead
of unit balls, it is enough that $\eta \geq 2ra_{I}+C$ (instead of $\eta \geq 2a_{I}+C$).
\end{remark}}

\subsection{Master theorem for estimating the Hausdorff dimension of orbits
that undergo symbolic dynamics}

\medskip

\textbf{Condition \hypertarget{cond:C3}{C3}. }If
\begin{equation*}
\ell _{1},\ell _{2}\in L^{uu},\ell _{1}^{\prime }\in L^{ud},\ell
_{2}^{\prime }\in L^{du},\qquad \text{or}\qquad \ell _{1},\ell _{2}\in
L^{dd},\ell _{1}^{\prime }\in L^{du},\ell _{2}^{\prime }\in L^{ud},
\end{equation*}%
then the domain of $F_{\ell _{2},\varepsilon }\circ F_{\ell _{1},\varepsilon
}$ is disjoint from the domain of $F_{\ell _{2}^{\prime },\varepsilon }\circ
F_{\ell _{1}^{\prime },\varepsilon }$.\medskip

The theorem below gives us the lower bound on the Hausdorff dimension of a
set of orbits which follow the symbolic dynamics from Theorem~\ref%
{th:symbolic-dynamics-from-covering}.

\begin{theorem}[Hausdorff dimension]
\label{th:Hausdorff-dim} Assume that condition \hyperlink{cond:C2}{\textbf{C2}} holds, $\eta \geq
2a_{I}+C$, and $(I^{n})_{n\in \mathbb{N }}$ is such that $2\eta \varepsilon
\leq I^{n}\leq 1-2\eta \varepsilon $ and $|I^{n+1}-I^{n}|>2\eta \varepsilon $%
.

Then the set
\begin{equation}  \label{eqn:Hausdorff_set}
\{z_0 : \exists (z_{n})_{n\in\mathbb{N}} \text{ as in \eqref{eqn:z_n}},
\exists (k_{n})_{n\in\mathbb{N}} \text{ increasing s.t. } \forall n,\, |\pi
_{I}z_{k_{n}}-I^{n}|<\eta \varepsilon \}
\end{equation}
has Hausdorff dimension greater or equal to $n_{s}+2$.

If in addition $n_{u}=n_{s}=1$ and condition \hyperlink{cond:C3}{\textbf{C3}} holds, then the
Hausdorff dimension of the set (\ref{eqn:Hausdorff_set}) is strictly greater
than $n_{s}+2=3$.
\end{theorem}

%\begin{remark} Since in the case of the PRE3BP we have $n_{u}=n_{s}=1$, the bound on the Hausdorff dimension is strictly bigger than $n_{s}+2=3$. This follows from the mass distribution principle \cite{Falconer} and we comment on this in some more detail in section \ref{sec:Hausdorff}. \end{remark}

The proof is given in section \ref{sec:Hausdorff}. Note that different
points $z_{0}$ in the set \eqref{eqn:Hausdorff_set} correspond to different
orbits of the IFS that achieve the shadowing of the prescribed level sets of
the action.

%\begin{remark}
%\label{rem:epsilon-subdivision}Theorems~\ref{thm:diffusion}, \ref{th:symbolic-dynamics-from-covering} and \ref{th:Hausdorff-dim} are true for every $\varepsilon \in \lbrack \varepsilon _{1},\varepsilon _{2}]$, where $[\varepsilon _{1},\varepsilon _{2}]\subseteq (0,\varepsilon _{0}]$. This implies that, given an interval $(0,\varepsilon _{0}]$, if we are able to divide it into finitely many subintervals
%\begin{equation*}
%(0,\varepsilon _{0}]=(0,\varepsilon _{1}]\cup \lbrack \varepsilon_{1},\varepsilon _{2}]\cup \ldots \cup \lbrack \varepsilon_{k-1},\varepsilon _{k}]
%\end{equation*}%
%with $\varepsilon _{k}=\varepsilon _{0}$, and if the conditions \textbf{C1},  \textbf{C2},\textbf{\ C3 }are satisfied on each sub-interval, then the conclusions of Theorems~\ref{thm:diffusion}, \ref{th:symbolic-dynamics-from-covering} and \ref{th:Hausdorff-dim} hold true for $\varepsilon \in (0,\varepsilon _{0}]$, for some appropriate constants. In the computer assisted proof, the connecting sequences used for each subinterval $[\varepsilon _{k-1},\varepsilon _{k}]$ are independent of $\varepsilon $, but may different from one subinterval to another.
%\end{remark}

\correctiona{
\begin{remark}
\label{rem:epsilon-subdivision}
We can divide $(0,\varepsilon _{0}]$ into finitely many subintervals
\begin{equation*} (0,\varepsilon _{0}]=(0,\varepsilon _{1}]\cup \lbrack \varepsilon_{1},\varepsilon _{2}]\cup \ldots \cup \lbrack \varepsilon_{k-1},\varepsilon _{k}]
\end{equation*}%
with $\varepsilon _{k}=\varepsilon _{0}$, and validate \hyperlink{cond:C1}{\textbf{C1}}, \hyperlink{cond:C2}{\textbf{C2}}, \hyperlink{cond:C3}{\textbf{C3}} on each sub-interval. In the computer assisted proof the connecting sequences used for each subinterval $[\varepsilon _{k-1},\varepsilon _{k}]$ are independent of $\varepsilon $, but may be different from one subinterval to another.
\end{remark}}

\subsection{Master theorem for establishing stochastic behavior\label%
{sec:master-stochastic}}

We start by introducing some notation. We write $W_{t}$ for the standard
Brownian motion and assume that it is defined on a probability space $\left(
\Omega ,\mathbb{F},\mathbb{P}\right) .$ By $C[0,1]$ we denote the space of
continuous real functions on $[0,1].$ We endow $C[0,1]$ with the Borel $%
\sigma$-field denoted as $\mathcal{C}$, generated by open sets with topology
induced by the supremum norm.

Recall that for a stochastic process $X_{t}:\Omega \rightarrow \mathbb{R}$
and for $\omega \in \Omega $, a path is the function $t\mapsto X_{t}(\omega
) $, for $t\in \lbrack 0,1]$. All stochastic processes which we shall deal
with have continuous paths. We can therefore view a stochastic process $%
X_{t} $ as $X:\Omega \rightarrow C[0,1]$ by assigning to $\omega \in \Omega $
the corresponding path. See \cite{Billingsley}.

Let $\left( \Omega _{\varepsilon },\mathbb{F}_{\varepsilon },\mathbb{P}%
_{\varepsilon }\right) $ be a family of probability spaces parameterized by $%
\varepsilon \in (0,\varepsilon _{0}]$. We say that stochastic processes $%
X^{\varepsilon }:\Omega _{\varepsilon }\rightarrow C[0,1]$ converge to $%
X:\Omega \rightarrow C[0,1]$ in distribution on $C[0,1]$ as $\varepsilon
\rightarrow 0$, if for every $A\in \mathcal{C}$ with $\mathbb{P}(X\in
\partial A)=0$, we have $\mathbb{P}_{\varepsilon }(X^{\varepsilon }\in
A)\rightarrow \mathbb{P}(X\in A)$ as $\varepsilon \rightarrow 0$. (In our
setting all considered $X^{\varepsilon }:\Omega _{\varepsilon }\rightarrow
C[0,1]$ and $X:\Omega \rightarrow C[0,1]$ will be measurable.)

We now state our theorem concerning the convergence of energy paths to a
diffusion process.

\begin{theorem}[Stochastic behaviour]
\label{th:diffusion-process}Let $\gamma >\frac{3}{2}$. For every ${X_{0}}\in
\left( 0,1\right) $, $\mu \in \mathbb{R}$ and $\sigma >0$, consider the
stochastic processes
\begin{equation*}
X_{t}^{0}:=X_{0}+\mu t+\sigma W_{t},\qquad \text{for }t\in \left[ 0,1\right]
.
\end{equation*}%
%
%
%
%
%
%
%
%
%where $W_{t}$ is the standard Wiener process.
Let $I=H_{0}-h_{0}$. For $\varepsilon >0$ and a given point $z$ define the
energy path \footnote{%
Alternatively we could define the energy paths as $X_{t}^{\varepsilon
}(z):=I(\Phi _{\gamma \varepsilon ^{-3/2}t}^{\varepsilon }(z))$, by taking
sufficiently large $\gamma >0$. The explicit size of such $\gamma $ and the
related details are outlined in the footnotes in Section \ref{sec:stochastic}%
.}
\begin{equation*}
X_{t}^{\varepsilon }(z):=I\left( \Phi _{t\varepsilon ^{-\gamma
}}^{\varepsilon }\left( z\right) \right) ,\qquad \text{for }t\in \left[ 0,1%
\right] .
\end{equation*}%
Define the stopping times
\begin{equation*}
\tau =\tau (X^{\varepsilon }):=\inf \left\{ t:X_{t}^{\varepsilon }\geq 1%
\text{ or }X_{t}^{\varepsilon }\leq 0\right\} .
\end{equation*}%
If the condition \hyperlink{cond:C2}{\textbf{C2}} is satisfied, then for every $0<\varepsilon
\leq \varepsilon _{0}$ there exists a set $\Omega _{\varepsilon }\subset
\mathbb{R}^{n_{u}+n_{s}+2}\times \mathbb{T}$ of positive Lebesgue measure
such that

\begin{enumerate}
\item $\Omega _{\varepsilon }$ projects onto the $y,I,\theta $ coordinates
as $\pi _{y,I,\theta }\Omega_{\varepsilon }=\pi _{y,I,\theta } \left(
\mathbf{S}^{u}\cup \mathbf{S}^{d}\right)\cap\{I\in
[X_0-\varepsilon,X_0+\varepsilon]\}.$

\item Let $\Omega _{\varepsilon }$ be endowed with with the sigma field of
Borel sets and the normalized Lebesgue measure (i.e., $\mathbb{P}%
_{\varepsilon }\left( \Omega _{\varepsilon }\right) =1$). Then the family of
processes $X^{\eps}:\Omega _{\eps}\rightarrow C[0,1]$ has the following
limit in distribution on $C[0,1]$
\begin{equation}
\lim_{\varepsilon \rightarrow 0}X_{t\wedge \tau }^{\varepsilon }=X_{t\wedge
\tau }^{0},  \label{eq:limit-diffusion-process-master-thm}
\end{equation}%
where $t\wedge \tau =\min (t,\tau )$.
%and the convergence is in the Skorohod topology (see Section \ref{sec:stochastic}).
\end{enumerate}
\end{theorem}

The proof is given in Section \ref{sec:stochastic}.

%TCIDATA{Version=5.00.0.2606}
%TCIDATA{LaTeXparent=0,0,MMFedit.tex}

\section{Proof of the main theorem}\label{sec:proof_main_theorem}

\subsection{Computer assisted proofs\label{sec:cap-description}}

It is well known that numerical integration of differential equations
inherently carries numerical errors, due to computer rounding or to the
numerical methods involved. Therefore, the results obtained through
numerical experiments are in general non-rigorous. Computer assisted proofs
provide methods that use numerical experiments to produce rigorous results.
Instead of a numerical computation of an approximate solution, the computer
is utilized to return an enclosure, which is a set containing the true
solution.

One basic tool that we use in this paper is interval arithmetic. This
involves enclosing numbers in intervals that account for round-off errors,
and performing arithmetic operation on these intervals. The output of these
operations are intervals as well, which account for the numerical error and
contain the true result. Interval arithmetic methods have been extended to
operations with elementary functions, to solving linear systems, and to
computing high order derivatives of functions \cite{Rump}. Combined with the
Lohner algorithm \cite{Lohner,Lohner1,Lohner2} interval arithmetic can be
used to obtain rigorous enclosures of solutions of ordinary differential
equations, together with their partial derivatives with respect to initial
conditions, up to any given order. These methods can also be used to obtain
enclosures of images of Poincar\'e maps as well as of their partial
derivatives. All these have been implemented in the CAPD\footnote{%
Computer Assisted Proofs in Dynamics: http://capd.ii.uj.edu.pl/}
library \cite{WilczakCNSNS}, which is our tool of choice for our computer assisted proof\footnote{The code and the documentation for the computer assisted proof is available on the web page of the first author.}.

%To obtain a computer assisted proof, we use interval arithmetic methods as described above in combination with mathematical theorems. Such theorems formulate certain assumptions that can be verified numerically, and yield as conclusions the existence of true solutions with desired properties. This approach eliminates the problem of having to control rounding errors resulting from numerical methods or from the floating point computer representation of numbers. These are automatically taken care of by interval arithmetic, which uses directed rounding.\footnote{This methodology was also used to obtain a computer assisted proof of Smale's 14th Problem in \cite{tucker2002rigorous}}

\correctiona{To obtain a computer assisted proof, we use interval arithmetic methods to validate the assumptions of the mathematical theorems. This is done by a finite number of computations. The use of interval arithmetic eliminates the problem of having to control rounding errors resulting from numerical methods or from the floating point computer representation of numbers.\footnote{This methodology was also used to obtain a computer assisted proof of Smale's 14th Problem in \cite{tucker2002rigorous}}}\marginpar{\color{red}\hfill 19.}

Computer assisted proofs are usually conducted in a two-step procedure:

\begin{itemize}
\item first, one computes via standard (non-rigorous) numerical experiments
some geometric objects of interest, in order to obtain an intuitive
understanding of the properties of the system,

\item second, one performs rigorous, interval arithmetic-based estimates to
validate the assumptions of the appropriate theorems that establish the
properties observed in the non-rigorous numerical experiments.
\end{itemize}

%For the computer assisted proof in this paper we will need to verify rigorously the conditions \textbf{C1} and \textbf{C2}, which are needed for Theorems \ref{thm:diffusion}, \ref{th:symbolic-dynamics-from-covering}, \ref{th:Hausdorff-dim}, \ref{th:diffusion-process}. These conditions are formulated in terms of:

\correctiona{For the computer assisted proof in this paper we verify rigorously the conditions \hyperlink{cond:C1}{\textbf{C1}} and \hyperlink{cond:C2}{\textbf{C2}}, which are sufficient for Theorems \ref{thm:diffusion}, \ref{th:symbolic-dynamics-from-covering}, \ref{th:Hausdorff-dim}, \ref{th:diffusion-process}. These conditions are formulated in terms of:}%\marginpar{\color{red} \hfill 20.}

\begin{itemize}
\item correctly aligned windows,

\item cone conditions,

\item estimates on the change of energy along orbits.
%(conditions C1.iv, C3.iv--C2.v).
\end{itemize}

In the subsequent sections we discuss in detail how these are validated.
Here we just mention that correct alignment of windows can be validated by
checking inequalities between various projections of images of sets by
section-to-section maps. Our validation of cone conditions and of \hyperlink{cond:C1.iv}{(C1.iv)},
\hyperlink{cond:C2.iv}{(C2.iv}--\hyperlink{cond:C2.v}{C2.v)} is based on the enclosures of partial derivatives of
section-to-section maps, with respect to the initial conditions and with
respect to the parameter. Computation of enclosures of both the images and
partial derivatives of such maps with respect to the initial conditions are
a part of the CAPD library. Partial derivatives which include the parameter
can also be automatically obtained in CAPD by adding the parameter as one of
the variables of the system.

\subsection{Outline of the proof}

In Section \ref{sec:main-proof-theoretical} we give the theoretical tools
which we use for the validation of correct alignment of windows, for
validation of cone conditions, and for establishing estimates on the change
of energy along orbits.

In the next three subsections we describe how we choose the local maps for
the connecting sequences used in the proof of the main theorem. We choose
these maps to be section-to-section maps along the flow, expressed in
appropriate local coordinates. Section \ref{sec:main-proof-local-coord}
describes how we choose the local coordinates. In Section \ref%
{sec:main-proof-sections} we discuss how we choose the sections along the
flow. In Section \ref{sec:main-proof-connecting-sequences} we describe how
we choose our connecting sequences.

It is important to emphasize that our computer assisted proof is conducted
in the two-step procedure outlined in Section \ref{sec:cap-description}. The
choices described in Sections \ref{sec:main-proof-local-coord}, \ref%
{sec:main-proof-sections}, \ref{sec:main-proof-connecting-sequences} are
associated with the first step (non-rigorous numerics). The local maps for
connecting sequences are chosen based on the non-rigorous numerical
investigation of the properties of the system. Once these choices are made,
we conduct rigorous, interval arithmetic-based estimates to validate \hyperlink{cond:C1}{\textbf{%
C1}} and \hyperlink{cond:C2}{\textbf{C2}}, which leads to the proof of our main theorem. This is
done in Section \ref{sec:validation_connecting_sequences}.

\subsection{Tools for validating correct alignment of windows, cone
conditions and changes of energy along trajectories\label%
{sec:main-proof-theoretical}}

\label{sec:energy_change}\label{sec:validation_windows}\label%
{sec:validation_cones} We start by introducing some notation. For $E\subset
\mathbb{R}$ and $N\subset \mathbb{R}^{n}$ consider a family of $(k\times n)$%
-matrices $B(\eps,x)$, with $\eps\in E$ and $x\in N$. We will use the
following notation for the following subsets of $\mathbb{R}^{k\times n}$
\begin{eqnarray}
\left[ B\left( E,N\right) \right] &=&\left\{ B:B_{ji}\in \left[
\inf_{\varepsilon \in E,x\in N}B\left( \varepsilon ,x\right)
,\sup_{\varepsilon \in E,x\in N}B\left( \varepsilon ,x\right) \right]
\right\} ,  \label{eq:interval-derivative} \\
\left[ B\left( E,N\right) \right] &=&\left\{ B:B_{ji}\in \left[
\inf_{\varepsilon \in E,x\in N}B\left( \varepsilon ,x\right)
,\sup_{\varepsilon \in E,x\in N}B\left( \varepsilon ,x\right) \right]
\right\} .  \label{eq:interval-derivative-2}
\end{eqnarray}

We refer to a subset of $\mathbb{R}^{k\times n}$ which on each projection is
a closed interval as an interval matrix. For an interval matrix $\mathbf{A}%
\subset \mathbb{R}^{k\times n}$ we shall write%
\begin{equation*}
\left\Vert \mathbf{A}\right\Vert =\sup \left\{ \left\Vert A\right\Vert :A\in
\mathbf{A}\right\} .
\end{equation*}%
For a matrix $A\in \mathbb{R}^{k\times k}$ we write%
\begin{equation*}
m\left( A\right) =\left\{
\begin{array}{ll}
\inf_{\Vert x\Vert =1}\left\Vert Ax\right\Vert =\left\Vert A^{-1}\right\Vert
^{-1} & \qquad \text{if det}(A)\neq 0, \\
0 & \qquad \text{otherwise.}%
\end{array}%
\right.
\end{equation*}%
%
%
%
%
%
%
%
%
%
%
%
%
%
%
%
%
%
%
%
%
%
%
%
%
%
%
%We have $m\left( A\right)=\inf_{|x\|=1}\left\Vert Ax\right\Vert$.
For an interval matrix $\mathbf{A}\subset \mathbb{R}^{k\times k}$ we define%
\begin{equation*}
m\left( \mathbf{A}\right) =\inf \left\{ m\left( A\right) :A\in \mathbf{A}%
\right\} .
\end{equation*}

We now describe how we validate correct alignment of windows. To simplify
the discussion, here we restrict to the simple case when the dimension of
expanding coordinate is $1$. (This is the setting we encounter in the
PER3BP.)

For $z=\left( z_{1},z_{2}\right) \in \mathbb{R}^{n_{1}}\times \mathbb{R}%
^{n_{2}}$ we shall write $\pi _{z_{1}}$ and $\pi _{z_{2}}$ for the
projections onto the coordinate $z_{1}$ and $z_{2}$, respectively. To
validate correct alignment, in our computer assisted proof we use the
following lemma.

\begin{lemma}
\label{lem:covering}Let $N,M$ be two windows in $\mathbb{R}^{n_{1}}\times
\mathbb{R}^{n_{2}}$, with $n_{1}=1$, of the form%
\begin{equation*}
N=M=\left[ -1,1\right] \times \bar{B}^{n_{2}}.
\end{equation*}%
If $\pi _{z_{2}}f_{\varepsilon }\left( N\right) \subset B^{n_{2}}$, $\pi
_{z_{1}}f_{\varepsilon}\left( \left\{ -1\right\} \times \bar{B}^{n_{2}}\right) <-1$ and $%
\pi _{z_{1}}f_{\varepsilon}\left( \left\{ 1\right\} \times \bar{B}^{n_{2}}\right) >1, $
then $N\overset{f_{\varepsilon}}{\implies }M.$
\end{lemma}

\begin{proof}
Let $\chi :\left[ 0,1\right] \times N\rightarrow \mathbb{R}^{n_{1}}\times
\mathbb{R}^{n_{2}}$ be defined as
\begin{equation*}
\chi \left( \lambda ,z_{1},z_{2}\right) =\left( 1-\lambda \right) f_{\varepsilon}\left(
z_{1},z_{2}\right) +\left( \lambda 2z_{1},0\right) .
\end{equation*}%
It follows directly that $\chi $ satisfies the conditions from Definition %
\ref{def:covering}.
\end{proof}

For the case when $n_{1}>1$ there are established methods for the validation
of correct alignment, and we direct the reader to \cite{GideaZ04a,Zcc}.

%\marginpar{changed order to $(x,y,I,\theta )$ to unify with other sections}
We now describe the method with which we validate cone conditions. Consider
cones $Q_{a}^{\varepsilon }$ defined as in \eqref{eq:cone-Q} with $a=\left(
a_{y},a_{I},a_{\theta }\right) $. Let $E$ be an interval of parameters in $%
\mathbb{R}$. We consider a $C^{2}$-function
\begin{equation*}
f:E\times \mathbb{R}^{u}\times \mathbb{R}^{s}\times \mathbb{R}\times \mathbb{%
T\rightarrow R}^{u}\times \mathbb{R}^{s}\times \mathbb{R}\times \mathbb{T}.
\end{equation*}%
(Here it will be more convenient to write $f\left( \varepsilon ,z\right) $
than $f_{\varepsilon }(z)$; for $\varepsilon \in E$.) We consider
\begin{equation*}
N=\bar{B}^{u}\times \bar{B}^{s}\times J\times S,
\end{equation*}%
where $J$ is a closed interval in $\mathbb{R}$ and $S$ is a closed interval
in $[0,2\pi )$. Our objective will be to find $b=\left(
b_{y},b_{I},b_{\theta }\right) $ so that $Q_{a}^{\varepsilon }\left(
z-z^{\prime }\right) >0$ will imply $Q_{b}^{\varepsilon }(f(\varepsilon
,z)-f(\varepsilon ,z^{\prime }))>0$.

We use the following notation (below are interval matrices defined as in (%
\ref{eq:interval-derivative}--\ref{eq:interval-derivative-2}))
\begin{equation*}
\left[ \frac{\partial f_{\kappa }}{\partial h}\right] =\left[ \frac{\partial
f_{\kappa }}{\partial h}\left( E,N\right) \right] ,\qquad \left[ \frac{%
\partial ^{2}f_{\kappa }}{\partial \varepsilon \partial h}\right] =\left[
\frac{\partial ^{2}f_{\kappa }}{\partial \varepsilon \partial h}\left(
E,N\right) \right] ,\qquad \text{for }h,\kappa \in \left\{
x,y,I,\theta\right\} .
\end{equation*}

Below lemma is the main tool for validating cone conditions.%\marginpar{\color{red} \hfill 21.}

\correctiona{\begin{lemma}
\label{th:param-dep-cones} If for $b=\left( b_{y},b_{I},b_{\theta }\right)
\in \mathbb{R}_{+}^{3}$ we have
\begin{align*}
0& <\frac{\left\Vert \left[ \frac{\partial f_{\kappa }}{\partial x}\right]
\right\Vert +\left\Vert \left[ \frac{\partial f_{\kappa }}{\partial y}\right]
\right\Vert a_{y}+\varepsilon _{0}\left\Vert \left[ \frac{\partial f_{\kappa
}}{\partial I}\right] \right\Vert a_{I}+\left\Vert \left[ \frac{\partial
f_{\kappa }}{\partial \theta }\right] \right\Vert a_{\theta }}{m\left( \frac{%
\partial f_{x}}{\partial x}\right) -\left\Vert \left[ \frac{\partial f_{x}}{%
\partial y}\right] \right\Vert a_{y}-\varepsilon _{0}\left\Vert \left[ \frac{%
\partial f_{x}}{\partial I}\right] \right\Vert a_{I}-\left\Vert \left[ \frac{%
\partial f_{x}}{\partial \theta }\right] \right\Vert a_{\theta }}<b_{\kappa
}\qquad \text{for }\kappa \in \left\{ y,\theta \right\} , \\
0& <\frac{\left\Vert \left[ \frac{\partial f_{I}}{\partial \varepsilon
\partial x}\right] \right\Vert +\left\Vert \left[ \frac{\partial f_{y}}{%
\partial \varepsilon \partial y}\right] \right\Vert a_{y}+\left(
1+\varepsilon _{0}\left\Vert \left[ \frac{\partial f_{I}}{\partial
\varepsilon \partial I}\right] \right\Vert \right) a_{I}+\left\Vert \left[
\frac{\partial f_{I}}{\partial \varepsilon \partial \theta }\right]
\right\Vert a_{\theta }}{m\left( \frac{\partial f_{x}}{\partial x}\right)
-\left\Vert \left[ \frac{\partial f_{x}}{\partial y}\right] \right\Vert
a_{y}-\varepsilon _{0}\left\Vert \left[ \frac{\partial f_{x}}{\partial I}%
\right] \right\Vert a_{I}-\left\Vert \left[ \frac{\partial f_{x}}{\partial
\theta }\right] \right\Vert a_{\theta }}<b_{I}.
\end{align*}%
then
\begin{equation}
Q_{a}^{\varepsilon }\left( z-z^{\prime }\right) >0\qquad \implies \qquad
Q_{b}^{\varepsilon }\left( f(\varepsilon ,z)-f(\varepsilon ,z^{\prime
})\right) >0.  \label{eq:param-dep-cone-prop}
\end{equation}
\end{lemma}}

\begin{proof}
The proof is given in the Appendix.
\end{proof}

Lemma \ref{th:param-dep-cones} is used as follows. When validating cone
conditions for a connecting sequence $(N_{\ell ,0},\ldots ,N_{\ell ,k_{\ell
}})$, for some $\ell \in L$, we start with the cone $Q_{a}^{\varepsilon }$
in $N_{\ell ,0}$. We recall that by (\ref{eq:cones-same}) we take the same $%
a=\left( a_{y},a_{I},a_{\theta }\right) $ for all $\ell \in L$. We take $%
b_{0}=a$ and validate the cone conditions for the maps $f_{\ell
,i,\varepsilon }$, for $i=1,\ldots ,k_{\ell }$, by inductively applying
Lemma \ref{th:param-dep-cones}. This way we obtain a sequence $%
b_{i}=\left( b_{i,y},b_{i,I},b_{i,\theta }\right) $, for $i=0,\ldots
,k_{\ell }$, which imply that if $z,z^{\prime }\in N_{\ell ,i-1}$ and $%
Q_{b_{i-1}}^{\varepsilon }\left( z-z^{\prime }\right) >0$ then
\[Q_{b_{i}}^{\varepsilon }\left( f_{\ell ,i,\varepsilon }\left( z\right)
-f_{\ell ,i,\varepsilon }\left( z^{\prime }\right) \right) >0, \qquad\mbox{ for }
i=1,\ldots ,k_{\ell }.\]
Our objective is to return back to the cone $%
Q_{a}^{\varepsilon }$ in $N_{\ell ,k_{\ell }}$. This is established when $%
b_{k_{\ell },y}\leq a_{y},$ $b_{k_{\ell },I}\leq a_{I},$ and $b_{k_{\ell
},\theta }\leq a_{\theta }$, since in such case $Q_{b_{k_{\ell
}}}^{\varepsilon }\left( z-z^{\prime }\right) >0$ implies $%
Q_{a}^{\varepsilon }\left( z-z^{\prime }\right) >0.$

We now give two lemmas which we use to verify conditions \hyperlink{cond:C1.iv}{(C1.iv)}, \hyperlink{cond:C2.iv}{(C2.iv}--\hyperlink{cond:C2.v}{C2.v)}.

\begin{lemma}
\label{lem:c-bound-2}Consider two families of functions $f_{1,\varepsilon }$%
, $f_{0,\varepsilon }$ for $\varepsilon \in \left[ 0,\varepsilon _{0}\right]
.$ Consider also two windows $N_{0}$ and $N_{1}$. If for any $q_{0}\in N_{0}$
and $q_{1}\in N_{1}$ and any $\varepsilon \in \left[ 0,\varepsilon _{0}%
\right] $%
\begin{align*}
\pi _{I}q_{0}+\varepsilon c_{0}& <\pi _{I}f_{0,\varepsilon }(q_{0})<\pi
_{I}q_{0}+\varepsilon C_{0} \\
\pi _{I}q_{1}+\varepsilon c_{1}& <\pi _{I}f_{1,\varepsilon }(q_{1})<\pi
_{I}q_{1}+\varepsilon C_{1},
\end{align*}%
then for any $q_{0}\in N_{0},$ for which $f_{0,\varepsilon }(q_{0})\in
N_{1}, $ and for any $\varepsilon \in \left[ 0,\varepsilon_0 \right] $,
\begin{equation*}
\pi _{I}q_{0}+\varepsilon c_{0}+\varepsilon c_{1}<\pi _{I}f_{1,\varepsilon
}\circ f_{0,\varepsilon }(q_{0})<\pi _{I}q_{0}+\varepsilon C_{0}+\varepsilon
C_{1}.
\end{equation*}
\end{lemma}

\begin{proof}
Taking $q_{1}=f_{0,\varepsilon }(q_{0})$ we have%
\begin{multline*}
\pi _{I}f_{1,\varepsilon }\circ f_{0,\varepsilon }(q_{0})=\pi
_{I}f_{1,\varepsilon }(q_{1})>\pi _{I}q_{1}+\varepsilon c_{1} \\
=\pi _{I}f_{0,\varepsilon }(q_{0})+\varepsilon c_{1}>\pi
_{I}q_{0}+\varepsilon c_{0}+\varepsilon c_{1}.
\end{multline*}%
The upper bound follows from mirror computation, but with reversed
inequalities and by using $C_{0}$ and $C_{1}$ instead of $c_{0}$ and $c_{1}$%
, respectively.
\end{proof}

Lemma \ref{lem:c-bound-2} can be iterated by passing through a connecting
sequence. One can verify assumptions of Lemma \ref{lem:c-bound-2} as follows
(denoting $f_{i}\left( \varepsilon ,q\right):=f_{i,\varepsilon }\left(
q\right)$):

\begin{lemma}
\label{lem:c-bound-3}Assume that
\begin{equation}
c<\left[ \min_{\left( \varepsilon ,q\right) \in \left[ 0,\varepsilon _{0}%
\right] \times N}\frac{\partial \pi _{I}f}{\partial \varepsilon }%
(\varepsilon ,q),\max_{\left( \varepsilon ,q\right) \in \left[ 0,\varepsilon_{0}\right] \times N}\frac{\partial \pi _{I}f}{\partial \varepsilon }%
(\varepsilon ,q)\right] <C,  \label{eq:f-eps-c-bound}
\end{equation}
then for every $q\in N$ and every $\varepsilon \in \left[ 0,\varepsilon _{0}%
\right] $
\begin{equation*}
\pi _{I}q+\varepsilon c<\pi _{I}f(\varepsilon ,q)<\pi _{I}q+\varepsilon C.
\end{equation*}
\end{lemma}

\begin{proof}
For any $q\in N$ and any $\varepsilon \in \left[ 0,\varepsilon_0 \right] $
\begin{multline*}
\pi _{I}f(\varepsilon ,q)=\pi _{I}f(0,q)+\int_{0}^{1}\frac{d}{ds}\pi
_{I}f(s\varepsilon ,q)ds=\pi _{I}q+\varepsilon \int_{0}^{1}\frac{\partial }{%
\partial \varepsilon }\pi _{I}f(s\varepsilon ,q)ds \\
\in \pi _{I}q+\varepsilon \left[ \frac{\partial \pi _{I}f}{\partial
\varepsilon }(\left[ 0,\varepsilon_0 \right] ,N)\right] \in \left( \pi
_{I}q+\varepsilon c,\pi _{I}q+\varepsilon C\right) ,
\end{multline*}%
which completes the proof.
\end{proof}

\subsection{Construction of local coordinates\label%
{sec:main-proof-local-coord}}

\label{sec:local-coordinates}

Here we discuss the local coordinates which we use for our maps. At any
given point $q^{\ast }=\left( X^{\ast },Y^{\ast },P_{X}^{\ast },P_{Y}^{\ast
},\theta ^{\ast }\right) \in \mathbb{R}^{4}\times \mathbb{T}$, we define a $%
4 $-dimensional section $\Sigma =\Sigma \left( q^{\ast }\right) \subset
\mathbb{R}^{4}\times \mathbb{T}$ as in \eqref{eqn:sections}. Let $%
F_{Y}^{\ast }:=P_{X}^{\ast }+Y^{\ast }$, $F_{Y}^{\ast }:=P_{Y}^{\ast
}-X^{\ast }$, which correspond to the vector field (\ref{eq:PER3BP-ode}) at $%
q^{\ast}$ along the $X,Y$ coordinates, respectively. We consider two cases:

\begin{description}
\item[Case 1.] If $\left\vert F_{X}^{\ast }\right\vert >\left\vert
F_{Y}^{\ast }\right\vert $, let $a:=-F_{Y}^{\ast }/F_{X}^{\ast }$ and define
\begin{equation}
\Sigma :=\left\{ \left( X(Y),Y,P_{X},P_{Y},\theta \right) |X\left( Y\right)
=a\left( Y-Y^{\ast }\right) +X^{\ast }\right\} .  \label{eqn:section_Y}
\end{equation}%
This means that $\Sigma $ is parameterized by $\left( Y,P_{X},P_{Y},\theta
\right) $. On this section we define local coordinates $v=\left(
x,y,I,\theta \right) \in \mathbb{R}^{3}\times \mathbb{T}$ as in %
\eqref{eqn:v_coordinates}. These are given by
\begin{equation}
\left( X,Y,P_{X},P_{Y},\theta \right) =\Psi _{\varepsilon }\left( v\right)
:=R_{1}\left( p+\varepsilon w+Av\right) ,  \label{eq:change-case1}
\end{equation}%
where $p,w\in \mathbb{R}^{3}\times \mathbb{T}$ and $A$ is a suitable $%
4\times 4$ matrix (the particular choices of $p,w$ and $A$ we make are given
by \eqref{eq:qh-form}, \eqref{eq:w-form}, and \eqref{eq:A-coord-change-form}%
), and where
\begin{align*}
R_{1}\left( Y,P_{Y},h,\theta \right) & :=\left\{
\begin{array}{lll}
\left( X\left( Y\right) ,Y,\psi _{1}\left( Y,P_{Y},h\right) -Y,P_{Y},\theta
\right)  &  & \text{if }F_{X}^{\ast }>0, \\
\left( X(Y),Y,-\psi _{1}\left( Y,P_{Y},h\right) -Y,PY,\theta \right)  &  &
\text{otherwise,}%
\end{array}%
\right.  \\
\psi _{1}\left( Y,P_{Y},h\right) & :=\sqrt{2\left( h+\Omega \left(
X(Y),Y\right) \right) -\left( P_{Y}-X(Y)\right) ^{2}}.
\end{align*}%
Note that the inverse of $R_{1}$ is $R_{1}^{-1}:\Sigma \rightarrow \mathbb{R}%
^{3}\times \mathbb{T}$%
\begin{equation}
R_{1}^{-1}\left( X,Y,P_{X},P_{Y},\theta \right) =\left( Y,P_{Y},H_{0}\left(
X,Y,P_{X},P_{Y}\right) ,\theta \right) .  \label{eq:energy-prop-Q1}
\end{equation}

\item[Case 2.] If $\left\vert F_{X}^{\ast }\right\vert \leq \left\vert
F_{Y}^{\ast }\right\vert $ then we let $a:=-F_{X}^{\ast }/F_{Y}^{\ast }$ and
define the section as
\begin{equation}
\Sigma =\left\{ \left( X,Y(X),P_{X},P_{Y},\theta \right) |Y(X)=a\left(
X-X^{\ast }\right) +Y^{\ast }\right\} .  \label{eqn:section_X}
\end{equation}%
This means that $\Sigma $ is parameterized by $\left( X,P_{X},P_{Y},\theta
\right) $. We define local coordinates $v=\left( x,y,I,\theta \right) $ on $%
\Sigma $ as%
\begin{equation}
\left( X,Y,P_{X},P_{Y},\theta \right) =\Psi _{\varepsilon }\left( v\right)
:=R_{2}\left( p+\varepsilon w+Av\right) ,  \label{eq:Psi-coord-change}
\end{equation}%
where $A$ is some $4\times 4$ matrix, $p,w$ are some points in $\mathbb{R}%
^{3}\times \mathbb{T}$ and%
\begin{align*}
R_{2}\left( X,P_{X},h,\theta \right) & :=\left\{
\begin{array}{lll}
\left( X,Y(X),P_{X},\psi _{2}\left( X,P_{X},h\right) +X,\theta \right)  &  &
\text{if }F_{Y}^{\ast }>0, \\
\left( X,Y(X),P_{X},-\psi _{2}\left( X,P_{X},h\right) +X,\theta \right)  &
& \text{otherwise,}%
\end{array}%
\right.  \\
\psi _{2}\left( X,P_{X},h\right) & :=\sqrt{2\left( h+\Omega \left(
X,Y(X)\right) \right) -\left( P_{X}+Y(X)\right) ^{2}}.
\end{align*}%
Note that the inverse of $R_{2}$ is $R_{2}^{-1}:\Sigma \rightarrow \mathbb{R}%
^{3}\times \mathbb{T}$%
\begin{equation}
R_{2}^{-1}\left( X,Y,P_{X},P_{Y},\theta \right) =\left( X,P_{X},H_{0}\left(
X,Y,P_{X},P_{Y}\right) ,\theta \right) .  \label{eq:energy-prop-Q2}
\end{equation}
\end{description}

When dealing with Case 1, the section $\Sigma $ is transverse to the flow at
$q^{\ast }$ since the vector $\left( F_{X}^{\ast },F_{Y}^{\ast }\right) $ is
orthogonal to the line $\left( X(Y),Y\right) $ in the $(X,Y)$-plane. The
coordinate change $R_{1}\left( Y,P_{Y},h,\theta \right) $ allows us to use
the energy $h$ instead of the coordinate $P_{X}$. The function $R_{1}$ was
chosen so that%
\begin{equation}
H_{0}\left( R_{1}\left( Y,P_{Y},h,\theta \right) \right) =h.
\label{eq:H-Q1-property}
\end{equation}%
Similarly, when dealing with Case 2, the section $\Sigma $ is transverse to
the flow since $\left( F_{X}^{\ast },F_{Y}^{\ast }\right) $ is orthogonal to
$\left( X,Y(X)\right) ,$ and $R_{2}\left( X,P_{X},h,\theta \right) $ allows
us to use $h$ instead of $P_{Y}$. We also have%
\begin{equation}
H_{0}\left( R_{2}\left( X,P_{X},h,\theta \right) \right) =h.
\label{eq:H-Q2-property}
\end{equation}

\begin{remark}
\label{rem:sections_rigorous} We stress that a section $\Sigma $ as in %
\eqref{eqn:section_Y} (resp. \eqref{eqn:section_X}) is rigorously defined,
as it is obtained by fixing one coordinate $Y$ (resp. $X$) and solving for
the remaining coordinates $X,P_{X},P_{Y},\theta $ (resp. $%
Y,P_{X},P_{Y},\theta $). These coordinates are later transformed into the
coordinates $v=(x,y,I,\theta )$, which are then used in the numerical
computations.% (both, non-rigorous and rigorous) of the section-to-section
%maps. The points that are obtained via the numerical computations correspond
%to points that lie on the rigorously defined sections.
\end{remark}

In all of our coordinate changes, we will always choose%
\begin{equation}
p=\left\{
\begin{array}{lll}
\left( Y^{\ast },P_{Y}^{\ast },H_{0}\left( q^{\ast }\right) ,\theta ^{\ast
}\right) & \qquad & \text{if }q^{\ast }\text{ is as in Case 1,} \\
\left( X^{\ast },P_{X}^{\ast },H_{0}\left( q^{\ast }\right) ,\theta ^{\ast
}\right) &  & \text{if }q^{\ast }\text{ is as in Case 2,}%
\end{array}%
\right.  \label{eq:qh-form}
\end{equation}%
always take $w$ of the form
\begin{equation}
w=\left( w_{x},w_{y},0,0\right)  \label{eq:w-form}
\end{equation}%
and always choose $A$ of the form%
\begin{equation}
A=\left(
\begin{array}{llll}
a_{11} & a_{12} & a_{13} & 0 \\
a_{21} & a_{22} & a_{23} & 0 \\
0 & 0 & 1 & 0 \\
0 & 0 & 0 & 1%
\end{array}%
\right) .  \label{eq:A-coord-change-form}
\end{equation}

\begin{lemma}
\label{lem:H-I-relation}If $p,$ $w$ and $A$ are of the form (\ref{eq:qh-form}%
--\ref{eq:A-coord-change-form}) then
\begin{equation*}
H_{0}(\Psi _{\varepsilon }\left( x,y,I,\theta \right) )=H_{0}(q^{\ast })+I.
\end{equation*}
\end{lemma}

\begin{proof}
In Case 1, by (\ref{eq:change-case1}), (\ref{eq:qh-form}) and (\ref%
{eq:A-coord-change-form}) we have%
\begin{equation*}
\pi _{I}R_{1}^{-1}\left( \Psi _{\varepsilon }\left( x,y,I,\theta \right)
\right) =\pi _{I}\left( p+\varepsilon w+A\left( x,y,I,\theta \right) \right)
=H_{0}\left( q^{\ast }\right) +0+I,
\end{equation*}%
and the result follows from (\ref{eq:energy-prop-Q1}). In Case 2 the proof
is identical, using $R_{2}$ instead of $R_{1}$ and (\ref{eq:energy-prop-Q2})
instead of (\ref{eq:energy-prop-Q1}).
\end{proof}

From now on we assume that all subsequent coordinate changes defined by (\ref%
{eq:change-case1}) or (\ref{eq:Psi-coord-change}), involving $p,w$ and $A$,
take them of the form (\ref{eq:qh-form}--\ref{eq:A-coord-change-form}). The
choice of the coefficients of $w$ and $A$ can depend on the choice of $%
q^{\ast }$.

\begin{theorem}
\label{th:return-change-H}Consider a sequence of points $q_{0}^{\ast
},\ldots ,q_{k}^{\ast }$ with $q_{k}^{\ast }=q_{0}^{\ast }$ at which we
position sections $\Sigma _{0},\ldots ,\Sigma _{k},$ with $\Sigma
_{k}=\Sigma _{0}$. Consider section-to-section maps along the flow $\mathcal{%
P}_{i}^{\varepsilon }:\Sigma _{i-1}\rightarrow \Sigma _{i}$ together with
local maps $f_{i,\varepsilon }=\Psi _{i,\varepsilon }^{-1}\circ \mathcal{P}%
_{i}^{\varepsilon }\circ \Psi _{i-1,\varepsilon }$, for $i=1,\ldots ,k$. If $%
w_{k}=w_{0}$ and $A_{k}=A_{0}$, then for $q=\Psi _{0,\varepsilon }\left(
v\right) $%
\begin{equation*}
H_{0}\left( \mathcal{P}_{k}^{\varepsilon }\circ \ldots \circ \mathcal{P}%
_{1}^{\varepsilon }\left( q\right) \right) -H_{0}\left( q\right) =\pi
_{I}f_{k,\varepsilon }\circ \ldots \circ f_{1,\varepsilon }(v)-\pi _{I}v.
\end{equation*}
\end{theorem}

\begin{proof}
Since $q_{k}^{\ast }=q_{0}^{\ast },$ $w_{k}=w_{0}$ and $A_{k}=A_{0}$ we see
that $\Psi _{k,\varepsilon }=\Psi _{0,\varepsilon }$, so from the definition
of the maps $f_{i,\varepsilon }$ it follows that
\begin{equation*}
\Psi _{0,\varepsilon }\circ f_{k,\varepsilon }\circ \ldots \circ
f_{1,\varepsilon }(v)=\mathcal{P}_{k}^{\varepsilon }\circ \ldots \circ
\mathcal{P}_{1}^{\varepsilon }\circ \Psi _{0,\varepsilon }(v).
\end{equation*}%
From the fact that $q=\Psi _{0,\varepsilon }\left( v\right) $ and by
applying Lemma \ref{lem:H-I-relation} we see that
\begin{eqnarray*}
H_{0}\left( \mathcal{P}_{k}^{\varepsilon }\circ \ldots \circ \mathcal{P}%
_{1}^{\varepsilon }\left( q\right) \right) -H_{0}\left( q\right)
&=&H_{0}\left( \Psi _{0,\varepsilon }\left( f_{k,\varepsilon }\circ \ldots
\circ f_{1,\varepsilon }(v)\right) \right) -H_{0}\left( \Psi _{0,\varepsilon
}(v)\right) \\
&=&H_{0}(q_{0}^{\ast })+\pi _{I}f_{k,\varepsilon }\circ \ldots \circ
f_{1,\varepsilon }(v)-H_{0}(q_{0}^{\ast })-\pi _{I}v \\
&=&\pi _{I}f_{k,\varepsilon }\circ \ldots \circ f_{1,\varepsilon }(v)-\pi
_{I}v,
\end{eqnarray*}%
as required.
\end{proof}

\begin{remark}
Theorem \ref{th:return-change-H} ensures that regardless of the choices of
the particular coefficients of $w_{i}$ and $A_{i}$, if we return to the same
local coordinates, then the change in $I$ in the local coordinates
corresponds precisely to the change in $H_{0}$ in the original coordinates.
\end{remark}

Let $q_{1}^{\ast }$ and $q_{2}^{\ast }$ be two points which lie on a
trajectory of the flow. Let $\Sigma _{1}$ and $\Sigma _{2}$ be two sections
at $q_{1}^{\ast }$ and $q_{2}^{\ast },$ respectively, and let $\mathcal{P}%
^{\varepsilon }:\Sigma _{1}\rightarrow \Sigma _{2}$ be a section-to-section
map along the flow. We will choose the matrices $A_{1},A_{2}$ of the form (%
\ref{eq:A-coord-change-form}) to define local coordinate changes $\Psi
_{1,\varepsilon },\Psi _{2,\varepsilon }$ respectively, to have coefficients
so that for the local map
\begin{equation}
f_{\varepsilon }:=\Psi _{2,\varepsilon }^{-1}\circ \mathcal{P}^{\varepsilon
}\circ \Psi _{1,\varepsilon },  \label{eq:local-map-form}
\end{equation}%
we obtain%
\begin{equation}
Df_{\varepsilon =0}(0)=\left(
\begin{array}{cccc}
\lambda & 0 & 0 & 0 \\
0 & \frac{1}{\lambda } & 0 & 0 \\
0 & 0 & 1 & 0 \\
\Theta _{1} & \Theta _{2} & \Theta _{3} & 1%
\end{array}%
\right) ,  \label{eq:derivative-form-of-local-map}
\end{equation}%
where $\lambda >1$ and $\Theta _{1},$ $\Theta _{2},$ $\Theta _{3},$ are some
numbers. (The presence of $\Theta _{1},$ $\Theta _{2},$ $\Theta _{3}$ is a
result of having zeros in the lower left part of the matrices (\ref%
{eq:A-coord-change-form}).)

A good choice of $w_{1},w_{2}$ of the form (\ref{eq:w-form}) is one for
which
\begin{equation}
\frac{d}{d\varepsilon }\pi _{x,y}f_{\varepsilon }(0)|_{\varepsilon =0}=0,
\label{eq:choice-of-w}
\end{equation}%
or, if this is not possible, one that makes the left hand side of (\ref%
{eq:choice-of-w}) as close to zero as possible.

The reason for (\ref{eq:derivative-form-of-local-map}) and (\ref%
{eq:choice-of-w}) is that then for $v=\left( x,y,I,\theta \right) $ which is
close to zero and for $\varepsilon $ also close to zero we will have $\pi
_{x,y}f_{\varepsilon }(v)\approx \left( \lambda x,y/\lambda \right) .$ This
is useful for the validation of correct alignment of windows.

\begin{remark}
\label{rem:I-is-energy} We choose $A_{i}$ and $w_{i}$, for $i=1,2$, by
solving (\ref{eq:derivative-form-of-local-map}) and (\ref{eq:choice-of-w}).
We do not need though to solve these equations analytically. It is
sufficient to find their solutions by means of non-rigorous numerical
computations. It is then important to compute rigorous enclosures of $\Psi
_{i,\varepsilon }$ and $\Psi _{i,\varepsilon }^{-1}$ (for the particular
choices of $w_{i}$ and $A_{i}$ which we decide on) when performing rigorous,
interval arithmetic based validation of the needed conditions.
\end{remark}

\begin{figure}[tbp]
\begin{center}
\includegraphics[height=4cm]{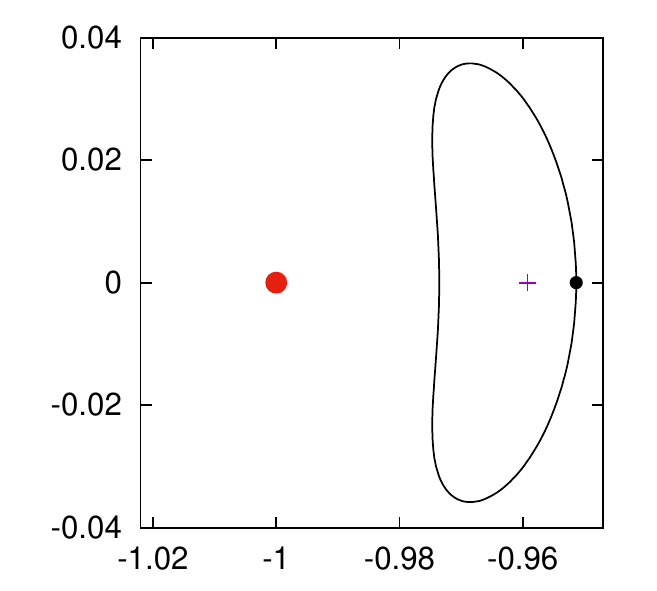}\qquad %
\includegraphics[height=4cm]{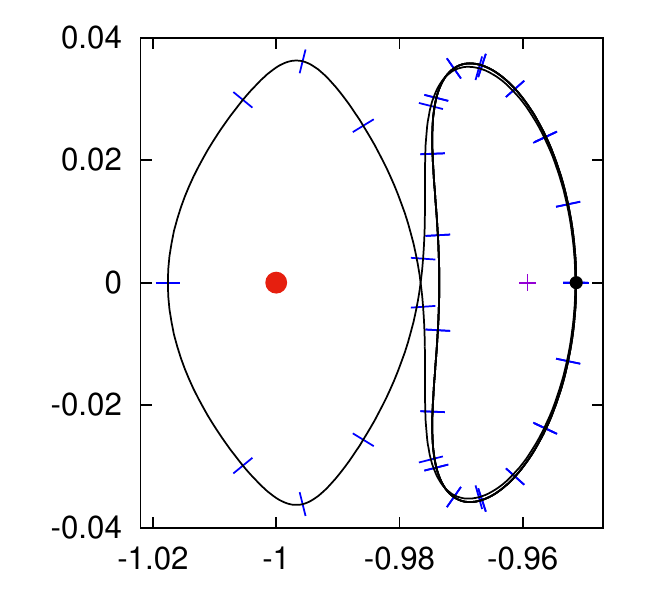}
\end{center}
\caption{The Lyapunov orbit to the left, and its homoclinic orbit to the
right. The point $L_{1}$ is depicted with a cross and the smaller primary is
depicted with a dot. Along the homoclinic we have a sequence of points, at
which we position sections, which are transversal to the flow. The dot on
the Lyapunov orbit indicates where the strips $S_{u}$ and $S_{d}$ are
positioned.}
\label{fig:LyapOrb}
\end{figure}

\subsection{Construction of sections and energy strips\label%
{sec:main-proof-sections}}

\label{sec:choice_of_sections} In this section we describe how we choose a
sequence of points at which we position sections. The section $\Sigma
=\{Y=0\}$ will play a special role in our construction, since on this
section we shall position our energy strips $\mathbf{S}^{u}$ and $\mathbf{S}%
^{d}$ as in (\ref{eq:strips-Sa-Sb}). We define the local coordinates on $%
\Sigma =\{Y=0\}$ by choosing
\begin{eqnarray}
q^{\ast } &=&\left( -0.9513385,0,0,-1.02124587611,0\right) ,
\label{eq:q-star-3bp} \\
p &=&\left( \pi _{X}q^{\ast },\pi _{P_{X}}q^{\ast },H_{0}\left( q^{\ast
}\right) ,0\right) ,  \label{eq:qh-3bp} \\
w &=&\left( 0,0,0,0\right) ,
\end{eqnarray}%
and taking%
\begin{equation}
A=\left(
\begin{array}{llll}
0.377372287914 & 0.377372287914 & 1.53559852923 & 0 \\
0.926061637427 & -0.926061637427 & 0 & 0 \\
0 & 0 & 1 & 0 \\
0 & 0 & 0 & 1%
\end{array}%
\right) .  \label{eq:A-3bp}
\end{equation}

At $q^{\ast }$ the vector field in the direction $X$ is zero, which means
that at the section $\Sigma =\left\{ Y=0\right\} $ we associate the local
coordinates from Case 2. This way we obtain a change to local coordinates at
$\Sigma =\left\{ Y=0\right\} $ defined as
\begin{equation}
\Psi \left( v\right) =R_{2}\left( p+Av\right) .
\label{eq:Psi-change-at-strip}
\end{equation}%
We note that since we take $w=0$ to define $\Psi $, it is independent of $%
\varepsilon $. (This will not always be the case and on other sections which
are used in our construction we can choose $w$ to be non zero.) The
particular coefficients of $A$ in (\ref{eq:A-3bp}) have been chosen so that
we obtain (\ref{eq:derivative-form-of-local-map}) for the local maps used in
the connecting sequences, which are introduced in Section \ref%
{sec:main-proof-connecting-sequences}.

We consider two energy strips in the local coordinates given by $\Psi
_{\varepsilon }$ as
\begin{eqnarray}
\mathbf{S}^{u} &=&\left[ -r,r\right] \times \left[ -r,r\right] \times
\mathbb{R}\times \lbrack \theta _{1}^{u},\theta _{2}^{u}],
\label{eq:Su-choice} \\
\mathbf{S}^{d} &=&\left[ -r,r\right] \times \left[ -r,r\right] \times
\mathbb{R}\times \lbrack \theta _{1}^{d},\theta _{2}^{d}],  \notag
\end{eqnarray}%
where $r=10^{-6}$ and
\begin{equation*}
\begin{array}{lll}
\theta _{1}^{u}=\frac{1}{2}\pi -0.08, & \qquad  & \theta _{2}^{u}=\frac{1}{2}%
\pi +0.08,\smallskip  \\
\theta _{1}^{d}=\frac{3}{2}\pi -0.08, &  & \theta _{2}^{d}=\frac{3}{2}\pi
+0.08.%
\end{array}%
\end{equation*}%
The coordinates in the above Cartesian product represent: expansion,
contraction, energy, angle, respectively. (In this case $n_{u}=n_{s}=n_{c}=1$%
.)

Both strips lie on the same section $\Sigma =\left\{ Y=0\right\} $. We shall
make a distinction though and write
\begin{equation*}
\Sigma _{u}:=\{Y=0\}\cap \{\theta \in \lbrack \theta _{1}^{u},\theta
_{2}^{u}]\}\qquad \text{and\qquad }\Sigma _{d}:=\{Y=0\}\cap \{\theta \in
\lbrack \theta _{1}^{d},\theta _{2}^{d}]\}.
\end{equation*}

The remainder of the points around which we position the sections are on a
homoclinic orbit to the Lyapunov orbit of the PCR3BP. The homoclinic orbit
is computed via non-rigorous numerics, and the corresponding points are
depicted in the right hand plot in Figure \ref{fig:LyapOrb}. We will take
sequences of points $\left\{ q_{i}\right\} _{i=0}^{n}$ on the homoclinic,
with $q_{0}\in \Sigma =\{Y=0\}$, and define a sequence of sections $\left\{
\Sigma _{i}\right\} _{i=0}^{n}$ and corresponding local coordinates as
discussed in Cases 1 and 2 from Section \ref{sec:local-coordinates}. The
sections $\Sigma _{i}$ are rigorously defined as per Remark \ref%
{rem:sections_rigorous}.

\begin{figure}[tbp]
\begin{center}
\includegraphics[width=12cm]{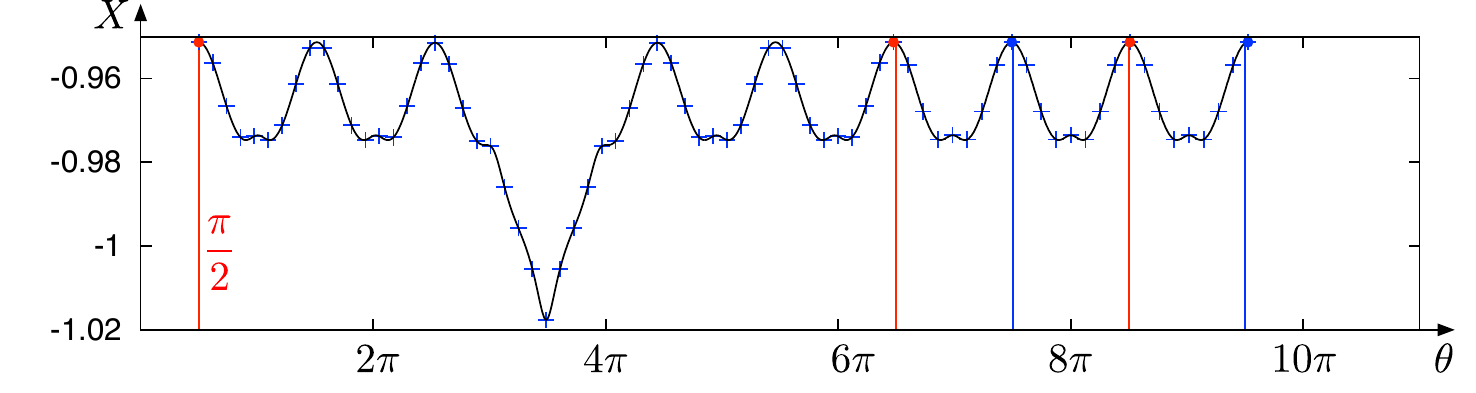}
\par
\includegraphics[height=2.5cm]{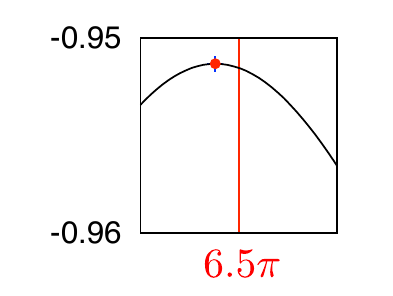} %
\includegraphics[height=2.5cm]{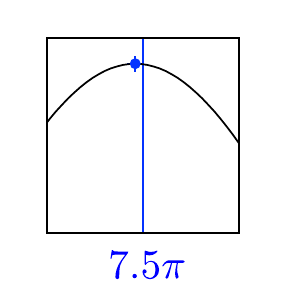} %
\includegraphics[height=2.5cm]{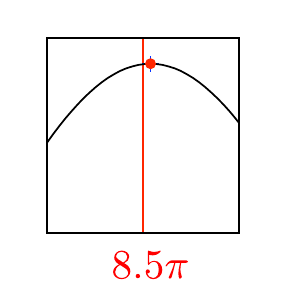} %
\includegraphics[height=2.5cm]{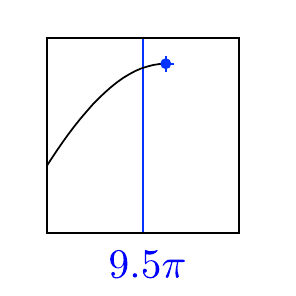}
\end{center}
\caption{\correctiona{The homoclinic orbit which starts from $\Sigma _{u}$. }
The crosses depict the points where the orbit intersects the sections
(compare with Figure \protect\ref{fig:LyapOrb}). The red and blue dots are
where it intersects with $\{Y=0\}$. The red lines are at $\protect\theta =%
\frac{1}{2}\protect\pi $ and the blue lines are at $\protect\theta =\frac{3}{%
2}\protect\pi $, modulo $2\protect\pi $. }
\label{fig:homoclinic_011}
\end{figure}

%\marginpar{\color{red} 22. \hfill 22.}
\correctiona{If we consider the homoclinic orbit in the extended phase space which
includes $\theta $, then we can observe the change of $\theta$ along $\left\{
q_{i}\right\} _{i=0}^{n}$. } We shall use the notation
\begin{equation*}
\theta _{0}=\pi _{\theta }q_{0}.
\end{equation*}%
We have chosen our Lyapunov orbit and its homoclinic from Figure \ref%
{fig:LyapOrb} so that for any choice of $\theta _{0}\in \lbrack 0,2\pi )$ we
have the following properties. (The discussion below is complemented by
Figure \ref{fig:homoclinic_011}, which corresponds to the case when $\theta
_{0}=\pi /2$.)

Depending on the number of turns the homoclinic orbit makes around the
Lyapunov orbit, it returns to $\Sigma =\{Y=0\}$ at a different angle. If we
start from $q_{0}\in \Sigma =\{Y=0\}$ and take $2$ turns around the Lyapunov
orbit, go around the smaller primary, and make $2$ turns around the Lyapunov
orbit again, then we arrive at a point $q_{50}\in \Sigma =\{Y=0\}$, for
\begin{equation}
\pi _{\theta }q_{50}=\theta _{0}-0.0750229.  \label{eq:n_uu_1}
\end{equation} which
(For the case when $\theta _{0}=\pi /2$, this feature is depicted in Figure %
\ref{fig:homoclinic_011}. On the plot, the points $q_{i}$ are depicted by
the crosses. The red and blue vertical lines represent the positions of
center angles of strips $\mathbf{S}^{u}$ and $\mathbf{S}^{d}$ ,
respectively. The projection of the point $q_{50}$ onto the $\theta ,X$
coordinates is depicted by the red dot above $6.5\pi $.)

When we take another turn along the homoclinic, we return to a point $
q_{58}\in \Sigma =\{Y=0\}$, for which%
\begin{equation}
\pi _{\theta }q_{58}=\theta _{0}+\pi -0.0249904.  \label{eq:n_ud_1}
\end{equation}%
(For the case when $\theta _{0}=\pi /2$ the projection of the point $q_{58}$
onto the $\theta ,X$ coordinates is depicted by the blue dot above $7.5\pi $
in Figure \ref{fig:homoclinic_011}.)

After another turn around the Lyapunov orbit we arrive at $q_{66}\in \Sigma
=\{Y=0\}$, for which%
\begin{equation}
\pi _{\theta }q_{66}=\theta _{0}+0.0250421.  \label{eq:n_uu_2}
\end{equation}%
(Red dot above $8.5\pi $ in Figure \ref{fig:homoclinic_011}.)

Yet another turn around the Lyapunov orbit results in $q_{74}\in \Sigma
=\{Y=0\}$, for which%
\begin{equation}
\pi _{\theta }q_{74}=\theta _{0}+\pi +0.0750746.  \label{eq:n_ud_2}
\end{equation}%
(Blue dot above $9.5\pi $ in Figure \ref{fig:homoclinic_011}.)

From (\ref{eq:n_uu_1}), (\ref{eq:n_ud_1}), (\ref{eq:n_uu_2}), and (\ref%
{eq:n_ud_2}) and looking at Figure \ref{fig:homoclinic_011}, we can see that
for $q_{0}\in \Sigma _{u}$ we will have:

\begin{itemize}
\item When $\theta _{0}\geq \pi /2,$ the composition of $n_{1}^{uu}:=50$
section-to-section maps takes the point $q_{0}\in \Sigma _{u}$ to $%
q_{n_{1}^{uu}}\in \Sigma _{u}$.

\item When $\theta _{0}\geq \pi /2,$ the composition of $n_{1}^{ud}:=58$
section-to-section maps takes the point $q_{0}\in \Sigma _{u}$ to $%
q_{n_{1}^{ud}}\in \Sigma _{d}$.

\item When $\theta _{0}\leq \pi /2,$ the composition of $n_{2}^{uu}:=66$
section-to-section maps takes the point $q_{0}\in \Sigma _{u}$ to $%
q_{n_{2}^{uu}}\in \Sigma _{u}$.

\item When $\theta _{0}\leq \pi /2,$ the composition of $n_{2}^{ud}:=74$
section-to-section maps takes the point $q_{0}\in \Sigma _{u}$ to $%
q_{n_{2}^{ud}}\in \Sigma _{d}$.
\end{itemize}

The choices of the superscripts in $n_{1}^{ui},n_{2}^{ui}$, for $i\in
\{u,d\} $ are to indicate that it takes such numbers of section-to-section
maps to get from $\Sigma _{u}$ to $\Sigma _{i}$.

The above was possible because we have chosen our Lyapunov orbit carefully.
Our Lyapunov orbit has a homoclinic orbit such that after we pass through an
excursion along it from $\{Y=0\}$ to $\{Y=0\}$ we ``move with the angle to
the left". Our Lyapunov orbit has a period slightly longer than $\pi $ so by
making turns around it from $\{Y=0\}$ to $\{Y=0\}$, we ``move with the angle
to the right".

\begin{figure}[tbp]
\begin{center}
\includegraphics[width=12cm]{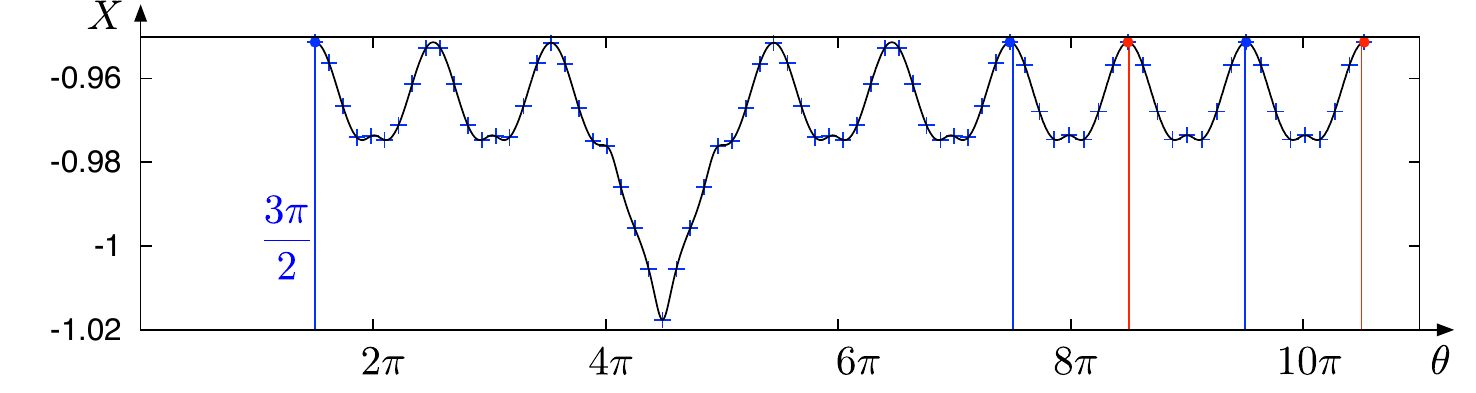}
\par
\includegraphics[height=2.5cm]{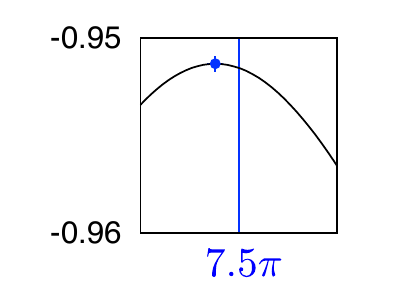} %
\includegraphics[height=2.5cm]{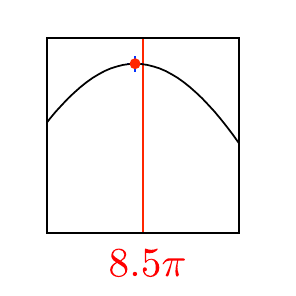} %
\includegraphics[height=2.5cm]{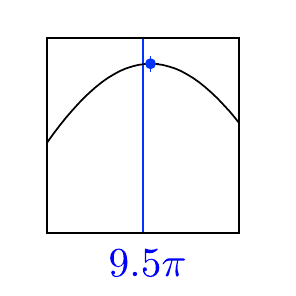} %
\includegraphics[height=2.5cm]{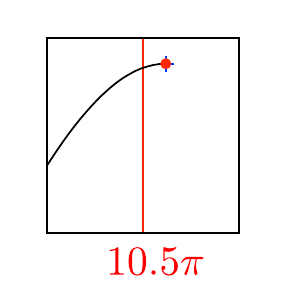}
\end{center}
\caption{\correctiona{The homoclinic orbit starting from $\Sigma _{d}$.}
(Compare with Figures \protect\ref{fig:LyapOrb} and \protect\ref%
{fig:homoclinic_011}). The red and blue dots are where the orbit intersects
with $\{Y=0\}$. The red lines are at $\protect\theta =\frac{1}{2}\protect\pi
$ and the blue lines are at $\protect\theta =\frac{3}{2}\protect\pi $,
modulo $2\protect\pi $.}
\label{fig:homoclinic_101}
\end{figure}

We can also choose $q_{0}\in \Sigma _{d}$. (The case of $\theta _{0}=3\pi /2$
is depicted in Figure \ref{fig:homoclinic_101}.) Then from (\ref{eq:n_uu_1}%
), (\ref{eq:n_ud_1}), (\ref{eq:n_uu_2}) and (\ref{eq:n_ud_2}) and looking at
Figure \ref{fig:homoclinic_101} we see that:

\begin{itemize}
\item When $\theta _{0}\geq 3\pi /2,$ the composition of $n_{1}^{dd}:=50$
section-to-section maps takes the point $q_{0}\in \Sigma _{d}$ to $%
q_{n_{1}^{dd}}\in \Sigma _{d}$.

\item When $\theta _{0}\geq 3\pi /2,$ the composition of $n_{1}^{du}:=58$
section-to-section maps takes the point $q_{0}\in \Sigma _{d}$ to $%
q_{n_{1}^{du}}\in $ $\Sigma _{u}$.

\item When $\theta _{0}\leq 3\pi /2,$ the composition of $n_{2}^{dd}:=66$
section-to-section maps takes the point $q_{0}\in \Sigma _{d}$ to $%
q_{n_{2}^{dd}}\in \Sigma _{d}$.

\item When $\theta _{0}\leq 3\pi /2,$ the composition of $n_{2}^{du}:=74$
section-to-section maps takes the point $q_{0}\in \Sigma _{d}$ to $%
q_{n_{2}^{du}}\in \Sigma _{u}$.
\end{itemize}

\begin{remark}
In the language of the scattering map theory \cite{DelshamsLS08a}, the above %\marginpar{\color{red} 23. \hfill 23.}
construction can be expressed by saying that we obtain pseudo-orbits
generated by composing the scattering map from the \correctiona{Lyapunov NHIM} \correctiona{(consisting of a family of Lyapunov orbits over some energy interval)} to itself,
with iterates of inner map restricted the \correctiona{Lyapunov NHIM. }Following the
scattering map yields a shift of the angle $\theta$ to the left, while
following the inner dynamics along the \correctiona{Lyapunov NHIM }yields a shift of the
angle $\theta$ to the right. These pseudo-orbits are shadowed by true orbits
\cite{gidea2019general}. In this paper we find these latter orbits directly,
without constructing the underlying pseudo-orbits.
\end{remark}

\begin{remark}
%There is nothing special in our geometric setting. We have a whole family of Lyapunov orbits with different periods. Each Lyapunov orbit has several homoclinic orbits. We could combine several of these (as was done for instance in \cite{CapinskiGL14}) to obtain section-to-section maps having the desired changes in angles. We have chosen to work with a singe homoclinic orbit simply because it is easier to choose points and coordinate changes along a single orbit rather than several.
\correctiona{Our construction admits many generalizations. Here we have chosen to work with a single homoclinic orbit and to combine it with the inner dynamics. For instance, we could combine several homoclinic orbits instead of just one, as it  was done for instance in \cite{CapinskiGL14}}. %\marginpar{\color{red} 24. \hfill 24.}
\end{remark}

\begin{remark}
We have chosen $\mathbf{S}^{u}$ and $\mathbf{S}^{d}$ to be at particular
angles for the following reason. It turns out that for $\varepsilon >0$
following the homoclinic orbit from $\mathbf{S}^{u}$ to $\mathbf{S}^{u}$
yields a gain in energy. Following the homoclinic orbit from $\mathbf{S}^{d}$
to $\mathbf{S}^{d}$ yields a reduction in the energy. We establish this
rigorously in our computer assisted proof in Section \ref%
{sec:validation_connecting_sequences}. The change of energy happens as a
trajectory goes around the primary. This part of the excursion is the lowest
`wedge' from the plots in Figures \ref{fig:homoclinic_011} and \ref%
{fig:homoclinic_101}. If the tip is at $2k\pi -\pi /2$, $k\in \mathbb{Z}$,
then we are gaining energy; when it is at $2k\pi +\pi /2$, $k\in \mathbb{Z}$%
, then we are loosing energy. Traveling along the Lyapunov orbits does not
yield significant energy changes.
\end{remark}

\subsection{Construction of connecting sequences\label%
{sec:main-proof-connecting-sequences}}

\label{sec:connecting_sequences} Around the homoclinic points for PCRTBP
described in Section \ref{sec:choice_of_sections} we construct connecting
sequences. As the corresponding homoclinic orbit results from an
intersection of a $2$-dimensional unstable manifold and a $2$-dimensional
stable manifold of a Lyapunov orbit, there are well defined unstable and
stable directions along the homoclinic orbit. These are given by the
corresponding tangent vectors to the unstable and stable leaves from of the
foliations of the manifolds. In the sequel, we only need approximations of
the unstable/stable directions which are obtained via non-rigorous numerics.

We fix a sequence of points $q_{0}^{\ast },\ldots ,q_{n}^{\ast }$ along the
homoclinic orbit, and along these points we position sections $\Sigma
_{0},\ldots ,\Sigma _{n}$, as discussed in Section \ref%
{sec:main-proof-sections}; see Figure \ref{fig:LyapOrb}. We choose this
orbit in the extended phase space which includes $\theta $, by selecting $%
\pi _{\theta }q_{0}^{\ast }=0.$ Recall that $q_{0}^{\ast }\in \Sigma
=\{Y=0\} $ and also that we have chosen our points along the homoclinic so
that for $k\in \left\{ 50,58,66,74\right\} $ we have $q_{k}^{\ast }\in
\Sigma =\left\{ Y=0\right\} $.

Each connecting sequence will always start with a window on the section $%
\Sigma =\{Y=0\}$ and finish also with a window on $\Sigma =\{Y=0\}.$ At the
section $\Sigma =\{Y=0\}$ we use the local coordinate change $\Psi $ given
by (\ref{eq:Psi-change-at-strip}). In our connecting sequences we allow the
local coordinates $\Psi _{\ell ,i,\varepsilon }$ on the intermediate
sections $\Sigma _{i}$, for $i=1,\ldots ,k_{\ell }-1$, to be dependent on
the choice of $\ell, i, \varepsilon$, but we always start and finish at $%
\Sigma =\{Y=0\}$, and always in the same local coordinates given by $\Psi $.
The $A_{\ell ,i}$ and $w_{\ell ,i}$ for the coordinates $\Psi _{\ell
,i,\varepsilon }$ at the sections $\Sigma _{i}$ for $i=1,\ldots ,k_{\ell }-1$
are chosen so that we obtain (\ref{eq:derivative-form-of-local-map}) and (%
\ref{eq:choice-of-w}) along the connecting sequence for the local maps of
the form (\ref{eq:local-map-form}). (Our choices of $A_{\ell ,i}$ and $%
w_{\ell ,i}$ for such local coordinates are established numerically, by
using non-rigorous computations.)

\begin{figure}[tbp]
\begin{center}
\includegraphics[width=6cm]{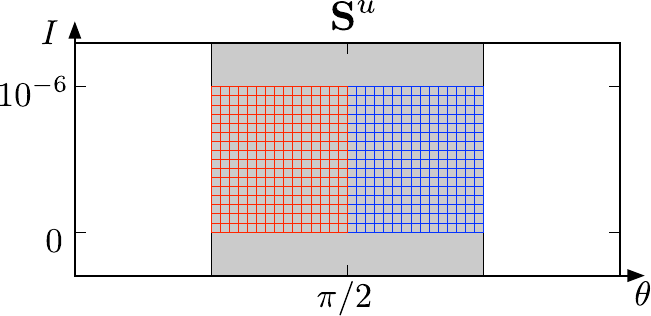}%
\includegraphics[width=6cm]{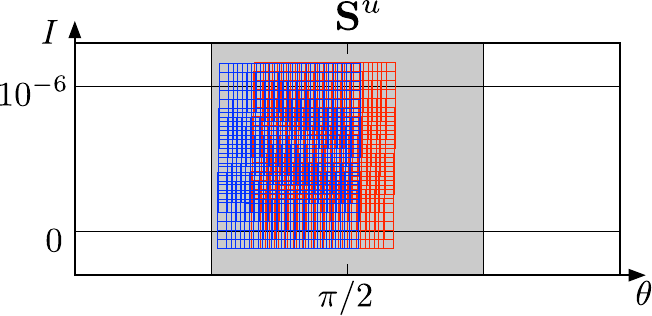}
\end{center}
\caption{Transition from $\mathbf{S}^{u}$ to $\mathbf{S}^{u}$. The strip $%
\mathbf{S}^{u}$ is in grey. On the left hand side plot we see the windows $%
N_{\ell ,0}$, for $\ell \in L_{1}^{uu}$ in blue and for $\ell \in L_{2}^{uu}$
in red. On the right hand side we have the windows $N_{\ell ,k_{\ell }}$.
For $\ell \in L_{1}^{uu}$ we take $k_{\ell }=n_{1}^{uu}=50$ and for $\ell
\in L_{2}^{uu}$ we take $k_{\ell }=n_{2}^{uu}=66$ as the length of the
connecting sequences.}
\label{fig:transition_aa}
\end{figure}

On the strips $\mathbf{S}^{u}$ and $\mathbf{S}^{d}$ at $\Sigma =\{Y=0\}$
(see Section \ref{sec:main-proof-sections}) we consider cones $%
Q_{a}^{\varepsilon }$ as in \eqref{eq:cone-Q}, with $a=\left(
a_{y},a_{I},a_{\theta }\right) $,
\begin{equation*}
a_{y}=10^{-3},\qquad a_{I}=1,\qquad a_{\theta }=10.
\end{equation*}

We construct connecting sequences from $\mathbf{S}^{u}$ to $\mathbf{S}^{u}$
by subdividing $\mathbf{S}^{u}\cap \{I\in \lbrack 0,10^{-6}]\}$ into $%
16\times 30=480$ windows $N_{\ell ,0}$, overlapping along the $I,\theta $
coordinates. Their projections onto $I,\theta $ are depicted on the left
hand side plot in Figure \ref{fig:transition_aa}. We ensure that the windows
overlap so that conditions \hyperlink{cond:C1.iii}{(C1.iii)} and \hyperlink{cond:C2.iii}{(C2.iii)} are fulfilled\footnote{%
Since $r=10^{-6}$, $\varepsilon _{0}=1.6\cdot 10^{-5}$ are very small, the
required overlap of size $\varepsilon _{0}a_{I}r$ and $a_{\theta }r$ (see Remarks \ref{rem:r-a-issue} and \ref{rem:r-a-issue-2})
%footnotes to conditions \hyperlink{cond:C1.iii}{(C1.iii)} 
%and \hyperlink{cond:C2.iii}{(C2.iii)}) 
along $I$ and $\theta $,
respectively, is so small that it is not visible on the plots at this
resolution.}. We take the windows $N_{\ell ,0}$ on the $x,y$ coordinates to
be $\left[ -r,r\right] \times \left[ -r,r\right] $, the same as for the
strip $\mathbf{S}^{u}$. We label the window with $\theta >\pi /2$ with a set
of labels denoted $L_{1}^{uu}$ and the windows with $\theta <\pi /2$ with a
set of labels denoted $L_{2}^{uu}$. For each $\ell \in L_{j}^{uu}$ we
consider a connecting sequence of length $n_{j}^{uu}$, for $j\in \left\{
1,2\right\} $ and construct a connecting sequences
\begin{equation}
(N_{\ell ,0},Q_{a}^{\varepsilon })\overset{f_{\ell ,1,\varepsilon }}{%
\Longrightarrow }\ldots \overset{f_{\ell ,n_{j}^{uu},\varepsilon }}{%
\Longrightarrow }(N_{\ell ,n_{j}^{uu}},Q_{a}^{\varepsilon }),\qquad \ell \in
L_{j}^{uu},\,j\in \left\{ 1,2\right\} .  \label{eqn:validation_connecting_uu}
\end{equation}
%\correction{The windows within each connecting sequence are independent from $\varepsilon$.}

\begin{figure}[tbp]
\begin{center}
\includegraphics[width=6cm]{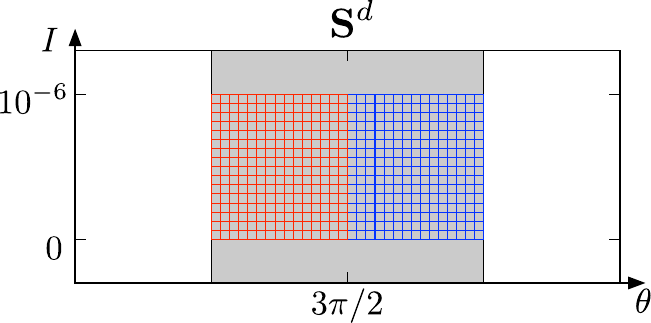}%
\includegraphics[width=6cm]{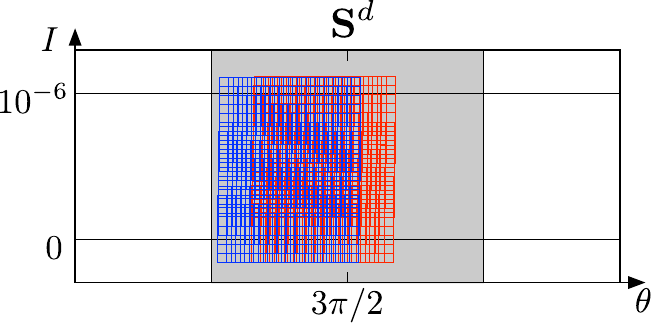}
\end{center}
\caption{Transition from $\mathbf{S}^{d}$ to $\mathbf{S}^{d}$. On the left
we have the windows $N_{\ell ,0}$, for $\ell \in L_{1}^{dd}$ in blue, and
for $\ell \in L_{2}^{dd}$ in red. On the right are the respective windows $%
N_{\ell ,k_{\ell }}$.}
\label{fig:transition_bb}
\end{figure}

In a similar way we construct connecting sequences from $\mathbf{S}^{d}$ to $%
\mathbf{S}^{d}$. These sequences are of the lengths $n_{1}^{dd}$ and $%
n_{2}^{dd},$ for $L_{1}^{dd}$ and $L_{2}^{dd}$, respectively; see Figure \ref%
{fig:transition_bb}.

We construct connecting sequences from $\mathbf{S}^{u}$ to $\mathbf{S}^{d}$
by subdividing $\mathbf{S}^{u}\cap \{I\in \lbrack 0,10^{-6}]\}$ into $%
16\times 4=64$ windows $N_{\ell ,0}$ as in Figure \ref{fig:transition_ab}.
(The reason why we construct fewer windows than for previous transitions
will be explained in Section \ref{sec:validation_connecting_sequences}). We
also construct connecting sequences from $\mathbf{S}^{d}$ to $\mathbf{S}^{u}$
in an analogous fashion; see Figure \ref{fig:transition_ba}.

\begin{figure}[tbp]
\begin{center}
\includegraphics[width=6cm]{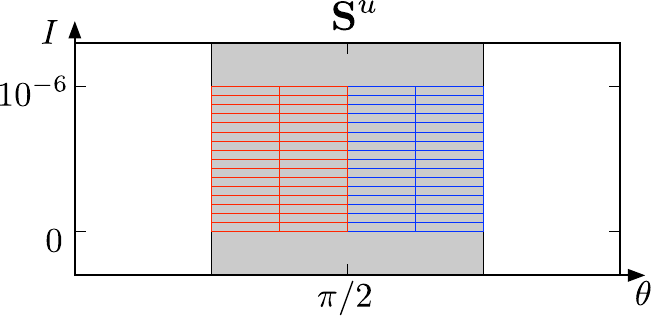}%
\includegraphics[width=6cm]{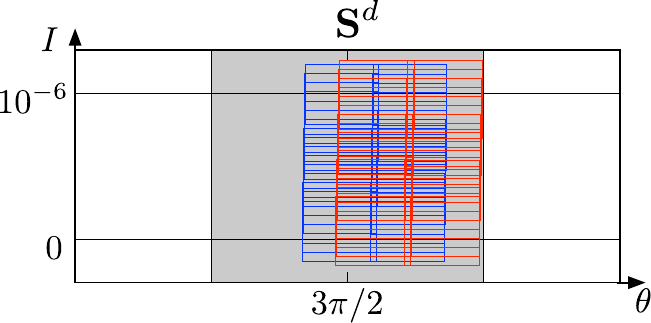}
\end{center}
\caption{Transition from $\mathbf{S}^{u}$ to $\mathbf{S}^{d}$. On the left
we have the windows $N_{\ell ,0}$, for $\ell \in L_{1}^{ud}$ in blue, and
for $\ell \in L_{2}^{ud}$ in red. On the right are the respective windows $%
N_{\ell ,k_{\ell }}$.}
\label{fig:transition_ab}
\end{figure}

\begin{figure}[tbp]
\begin{center}
\includegraphics[width=6cm]{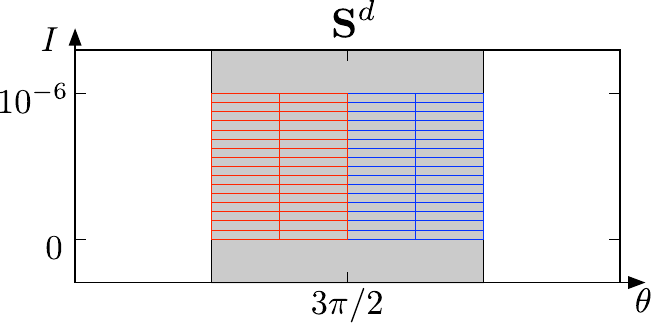}%
\includegraphics[width=6cm]{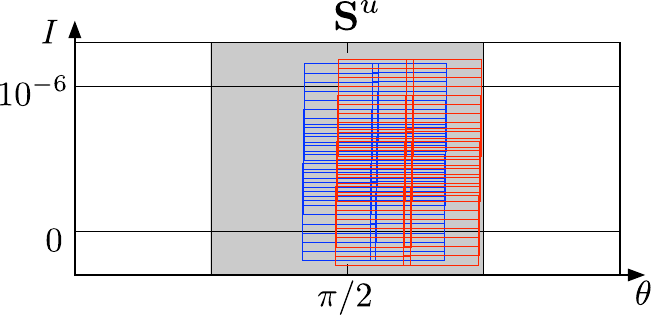}
\end{center}
\caption{Transition from $\mathbf{S}^{d}$ to $\mathbf{S}^{u}$. On the left
we have the windows $N_{\ell ,0}$, for $\ell \in L_{1}^{du}$ in blue, and
for $\ell \in L_{2}^{du}$ in red. On the right are the respective windows $%
N_{\ell ,k_{\ell }}$.}
\label{fig:transition_ba}
\end{figure}

\subsection{Proof of Theorem \protect\ref{th:main-3bp}}

\label{sec:validation_connecting_sequences}\label%
{sec:validation_energy_change} Here we start the rigorous,
interval-arithmetic validation of conditions \hyperlink{cond:C1}{{\bf C1}} and \hyperlink{cond:C2}{{\bf C2}} which lead to the
proof of the main theorem.

%\begin{remark}
Earlier sections provided a description of the choices of local maps \correctiona{(\ref{eq:local-map-form}) }we use
in our program. The key issue is that due to the form of the coordinate
changes we consider, regardless of the particular choices of the
coefficients for our local maps, we are sure that the change in $I$ reflects
precisely the change of the energy $H_{0}$. This is ensured by Theorem \ref%
{th:return-change-H} and the fact that our connecting sequences always start
and finish in the same section $\Sigma =\{Y=0\}$, with the same local
coordinates $\Psi $.
%\end{remark}

%\begin{remark}
In the computer assisted proof, for each map $f_{\ell ,i,\varepsilon }$
involved in conditions \hyperlink{cond:C1}{{\bf C1}} and \hyperlink{cond:C2}{{\bf C2}} our program validates the
correct alignment of windows, cone conditions, and establishes estimates on
the change in energy. If at any point one of the required conditions would
not be fulfilled, our program would stop and report an error.
%This means that if our choices of the coefficients for the local coordinates would be inadequate, then the consequence would be that the computer assisted proof would not go through and the result would not be established.
%\end{remark}

Using Lemmas \ref{lem:covering} and \ref{th:param-dep-cones}, we
validate the correct alignment and cone condition along each step of the
connecting sequences constructed in Section \ref{sec:connecting_sequences}.
We do this by
subdividing the parameter interval $[0,\varepsilon_0]$ into eight smaller subintervals of equal length,
and on each of them we perform the validation separately, as described in Remark \ref%
{rem:epsilon-subdivision}. 
The plots from Figure \ref{fig:transition_aa}, %
\ref{fig:transition_bb}, \ref{fig:transition_ab}, \ref{fig:transition_ba}
show in fact the results obtained through the computer assisted proof and
have been validated using rigorous interval arithmetic computations. Note
that the windows which pass through $n_{1}^{uu}$ sections shift in the angle
$\theta $ to the left (blue), and those passing through $n_{2}^{uu}$
sections (red) shift in the angle $\theta $ to the right, just as discussed
in Section \ref{sec:main-proof-sections}.

Thus we obtain a computer assisted proof of the following result.

\begin{theorem}
\label{th:3bp-i-iii}In the Neptune-Triton system, for every $\varepsilon \in
(0,\varepsilon _{0}]$, with $\varepsilon _{0}=1.6\cdot 10^{-5},$ for every $%
\kappa ,r\in \left\{ u,d\right\} $, there exist connecting sequences from $%
\mathbf{S}^{\kappa }$ to $\mathbf{S}^{r}$%
\begin{equation*}
(N_{\ell ,0},Q_{a}^{\varepsilon })\overset{f_{\ell ,1,\varepsilon }}{%
\Longrightarrow }\ldots \overset{f_{\ell ,n_{j}^{\kappa r},\varepsilon }}{%
\Longrightarrow }(N_{\ell ,n_{j}^{\kappa r}},Q_{a}^{\varepsilon }),\qquad
l\in L_{j}^{\kappa r},\,j\in \left\{ 1,2\right\} ,
\end{equation*}%
which fulfill the conditions \hyperlink{cond:C1.i}{(C1.i}--\hyperlink{cond:C1.iii}{C1.iii)} and \hyperlink{cond:C2.i}{(C2.i}--\hyperlink{cond:C2.iii}{C2.iii)}.
\end{theorem}

%Upon closer inspection of the right hand side plots from Figure \ref%{fig:transition_aa} we can see that the windows $N_{\ell ,k_{\ell }}$ have slightly higher $I$ than $N_{\ell ,0}$. Such plot does not prove though thatthe action is increased. To construct windows that are correctly aligned, in the center directions we increase the windows in the coordinates $I,\theta $%, but this does not guarantee that the action $I$ is increased along theconnecting sequence.
%(Note that on the plot there are windows $N_{k_{l}}^{l}$ which stick out below $I=0$.)A similar comment can be made about Figure \ref{fig:transition_bb}.\
%It is visible that $N_{k_{l}}^{l}$ have slightly lower $I$ than $N_{0}^{l}$, but the bounds do not prove that $I$ is reduced along the transition.
To estimate the change of $I$ along the connecting sequences, we need to
obtain estimates on the derivative of the return map with respect to the
parameter $\varepsilon \in (0,\varepsilon _{0}]$. 
For this we use Lemmas \ref{lem:c-bound-2}, \ref{lem:c-bound-3}, which together with computer assisted validation allows us to prove the following theorem.

%rem:epsilon-subdivision

\begin{theorem}
\label{th:3bp-C-bounds}Let
\begin{equation*}
c=\correctiona{0.00012},\qquad \text{and\qquad }C=\correctiona{0.0076}.
\end{equation*}%
In the Neptune-Triton system, for every $\varepsilon \in (0,\varepsilon
_{0}] $, with $\varepsilon _{0}=1.6\cdot 10^{-5},$ for every $z$ which
passes through the corresponding connecting sequence we have the following
bounds:
\begin{eqnarray*}
\left. c\varepsilon <\pi _{I}F_{\ell ,\varepsilon }(z)-\pi _{I}(z)\right.
&<&C\varepsilon \qquad \text{for }\ell \in L_{1}^{uu}\cup L_{2}^{uu}, \\
\left. c\varepsilon <\pi _{I}(z)-\pi _{I}F_{\ell ,\varepsilon }(z)\right.
&<&C\varepsilon \qquad \text{for }\ell \in L_{1}^{dd}\cup L_{2}^{dd}, \\
\left\vert \pi _{I}F_{\ell ,\varepsilon }(z)-\pi _{I}(z)\right\vert
&<&C\varepsilon \qquad \text{for }\ell \in L_{1}^{ud}\cup L_{2}^{ud}\cup
L_{1}^{du}\cup L_{2}^{du}.
\end{eqnarray*}
\end{theorem}

The main difficulty in the computer assisted validation of above theorem was
obtaining the lower bound $c$. This is the reason why we needed more
subdivisions of the strips for the transitions from $\mathbf{S}^{u}$ to $%
\mathbf{S}^{u}$ and from $\mathbf{S}^{d}$ to $\mathbf{S}^{d}$ than for $%
\mathbf{S}^{u}$ to $\mathbf{S}^{d}$ and $\mathbf{S}^{d}$ to $\mathbf{S}^{u}$%
. (The smaller the sets the sharper the interval arithmetic bounds.)

%We are ready to prove our main result.
We are ready to prove the main theorem.

\begin{proof}[Proof of Theorem \protect\ref{th:main-3bp}]
The{\ initial level of the energy }$h_{0}$ is the energy of the Lyapunov
orbit of the PCR3BP, which passes through $q^{\ast }$ from (\ref%
{eq:q-star-3bp}),%
\begin{equation*}
h_{0}=H_{0}\left( q^{\ast }\right) =-1.5050906397016.
\end{equation*}%
By Theorems \ref{th:3bp-i-iii}, \ref{th:3bp-C-bounds} conditions \hyperlink{cond:C1}{\textbf{C1 }}%
and \hyperlink{cond:C2}{\textbf{C2}} are satisfied, which by the Theorems \ref{thm:diffusion}, %
\ref{th:symbolic-dynamics-from-covering}, and \ref{th:diffusion-process}
gives (1), (2) and (4), respectively.

We will now show that \hyperlink{cond:C3}{\textbf{C3}} is satisfied. {Let }$\ell _{1},\ell
_{2}\in L^{uu}$ and $\ell _{1}^{\prime }\in L^{ud},$ $\ell _{2}^{\prime }\in
L^{du}$. Trajectories starting from points from the domain of $F_{\ell
_{2},\varepsilon }\circ F_{\ell _{1},\varepsilon }$ make different numbers
of turns around the Lyapunov orbit between the homoclinic excursions than
trajectories starting from the domain of $F_{\ell _{2}^{\prime },\varepsilon
}\circ F_{\ell _{1}^{\prime },\varepsilon }$ (see Figures \ref%
{fig:homoclinic_011}, \ref{fig:homoclinic_101}), hence the domains are
disjoint. The same argument can be applied when $\ell _{1},\ell _{2}\in
L^{dd}$ and $\ell _{1}^{\prime }\in L^{du},$ $\ell _{2}^{\prime }\in L^{ud}$%
, which means that condition \hyperlink{cond:C3}{\textbf{C3}} is fulfilled. By Theorem \ref%
{th:Hausdorff-dim} the Hausdorff dimension of the set of initial points for
orbits that shadow a given energy sequence is strictly {\correctiona{greater}} than 3 in the
four dimensional section $\{Y=0\}$. Due to the smooth dependence of
solutions on initial conditions, we know that cone conditions, covering
relations and energy change bounds will also hold for connecting sequences
which start and finish at sections $\{Y=\mathcal{Y}\}$, for all $|\mathcal{Y}%
|<\mathcal{Y}_{0}$, for sufficiently small $\mathcal{Y}_{0}>0$. \correctiona{This means
that in the full phase space the Hausdorff dimension of the sets of orbits that shadow an
energy sequence is strictly {\correctiona{greater than 4.}}}%\marginpar{\color{red}25. \hfill 25.}

Now we verify the choice of $T$ and $\eta $ from Theorem \ref{th:main-3bp}.
From Theorem \ref{thm:diffusion} we know that to diffuse the distance of 1
along $I$ we need no more than $1/\left( c\varepsilon \right) $ transitions
between strips. Here we diffuse over a shorter distance $C_{h_{0}}=10^{-6}$,
so the number of needed transitions is $C_{h_{0}}/\left( c\varepsilon
\right) $. Each transition takes less time than $(9+\frac{1}{10})\pi $ (see
Figures \ref{fig:homoclinic_011}, \ref{fig:homoclinic_101}). This means that
the time needed to diffuse over the distance $C_{h_{0}}$ is less than%
\begin{equation*}
\frac{(9+\frac{1}{10})\pi C_{h_{0}}}{c\varepsilon }<
\correctiona{\frac{1}{4\varepsilon}}=\frac{1}{\varepsilon }T.
\end{equation*}%
From Theorem \ref{th:symbolic-dynamics-from-covering} 
%(see also the footnote) 
\correctiona{(see also Remark \ref{rem:eta-r}) }we know that we need
\begin{equation*}
\eta \geq 2ra_{I}+C=2\cdot 10^{-6}\correctiona{\cdot 1}+\correctiona{0.0076}.
\end{equation*}%
Clearly $\eta =10^{-2}$ satisfies this condition.
\end{proof}

The total CPU time of the computer assisted proof was \correctiona{186406 }seconds. \correctiona{The
proof was conducted using a parallel computation on eight threads of an iMac
desktop computer with four 3.1 GHz Intel i7 cores. Running it once takes under six and a half hours.}
%\correctiona{The
%proof was conducted using a parallel computation on eight threads of an iMac
%desktop computer with four 3.1 GHz Intel i7 cores. Running it once takes under 6
%hours and 10 minutes.}\marginpar{\color{red}26. \hfill 26.}

%TCIDATA{Version=5.00.0.2606}
%TCIDATA{LaTeXparent=0,0,MMFedit.tex}

\section{Propagation of discs}

In this section we introduce the notion of horizontal discs which satisfy 
$Q^{\varepsilon }$-cone condition. Such discs will be the main building
blocks used for the proofs of Theorems \ref{thm:diffusion}, \ref%
{th:symbolic-dynamics-from-covering}, \ref{th:Hausdorff-dim} and \ref%
{th:diffusion-process}.

We stress that discs are not directly used in the computer assisted proof of
Theorem \ref{th:main-3bp}. That proof relies on the construction of
connecting sequences.

\begin{definition}
\label{def:horizontal-disc} A continuous map $h:\bar{B}%
^{n_{u}}\rightarrow \mathbb{R}^{n_{u}}\times \mathbb{R}^{n_{s}+2}$ that
satisfies a graph condition that $h(x)=(x,\pi _{y ,I,\theta}h(x))$, for all $%
x\in \bar{B}^{n_{u}}$ will be referred to as a  horizontal disc. If $h(\bar{B}^{n_{u}})\subseteq N$ then $h$ is said to be a horizontal disc in $N$.
\end{definition}

Recall that in (\ref{eq:cone-Q}) we have considered $Q^{\varepsilon }:
\mathbb{R}^{n_u}\times \mathbb{R}^{n_s}\times \mathbb{R}\times \mathbb{R}
\rightarrow \mathbb{R}$ which define cones, and that we have assumed that $%
Q^{\varepsilon}$ satisfy condition (\ref{eq:cone-slopes-assumption}).

\begin{definition}
\label{def:disc-cone-cond} Given a cone $Q^{\varepsilon }:\mathbb{R}%
^{n_{u}}\times \mathbb{R}^{n_{s}+2}\rightarrow \mathbb{R}$, a horizontal disc $h:\bar{B}%
^{n_{u}}\rightarrow \mathbb{R}^{n_{u}}\times \mathbb{R}^{n_{s}+2}$ is said
to satisfy a $Q^{\varepsilon }$-cone condition if $Q^{\varepsilon
}(h(x)-h(x^{\prime }))>0$, for all $x,x^{\prime }\in \bar{B}^{n_{u}}$ with $%
x\neq x^{\prime }$. We shall refer to such $h$ as a $Q^{\varepsilon }$-disc.
\end{definition}

The image $h(\bar{B}^{n_{u}})$ of a $Q^{\varepsilon }$-disc is a $n_{u}$%
-dimensional topological disc. In the sequel we will use the same notation $h
$ for the image $h(\bar{B}^{n_{u}})$ of the map $h$ as for the map itself.
From now in the case of a $Q^{\varepsilon }$-disc, we shall add a
superscript and write $h^{\varepsilon }$.

\begin{remark}
Definition \ref{def:horizontal-disc} is an analogue of the definition of a
`horizontal disc' from \cite{Zcc}. It is slightly simplified compared to
\cite{Zcc}. A horizontal disc in the sense of \cite{Zcc}, which satisfies
the cone condition, is a $Q^{\varepsilon }$-disc in the sense of Definition %
\ref{def:horizontal-disc}; and vice versa. (See \cite[Lemma 5]{Zcc}.)
\end{remark}

\begin{remark}
\label{rem:discs-slim} An important feature is that the length of a $%
Q^{\varepsilon }$-disc $h^{\varepsilon }$ is bounded along the $\theta ,I$
directions, with the bound in $I$ of order $O(\varepsilon )$. This is
because by (\ref{eq:cone-slopes-assumption}) for any $x,x^{\prime }\in \bar{B%
}^{n_{u}}$ from $Q^{\varepsilon }(h^{\varepsilon }(x)-h^{\varepsilon
}(x^{\prime }))>0$ it follows that
\begin{eqnarray*}
\left\Vert \pi _{\theta }\left( h^{\varepsilon }(x)-h^{\varepsilon
}(x^{\prime })\right) \right\Vert  &\leq &a_{\theta }\left\Vert \pi
_{x}\left( x-x^{\prime }\right) \right\Vert \leq 2a_{\theta }, \\
\left\Vert \pi _{I}\left( h^{\varepsilon }(x)-h^{\varepsilon }(x^{\prime
})\right) \right\Vert  &\leq &\varepsilon a_{I}\left\Vert \pi _{x}\left(
x-x^{\prime }\right) \right\Vert \leq 2\varepsilon a_{I}.
\end{eqnarray*}
\end{remark}

Let us now turn to our IFS (\ref{eq:IFS}) from Section \ref{sec:IFS}. Let us
assume that for this IFS we have a finite collection of connecting sequences
indexed by $L$ as in Section \ref{sec:connecting-sequences}.

\begin{theorem}
\cite[Theorem 7]{Zcc}\label{th:hor-disc-propagation} Given $\ell \in L$, for
every connecting sequence (\ref{eq:connecting-sequence}) and for every $%
Q^{\varepsilon }$-disc $h^{\varepsilon }$ in $N_{\ell ,0}$ there exists a
topological disc $\mathscr{D}\subset h^{\varepsilon }$ such that for all $%
z\in \mathscr{D}$ holds
\begin{equation*}
(f_{\ell ,m,\varepsilon }\circ f_{\ell ,m-1,\varepsilon }\circ \ldots \circ
f_{\ell ,1,\varepsilon })\left( z\right) \in N_{\ell ,m},\qquad \text{for }%
m=1,\ldots ,k_{\ell }-1,
\end{equation*}%
and $F_{\ell ,\varepsilon }\left( \mathscr{D}\right) =(f_{\ell ,k_{\ell
},\varepsilon }\circ f_{\ell ,k_{\ell }-1,\varepsilon }\circ \ldots \circ
f_{\ell ,1,\varepsilon })\left( \mathscr{D}\right) $ is a graph of some $%
Q^{\varepsilon }$-disc $\tilde{h}^{\varepsilon }$ in $N_{\ell ,k_{\ell }}$;
i.e. $\tilde{h}^{\varepsilon }(\bar{B}^{n_{u}})=F_{\ell ,\varepsilon }\left( %
\mathscr{D}\right) .$
\end{theorem}

\begin{comment}
\begin{remark}
In \cite[Theorem 7]{Zcc} instead of (\ref{eq:cc}), which we use to define
our cone condition, it is required that
\begin{equation}
Q_{2}^{\varepsilon }\left( f\left( z\right) -f(z^{\prime })\right)
>Q_{1}^{\varepsilon }\left( z-z^{\prime }\right) .  \label{eq:cc-orig}
\end{equation}
Theorem \ref{th:hor-disc-propagation} is true under (\ref{eq:cc}). The
result follows by simply replacing (\ref{eq:cc-orig}) with (\ref{eq:cc}) in
the proof \cite[Theorem 7]{Zcc}. The condition (\ref{eq:cc-orig}) is
stronger than (\ref{eq:cc}) and in \cite{Zcc} it is used for proofs of
invariant manifolds for hyperbolic fixed points. For Theorem \ref
{th:hor-disc-propagation}, (\ref{eq:cc}) is sufficient.
\end{remark}
\end{comment}

When applying Theorem \ref{th:hor-disc-propagation} we shall say that the $%
Q^{\varepsilon }$-disc $h^{\varepsilon }$ is propagated (or moved) through
the connecting sequence to become the $Q^{\varepsilon }$-disc $\tilde{h}%
^{\varepsilon }$.\marginpar{\color{red} 28.}
\correctiona{\begin{remark}From Remark \ref{rem:discs-slim} we know that horizontal discs are `slim' along the central coordinates. 
 The sizes of successive windows in a connecting sequence can grow along central coordinates. 
 After propagating a horizontal disc though through such sequence we obtain a horizontal disc, which remains `slim' in the central coordinates. 
 We can then enclose it in another window, which is again `slim' in the central coordinates, and propagate once more through another connecting sequence. We can repeat such procedure infinitely many times.
The fact that the propagated horizontal discs do not grow along the central coordinates is the key feature that allows to topologically control the central coordinate and obtain infinite orbits. For this, we require both correct alignment of windows as well as cone conditions. The latter ensure that the discs remain `slim'.
\end{remark} }

%TCIDATA{Version=5.00.0.2606}
%TCIDATA{LaTeXparent=0,0,MMFedit.tex}

\section{Existence of diffusing orbits\label{sec:diffusion}}

In this section we will give the proof of Theorem \ref{thm:diffusion}. We
first give a version of this theorem in terms of $Q^{\varepsilon }$-discs --
Theorem \ref{th:disk_iterate} -- from which we derive Theorem \ref%
{thm:diffusion}.

\subsection{Diffusion under conditions on $Q^{\protect\varepsilon }$-disks}

Consider an iterated function system (\ref{eq:IFS}) as described in Section %
\ref{sec:IFS}, and the strip $\mathbf{S}^{u}\subset \Sigma _{0}$ defined in (%
\ref{eq:strips-Sa-Sb}). We formulate the following condition.

\textbf{Condition \hypertarget{cond:A1}{A1}.} There exists a constant $0<c$, such that for every $%
\varepsilon \in (0,\varepsilon _{0}]$ and every $Q^{\varepsilon }$-disc $%
h^{\varepsilon }\subseteq \mathbf{S}^{u}$, the following hold:

\begin{description}
\item[(A1.1)] \hypertarget{cond:A1.1}{}There exist $F_{\ell ,\varepsilon }$ in $\mathscr{F}%
_{\varepsilon }$ and a $Q^{\varepsilon }$-disc $\tilde{h}^{\varepsilon }:%
\bar{B}^{n_{u}}\rightarrow \mathbf{S}^{u}$ satisfying
\begin{equation*}
\tilde{h}^{\varepsilon }\subseteq F_{\ell ,\varepsilon }(h^{\varepsilon });
\end{equation*}

\item[(A1.2)] \hypertarget{cond:A1.2}{}For each $z\in h^{\varepsilon }$ and $\tilde{z}%
=F_{\ell,\varepsilon }(z)\in \tilde{h}^{\varepsilon }$, we have that
\begin{equation}
\pi _{I}(\tilde{z})-\pi _{I}(z)>c\varepsilon .  \label{eqn:easy2}
\end{equation}
\end{description}

\begin{theorem}
\label{th:disk_iterate} Assume that the IFS $\mathscr{F}_{\varepsilon }$
satisfies Condition \hyperlink{cond:A1}{\bf A1}.

Then, for every $M>0$, and every $\varepsilon \in (0,\varepsilon _{0}]$,
there exist $z\in \mathbf{S}^{u}$, $m\le M/(c\varepsilon)$, and a sequence
of functions $F_{\ell _{1},\varepsilon },\ldots ,F_{\ell _{m},\varepsilon }$
in the IFS $\mathscr{F}_{\varepsilon }$, such that $\tilde{z}=(F_{\ell
_{m},\varepsilon }\circ \ldots \circ F_{\ell _{1}\varepsilon })(z)$
satisfies
\begin{equation}
\pi _{I}(\tilde{z})-\pi _{I}(z)>M.  \label{eqn:easy3}
\end{equation}
\end{theorem}

\begin{proof}
Fix $\varepsilon \in (0,\varepsilon _{0}]$.

We start with any $Q^{\varepsilon }$-disk $h_{0}^{\varepsilon }:\bar{B}%
^{n_{u}}\rightarrow \mathbf{S}^{u}$. (We can let, for instance, $%
h_{0}^{\varepsilon }\left( x\right) =\left( x,y_{0},I_{0},\theta _{0}\right)
$ for some $y_{0}\in \bar{B}^{n_{s}},I_{0}\in \mathbb{R}$ and $\theta
_{0}\in S_{\theta }^{u}$.) By our assumptions, there exists a $%
Q^{\varepsilon }$-disk $h_{1}^{\varepsilon }\subseteq F_{\ell
_{1},\varepsilon }(h_{0}^{\varepsilon })$ as in \hyperlink{cond:A1.1}{(A1.1)}. The set of points $%
z_{0}\in h_{0}^{\varepsilon }$ for which $z_{1}=F_{\ell _{1},\varepsilon
}(h_{0}^{\varepsilon }(z_{0}))\in h_{1}^{\varepsilon }$ forms a $n_{u}$%
-dimensional topological disk $\mathscr{D}_{0}^{\varepsilon }\subseteq
h_{0}^{\varepsilon }$, which is the image under $h_{0}^{\varepsilon }$ of a $%
n_{u}$-dimensional disk in $\bar{B}^{n_{u}}$. By \hyperlink{cond:A1.2}{(A1.2)} we have $\pi
_{I}(z_{1})-\pi _{I}(z_{0})>c\varepsilon $.

Inductively, at the $n$-th step, there exist a topological disc $\mathscr{D}%
_{n-1}^{\varepsilon }\subseteq h_{0}^{\varepsilon }$, a $Q^{\varepsilon }$%
-disk $h_{n}^{\varepsilon }$, and some functions $F_{\ell _{1},\varepsilon
},\ldots ,F_{\ell _{n},\varepsilon }$ in the IFS $\mathscr{F}_{\varepsilon }$
such that $z_{n}=(F_{\ell _{n},\varepsilon }\circ \ldots \circ F_{\ell
_{1},\varepsilon })(z_{0})\in h_{n}^{\varepsilon }$ for every $z_{0}\in %
\mathscr{D}_{n-1}^{\varepsilon }$, with the disc $\mathscr{D}%
_{n-1}^{\varepsilon }$ satisfying
\begin{equation*}
\mathscr{D}_{n-1}^{\varepsilon }\subseteq \mathscr{D}_{n-2}^{\varepsilon
}\subseteq \ldots \subseteq \mathscr{D}_{0}^{\varepsilon },
\end{equation*}%
where $\mathscr{D}_{n-2}^{\varepsilon },\ldots ,\mathscr{D}_{0}^{\varepsilon
}$ are the discs constructed at the previous steps. Also, $\pi
_{I}(z_{n})-\pi _{I}(z_{0})>n\varepsilon \cdot c$.

For the induction step, there exists a $Q^{\varepsilon }$-disc
\begin{equation*}
h_{n+1}^{\varepsilon }\subseteq F_{\ell _{n+1},\varepsilon
}(h_{n}^{\varepsilon }).
\end{equation*}%
The set of points $z_{n+1}=F_{\ell _{n+1},\varepsilon }(z_{n})\in
h_{n+1}^{\varepsilon }$ with $z_{n}\in h_{n}^{\varepsilon }$ of the form $%
z_{n}=(F_{\ell _{n},\varepsilon }\circ \ldots \circ F_{\ell _{1},\varepsilon
})(z_{0})$ corresponds to a set of points $z_{0}\in \mathscr{D}%
_{n}^{\varepsilon }\subseteq \mathscr{D}_{n-1}^{\varepsilon }$. We also have
$\pi _{I}(z_{n+1})-\pi _{I}(z_{n})>c\varepsilon $. We obtain that
\begin{equation*}
z_{n+1}=F_{\ell _{n+1},\varepsilon }\circ F_{\ell _{n+1},\varepsilon }\circ
\ldots \circ F_{\ell _{1},\varepsilon }(z_{0})
\end{equation*}%
for $z_{0}\in \mathscr{D}_{n}^{\varepsilon }$, and $\pi _{I}(z_{n+1})-\pi
_{I}(z_{0})>c(n+1)\varepsilon $.

Repeating the iterative procedure for $m$-steps, with $m\le M/(c\eps)$,
yields an orbit along which the change in the action $I$ is more than $M$,
which concludes the proof of the theorem.

For future reference, note that the points $z_{0}\in h_{0}^{\varepsilon }$,
for which $z_{m}=(F_{\ell_{m},\varepsilon }\circ \ldots \circ
F_{\ell_{1}})(z)$ is as in the statement of the theorem, form a disc $%
\mathscr{D}_{m-1}^{\varepsilon }$, and we have $\mathscr{D}
_{m-1}^{\varepsilon }\subseteq \mathscr{D}_{m-2}^{\varepsilon }\subseteq
\ldots \subseteq \mathscr{D}_{0}^{\varepsilon }$.
\end{proof}

%\begin{remark} \label{rem:diffusion-from-discs-multi-action}Theorem \ref{th:disk_iterate} can be generalized to a setting where action angle coordinates are multidimensional $\left( I,\theta \right) =\left( I_{1},\ldots ,I_{n_{c}},\theta _{1},\ldots ,\theta _{n_{c}}\right) $. This can be done as follows. We define the strip $\mathbf{S}^{u}=\bar{B}^{n_{u}}\times \bar{B} ^{n_{s}}\times \mathbb{R}^{n_{c}}\times S_{\theta }^{u}$ with $S_{\theta}^{u}=\Pi _{i=1}^{n_{c}}[\theta _{i}^{u,1},\theta _{i}^{u,2}]$. Under the conditions as in Theorem \ref{th:disk_iterate}, if (\ref{eqn:easy2}) holds for one of the action variable $I_{i_{\ast }}$, for $i_{\ast }\in \{1,\ldots ,d\}$, i.e., $\pi _{I_{i_{\ast }}}(\tilde{z})-\pi _{I_{i_{\ast }}}(z)>c\varepsilon $, then we obtain diffusion along~$I_{i_{\ast }}$.\end{remark}

\subsection{Diffusion under conditions on connecting sequences}

In this section we give the proof of Theorem \ref{thm:diffusion} and make a
remark how it can be generalized to the setting when action angle
coordinates are higher dimensional.

\begin{proof}[Proof of Theorem \protect\ref{thm:diffusion}]
The result follows from an analogous construction to the proof of Theorem %
\ref{th:disk_iterate}, which is based on propagating $Q^{\varepsilon }$%
-discs from $\mathbf{S}^{u}$ to $\mathbf{S}^{u}$. There are some technical
differences though that we will point out below. In this new construction,
we choose the initial $Q^{\varepsilon }$-disc $h_{0}^{\varepsilon }$ to be $%
h_{0}^{\varepsilon }:\bar{B}^{n_{u}}\rightarrow \mathbf{S}^{u}$ (see proof
of Theorem \ref{th:disk_iterate}), defined as $h_{0}^{\varepsilon }\left(
x\right) :=\left( x,y_{0},\theta _{0},I_{0}\right) $, for some (arbitrary)
fixed $y_{0}\in \bar{B}^{n_{s}}$, $\theta _{0}\in S_{\theta }^{u}$ and for $%
I_{0}=0$.

By Theorem \ref{th:hor-disc-propagation}, assumption \hyperlink{cond:A1.1}{(A1.1)} from Theorem \ref%
{th:disk_iterate} follows from \hyperlink{cond:C1.i}{(C1.i}--\hyperlink{cond:C1.iii}{C1.iii)}. In more detail, by Remark \ref%
{rem:discs-slim} any $Q^{\varepsilon }$-disc $h^{\varepsilon }$ in $\mathbf{S%
}^{u}\cap \{I\in \left[ 0,1\right] \}$ is contained in a multi-dimensional
rectangle $\bar{B}^{n_{u}}\times \bar{B}^{n_{s}}\times \bar{B}\left( I^{\ast
},\varepsilon _{0}a_{I}\right) \cap \left[ 0,1\right] \times \bar{B}\left(
\theta ^{\ast },a_{\theta }\right) \cap \lbrack \theta ^{a,1},\theta ^{a,2}]$
for $I^{\ast }=\pi _{I}h^{\varepsilon }(0)$ and $\theta ^{\ast }=\pi
_{\theta }h^{\varepsilon }(0)$. Condition \hyperlink{cond:C1.ii}{(C1.ii)} combined with \hyperlink{cond:C1.iii}{(C1.iii)}
implies that the multi dimensional rectangle is in some set $N_{\ell ,0}$.
(In case that $(I^{\ast },\theta ^{\ast })\in \pi _{I,\theta }N_{\ell ,0}$
for just one $\ell \in L$ we use \hyperlink{cond:C1.iii}{(C1.iii)} with $\ell =\ell ^{\prime }$.)
This ensures that we can enclose $h^{\varepsilon }$ in some set $N_{\ell ,0}$, and use Theorem \ref{th:hor-disc-propagation} to obtain $\tilde{h}%
^{\varepsilon }=F_{l,\varepsilon }\left( \mathscr{D}\right) .$

We note that by \hyperlink{cond:C1.ii}{(C1.ii)} we can only ensure that $\tilde{h}^{\varepsilon }$
is in $N_{\ell ,k_{\ell }}\subset \mathbf{S}^{u}$; we can not ensure that $%
\pi _{I}\tilde{h}^{\varepsilon }\subset \left[ 0,1\right] $. This means that
it is possible that after repeatedly propagating the $Q^{\varepsilon }$-disk
$h^{\varepsilon }$ the resulting $Q^{\varepsilon }$-disc $\tilde{h}%
^{\varepsilon }$ is no longer contained in $\mathbf{S}^{u}\cap \{I\in \left[
0,1\right] \}$. In that case, we might not be able to keep propagating $%
\tilde{h}^{\varepsilon }$ further along another connecting sequence. Having $%
\tilde{h}^{\varepsilon }\not\subset \mathbf{S}^{u}\cap \{I\in \left[ 0,1%
\right] \}$ implies that, for some $z\in \tilde{h}^{\varepsilon }$ we have $%
\pi _{I}z>1$, in which case (\ref{eqn:diffusionC}) is obtained, and we do
not need to propagate further. In short, for any $Q^{\varepsilon }$-disc $%
h^{\varepsilon }$ in $\mathbf{S}^{u}\cap \{I\in \left[ 0,1\right] \}$ we can
either obtain the $Q^{\varepsilon }$-disc $\tilde{h}^{\varepsilon }$ in $%
\mathbf{S}^{u}\cap \{I\in \left[ 0,1\right] \}$, as in condition \hyperlink{cond:A1.1}{(A1.1)} from
Theorem \ref{th:disk_iterate}, which enables us to continue with another
step of the construction, or the $Q^{\varepsilon }$-disc $\tilde{h}%
^{\varepsilon }$ contains a point outside of $\{I\leq 1\}$, implying that we
have already achieved the required change in $I$ and we do not need to
propagate the $Q^{\varepsilon }$-disc any more.

Condition \hyperlink{cond:C1.iv}{(C1.iv)} is equivalent to condition \hyperlink{cond:A1.2}{(A1.2)} from Theorem \ref%
{th:disk_iterate}. It implies that the $Q^{\varepsilon }$-discs which are
propagated from $\mathbf{S}^{u}$ to $\mathbf{S}^{u}$ by our construction
must exit $\{I\leq 1\}$ after a finite number of iterates of the maps.

Now we will show how we obtain the bound on the Lebesgue measure for orbits
which diffuse the distance $\frac{1}{3}$ in $I$. Take a $Q^{\varepsilon }$%
-disc $h_{0}^{\varepsilon }(x)=\left( x,y_{0},I_{0},\theta _{0}\right) $
with $y_{0}\in \bar{B}^{n_{s}}$, $I_{0}\in \left[ 0,1/3\right] $, $\theta
\in S_{\theta }^{u}$. By the construction from our proof, we know that there
exists an $m,$ a sequence $\ell _{1},\ldots ,\ell _{m}\in L,$ a topological
disc $\mathscr{D}\subset h_{0}^{\varepsilon },$ and a $Q^{\varepsilon }$%
-disc $\tilde{h}^{\varepsilon }=F_{\ell _{m},\varepsilon }\circ \ldots \circ
F_{\ell _{1},\varepsilon }\left( \mathscr{D}\right) $ such that $\pi _{I}%
\tilde{h}^{\varepsilon }>2/3.$ Since as we propagate $h_{0}^{\varepsilon }$
with each transition we gain more than $\varepsilon c$ in $I$, we see that $%
m\leq 2/\left( 3\varepsilon c\right) $.

Let us now take two points $z_{0},z_{1}\in \mathscr{D}$ such that $F_{\ell
_{m},\varepsilon }\circ \ldots \circ F_{\ell _{1},\varepsilon }\left(
z_{1}\right) =\tilde{z}\in \tilde{h}^{\varepsilon }\left( \partial
B^{n_{u}}\right) ,$ and $F_{\ell _{m},\varepsilon }\circ \ldots \circ
F_{\ell _{1},\varepsilon }\left( z_{0}\right) =\tilde{h}^{\varepsilon
}\left( 0\right) $. From (\ref{eq:expansion-bound}), by the mean value
theorem,%
\begin{eqnarray*}
1 &=&\Vert \pi _{x}(\tilde{z}-\tilde{h}^{\varepsilon }\left( 0\right) )\Vert
\\
&=&\left\Vert \pi _{x}\left( F_{\ell _{m},\varepsilon }\circ \ldots \circ
F_{\ell _{1},\varepsilon }\left( z_{1}\right) -F_{\ell _{m},\varepsilon
}\circ \ldots \circ F_{\ell _{1},\varepsilon }\left( z_{0}\right) \right)
\right\Vert \\
&\leq &\alpha ^{m}\left\Vert z_{1}-z_{0}\right\Vert ,
\end{eqnarray*}%
which means that $B^{n_{u}}(z_{0},\alpha ^{-m})\subset \pi _{x}\mathscr{D}%
\subset h_{0}^{\varepsilon }.$ This holds for any choice of $y_{0}\in \bar{B}%
^{n_{s}}$, $I_{0}\in \left[ 0,1/3\right] $, $\theta \in S_{\theta }^{u}$,
which since $\alpha ^{-m}\leq \alpha ^{-2/\left( 3\varepsilon c\right) }$
implies the claim; the $n_{u}$ in the bound $\alpha ^{-2n_{u}/\left(
3\varepsilon c\right) }$ in the statement of the theorem comes from the fact
that for any $r>0$ we have $\mu \left( B^{n_{u}}(z_{0},r)\right)
=r^{n_{u}}\mu \left( B^{n_{u}}\right) $.
\end{proof}

\correctiona{\begin{remark}
Theorems \ref{th:disk_iterate}  and \ref{thm:diffusion} can be generalized
to a setting where the action angle coordinates are higher dimensional
$\left( I_{1},\ldots ,I_{n_{c}},\theta _{1},\ldots,\theta _{n_{c}}\right)$.

To obtain a higher dimensional version of Theorem \ref{th:disk_iterate} we can proceed as follows. We define the strip $\mathbf{S}^{u}=\bar{B}^{n_{u}}\times \bar{B} ^{n_{s}}\times \mathbb{R}^{n_{c}}\times S_{\theta }^{u}$ with $S_{\theta}^{u}=\Pi _{i=1}^{n_{c}}[\theta _{i}^{u,1},\theta _{i}^{u,2}]$. Under the conditions as in Theorem \ref{th:disk_iterate}, if (\ref{eqn:easy2}) holds for one of the action variable $I_{i_{\ast }}$, for $i_{\ast }\in \{1,\ldots ,d\}$, i.e., $\pi _{I_{i_{\ast }}}(\tilde{z})-\pi _{I_{i_{\ast }}}(z)>c\varepsilon $, then we obtain diffusion along~$I_{i_{\ast }}$.

For Theorem \ref{thm:diffusion},  to obtain diffusion
of length $\frac{1}{2}$ in one of the actions, we can proceed as  follows.
We define the strip $\mathbf{S}^{u}$ as above, and consider $I\in \left[ 0,1\right]
^{n_{c}}$. We need the condition \hyperlink{cond:C1}{\textbf{C1}}, but it is enough that (\ref%
{eqn:C1_iv}) holds for one action $I_{i_{\ast }}$ that is, $\pi _{I_{i_{\ast
}}}F_{l,\varepsilon }(z)-\pi _{I_{i_{\ast }}}(z)>c\varepsilon $. We can then
repeat the construction from the proof of Theorem \ref{thm:diffusion} taking
$h^{\varepsilon }\left( x\right) =\left( x,y_{0},I_{0},\theta _{0}\right) $
for arbitrary $y_{0}\in \bar{B}^{n_{s}},$ $I_{0}=\left( \frac{1}{2},\ldots ,
\frac{1}{2}\right) $ and any $\theta _{0}\in S_{\theta }^{u}$. As we
propagate $h^{\varepsilon }$ through successive connecting sequences we
increase in $I_{i_{\ast }}$. There is a possibility that as we propagate the
discs some of them will leave $\{I\in \left[ 0,1\right] ^{n_{c}}\}$ along
some other action than $I_{i_{\ast }}$. Nevertheless, for this to happen we
also achieve a change of order $O(1)$ in the $I_{i_{\ast }}$-direction. In
such case we obtain diffusion in the $I_{i_{\ast }}$-direction as well as in
some other action direction.
\end{remark}}

%\begin{remark} Theorem \ref{thm:diffusion} can be generalized to a setting where action angle coordinates are multidimensional $\left( I_{1},\ldots ,I_{n_{c}},\theta _{1},\ldots ,\theta _{n_{c}}\right) $, to obtain diffusion of length $\frac{1}{2}$ in one of the actions. This can be done as follows. We define the strip $\mathbf{S}^{u}$ as in Remark \ref{rem:diffusion-from-discs-multi-action} and consider $I\in \left[ 0,1\right]^{n_{c}}$. We need the condition \textbf{C1}, but it is enough that (\ref{eqn:C1_iv}) holds for one action $I_{i_{\ast }}$ that is, $\pi _{I_{i_{\ast}}}F_{l,\varepsilon }(z)-\pi _{I_{i_{\ast }}}(z)>c\varepsilon $. We can then repeat the construction from the proof of Theorem \ref{thm:diffusion} taking $h^{\varepsilon }\left( x\right) =\left( x,y_{0},I_{0},\theta _{0}\right) $ for arbitrary $y_{0}\in \bar{B}^{n_{s}},$ $I_{0}=\left( \frac{1}{2},\ldots ,\frac{1}{2}\right) $ and any $\theta _{0}\in S_{\theta }^{u}$. As we propagate $h^{\varepsilon }$ through successive connecting sequences we increase in $I_{i_{\ast }}$. There is a possibility that as we propagate the discs some of them will leave $\{I\in \left[ 0,1\right] ^{n_{c}}\}$ along some other action than $I_{i_{\ast }}$. Nevertheless, for this to happen we also achieve a change of order $O(1)$ in the $I_{i_{\ast }}$-direction. In such case we obtain diffusion in the $I_{i_{\ast }}$-direction as well as in some other action direction.\end{remark}

%TCIDATA{Version=5.00.0.2606}
%TCIDATA{LaTeXparent=0,0,MMFedit.tex}

\section{Symbolic dynamics\label{sec:symbolic}}

In this section we prove Theorem~\ref{th:symbolic-dynamics-from-covering}.
We do so by first formulating a version of this theorem in terms of $%
Q^{\varepsilon }$-discs -- Theorem \ref{th:symbolic-dynamics-general} --
which we then use for the proof of Theorem~\ref%
{th:symbolic-dynamics-from-covering}. We also discuss how the results can be
generalized to higher dimensional $I$.

\subsection{Symbolic dynamics under conditions on $Q^{\protect\varepsilon }$%
-discs\label{sec:symbolic_horizontal}}

Here we introduce a theorem which establishes the symbolic dynamics under
assumptions on propagation of $Q^{\varepsilon }$-discs. As before, for $\ell
\in L$ we let $F_{\ell ,\varepsilon }=f_{\ell ,k_{\ell },\varepsilon }\circ
\ldots \circ f_{\ell ,1\varepsilon }\in \mathscr{F}_{\varepsilon }$, where $%
\mathscr{F}_{\varepsilon }$ is our IFS from (\ref{eq:IFS}). We formulate the
following condition.

\textbf{Condition \hypertarget{cond:A2}{A2}.} There exist constants $0<c<C$, such that for every $%
\varepsilon \in (0,\varepsilon _{0}]$ and every $Q^{\varepsilon }$-disk $%
h^{\varepsilon }\subseteq \mathbf{S}^{u}\cup \mathbf{S}^{d}$ \correctiona{the
following holds:}\marginpar{\color{red}31.\hfill 31.}

\begin{description}
\item[(A2.1)] \hypertarget{cond:A2.1}{}For each $\kappa \in \{u,d\}$ there exist $F_{\ell
,\varepsilon }$ in $\mathscr{F}_{\varepsilon }$ and a $Q^{\varepsilon }$%
-disk $\tilde{h}^{\varepsilon }:\bar{B}^{u}\rightarrow \mathbf{S}^{\kappa }$
with
\begin{equation*}
\tilde{h}^{\varepsilon }\subseteq F_{\ell ,\varepsilon }(h^{\varepsilon }).
\end{equation*}

\item[(A2.2)] \hypertarget{cond:A2.2}{}For $h^{\varepsilon }$, $\tilde{h}^{\varepsilon }$ as in (1),
for all pairs of points $z\in h^{\varepsilon }$, $\tilde{z}\in \tilde{h}%
^{\varepsilon }$ with $\tilde{z}=F_{\ell ,\varepsilon }(z)$ we have the
following

\begin{description}
\item[(A2.2.i)]\hypertarget{cond:A2.2.i}{}
\begin{equation}
\left\vert \pi _{I}(\tilde{z})-\pi _{I}(z)\right\vert <C\varepsilon .
\label{eq:move-ud}
\end{equation}

\item[(A2.2.ii)] \hypertarget{cond:A2.2.ii}{}If $h^{\varepsilon }\subseteq \mathbf{S}^{u}$ (resp. $%
\mathbf{S}^{d}$) and $\tilde{h}^{\varepsilon }\subseteq \mathbf{S}^{u}$
(resp. $\mathbf{S}^{d}$) we have the following
\begin{equation}  \label{eq:move-uu}
\begin{split}
c\varepsilon <&\pi _{I}(\tilde{z})-\pi _{I}(z) \\
(\text{resp. } c\varepsilon <&\pi _{I}(z)-\pi _{I}(\tilde{z})).
\end{split}%
\end{equation}
\end{description}
\end{description}

\begin{theorem}
\label{th:symbolic-dynamics-general} Assume that the IFS $\mathscr{F}%
_{\varepsilon }$ satisfies Condition \hyperlink{cond:A2}{\bf A2}. Let $\eta =2a_{I}+C$.

Then for every $\varepsilon \in (0,\varepsilon _{0}]$ and every infinite
sequence of $I$ -level sets $\left( I^{n}\right) _{n\in \mathbb{N}}\subseteq
\mathbb{R}$, with $\left\vert I^{n+1}-I^{n}\right\vert >2\eta \varepsilon $
there exists an orbit $\left( z_{n}\right) _{n\in \mathbb{N}}$ of the IFS,
and an increasing sequence $\left( k_{n}\right) _{n\in \mathbb{N}}\subseteq
\mathbb{N}$, such that for every $n\in \mathbb{N}$ we have
\begin{equation}
\left\vert \pi _{I}(z_{k_{n}})-I^{n}\right\vert <\varepsilon \eta .
\label{eq:energy-shadow}
\end{equation}
\end{theorem}

\begin{proof}
The main idea is the following. If we want to increase the $I$-coordinate we
move the $Q^{\varepsilon }$-disc by iterating from the energy strip $\mathbf{%
S}^{u}$ to itself. If we want to decrease the $I$-coordinate we \correctiona{propagate }%\marginpar{\color{red}32.\hfill 32.} 
the disk from the energy strip $\mathbf{S}^{d}$ to itself. If we want to switch
from increasing the $I$-coordinate to decreasing the $I$-coordinate, or from
decreasing the $I$-coordinate to increasing the $I$-coordinate, we \correctiona{propagate }the $%
Q^{\varepsilon }$-disc from $\mathbf{S}^{u}$ to $\mathbf{S}^{d}$, or from $%
\mathbf{S}^{d}$ to $\mathbf{S}^{u}$, respectively. We provide the details
below.

\emph{Case 1.A.} Assume $I^{0}<I^{1}$. Start with any $Q^{\varepsilon }$%
-disk $h_{0}^{\varepsilon }$ in $\mathbf{S}^{u}$ such that $\pi
_{I}(h_{0}^{\varepsilon })\subseteq (I^{0}-\eta ,I^{0}+\eta )$. By
assumption \hyperlink{cond:A2.1}{(A2.1)} there is a map $F_{\ell ,\varepsilon }\in \mathscr{F}%
_{\varepsilon }$ and a $Q^{\varepsilon }$-disk $\tilde{h}^{\varepsilon }$ in
$\mathbf{S}^{u}$ such that $\tilde{h}^{\varepsilon }\subseteq F_{\ell
,\varepsilon }(h_{0}^{\varepsilon })$. By \hyperlink{cond:A2.2.i}{(A2.2.{\scriptsize I})} and \hyperlink{cond:A2.2.ii}{(A2.2.{\scriptsize II})} for each
point $z_{0}\in h_{0}^{\varepsilon }$ with $\tilde{z}=F_{\ell ,\varepsilon
}(z_{0})\in \tilde{h}^{\varepsilon }$ we have
\begin{equation*}
\varepsilon c<\pi _{I}(\tilde{z})-\pi _{I}(z_{0})<\varepsilon C.
\end{equation*}%
That is, moving the $Q^{\varepsilon }$-disc $h_{0}^{\varepsilon }$ along the
composition of maps to a $Q^{\varepsilon }$-disk $\tilde{h}^{\varepsilon }$,
each point on $h_{0}^{\varepsilon }$ that lands on $\tilde{h}^{\varepsilon }$
changes its $I$-coordinate by at least $\varepsilon c$ and at most $%
\varepsilon C$. This means that, by repeating this procedure for finitely
many times we can obtain a $Q^{\varepsilon }$-disk $h_{1}^{\varepsilon }$
such that, for some point $z_{\ast }\in h_{1}^{\varepsilon }$ we have $%
\left\vert \pi _{I}\left( z_{\ast }\right) -I^{1}\right\vert <C\varepsilon .$
Then by Remark \ref{rem:discs-slim} for any $z_{1}\in h_{1}^{\varepsilon }$,
\begin{equation}
\left\vert \pi _{I}\left( z_{1}\right) -I^{1}\right\vert \leq \left\vert \pi
_{I}\left( z_{1}\right) -\pi _{I}\left( z_{\ast }\right) \right\vert
+\left\vert \pi _{I}\left( z_{\ast }\right) -I^{1}\right\vert <2\varepsilon
a_{I}+\varepsilon C\leq \varepsilon \eta ,  \label{eq:case1A-ineq}
\end{equation}%
as required for (\ref{eq:energy-shadow}).

\emph{Case 1.B.} If $I^{0}>I^{1}$, we proceed with a similar construction
starting with any $Q^{\varepsilon }$-disk $h_{0}^{\varepsilon }$ contained
in $\mathbf{S}^{d}$, and moving the disk from $\mathbf{S}^{d}$ to $\mathbf{S}%
^{d}$, until we obtain a $Q^{\varepsilon }$-disk $h_{1}^{\varepsilon }$
satisfying $\left\vert \pi _{I}\left( z_{1}\right) -I^{1}\right\vert \leq
\varepsilon \eta $ for all $z_{1}\in h_{1}^{\varepsilon }$.

\emph{Case 2.A.} Assume that we are as in \emph{Case 1.A}, and that $%
I^{1}<I^{2}$. Then, starting with the $Q^{\varepsilon }$-disk $%
h_{1}^{\varepsilon }$ obtained at the end of the construction, we proceed in
the same way as in \emph{Case 1.A}, moving the disk $h_{1}^{\varepsilon }$
repeatedly by passing from $\mathbf{S}^{u}$ to $\mathbf{S}^{u}$ (by doing so
we increase $I$), until we obtain a $Q^{\varepsilon }$-disk $%
h_{2}^{\varepsilon }$ satisfying %\begin{equation}
$\left\vert \pi _{I}(z_{2})-I^{2}\right\vert <\varepsilon \eta , $
%\label{eqn:symb2}
%\end{equation}%
for all $z_{2}\in h_{2}^{\varepsilon }$.

If we are as in \emph{Case 1.B}, and $I^{1}>I^{2}$, we proceed in a similar
fashion.

\emph{Case 2.B.} Assume that we are as in \emph{Case 1.A}, that is $%
I^{0}<I^{1}$, and that $I^{1}>I^{2}$. Consider the $Q^{\varepsilon }$-disk $%
h_{1}^{\varepsilon }$ obtained at the end of the construction of \emph{Case
1.A}. If for some $z\in h_{1}^{\varepsilon }$ we have $\pi _{I}\left(
z\right) -I^{2}<C\varepsilon $, then we \correctiona{propagate }the disc $h_{1}^{\varepsilon
}$ by passing from $\mathbf{S}^{u}$ to $\mathbf{S}^{u}$ (by doing so we
increase $I$) until we end up with a $Q^{\varepsilon }$-disc $%
h_{1}^{\varepsilon ,\ast }$ such that for any $z\in h_{1}^{\varepsilon ,\ast
}$, $\pi _{I}\left( z\right) -I^{2}>C\varepsilon $. Once this is achieved,
we choose a mapping $F_{\ell ,\varepsilon }$ in $\mathscr{F}_{\varepsilon }$
satisfying \hyperlink{cond:A2.1}{(A2.1)}, \hyperlink{cond:A2.2}{(A2.2)} with $\kappa=d$ and obtain a $Q^{\varepsilon }$%
-disc $\tilde{h}$ in $\mathbf{S}^{d}$, $\tilde{h}^{\varepsilon }\subseteq
F_{\ell ,\varepsilon }(h_{1}^{\varepsilon ,\ast })$. For any $\tilde{z}\in
\tilde{h}^{\varepsilon }$, $\tilde{z}=F_{\ell ,\varepsilon }\left( z\right) $
for some $z\in h_{1}^{\varepsilon ,\ast }$, we have
\begin{equation*}
\pi _{I}\left( \tilde{z}\right) -I^{2}=\left( \pi _{I}\left( \tilde{z}%
\right) -\pi _{I}\left( z\right) \right) +\left( \pi _{I}\left( z\right)
-I^{2}\right) >\left( -C\varepsilon \right) +\left( C\varepsilon \right) =0.
\end{equation*}%
If for all $z\in \tilde{h}^{\varepsilon }$ we have $\left\vert \pi
_{I}(z)-I^{2}\right\vert <\varepsilon \eta $, then we take $%
h_{2}^{\varepsilon }=\tilde{h}^{\varepsilon }$, and the step is finished. If
not, then we proceed with the construction from \emph{Case 1.B.} i.e.,
\correctiona{propagate }the disc $\tilde{h}^{\varepsilon }$ by passing from $\mathbf{S}^{d}$
to $\mathbf{S}^{d}$, as in Case 1.B (by doing so we decrease $I$) until we
end up with $h_{2}^{\varepsilon }$ for which for all $z_{2}\in
h_{2}^{\varepsilon }$ we have $|\pi _{I}\left( z_{2}\right)
-I^{2}|<\varepsilon \eta .$

If we are as in \emph{Case 1.B}, and $I^{1}<I^{2}$, that is $I^{0}>I^{1}$
and $I^{2}>I^{1}$, we proceed in a similar fashion, but switching from the
energy strip $\mathbf{S}^{d}$ to the energy strip $\mathbf{S}^{u}$.

Repeating these constructions for all $\{I^{n}\}_{\sigma \in \mathbb{N}}$,
we obtain at each step a $Q^{\varepsilon }$-disk $h_{n}^{\varepsilon }$ such
that for all $z_{n}\in h_{n}^{\varepsilon }$ we have %\begin{equation}
$| \pi _{I}(z_{n})-I^{n}| <\varepsilon \eta . $ %\label{eqn:symb4}
%\end{equation}%
We reach a $Q^{\varepsilon }$-disc $h_{n}^{\varepsilon }$ by iterating the $%
Q^{\varepsilon }$-disc $h_{n-1}^{\varepsilon }$ by successively applying
some maps of the IFS $\mathscr{F}_{\varepsilon }$, i.e., $h_{n}^{\varepsilon
}\subseteq F_{\ell _{i_{n}}^{n},\varepsilon }\circ \ldots F_{\ell
_{1}^{n},\varepsilon }\left( h_{n-1}^{\varepsilon }\right) $, for some $\ell
_{1}^{n},\ldots ,\ell _{i_{n}}^{n}\in L$. This means that there exists a
topological disc $\mathscr{D}_{n}^{\varepsilon }\subset h_{0}^{\varepsilon }$
such that
\begin{equation*}
h_{n}^{\varepsilon }=\left( \left( F_{\ell _{i_{n}}^{n},\varepsilon }\circ
\ldots F_{\ell _{1}^{n},\varepsilon }\right) \circ \ldots \circ \left(
F_{\ell _{i_{1}}^{1},\varepsilon }\circ \ldots F_{\ell _{1}^{1},\varepsilon
}\right) \right) \left( \mathscr{D}_{n}^{\varepsilon }\right) .
\end{equation*}%
Since $\mathscr{D}_{n}^{\varepsilon }\subseteq \ldots \subseteq \mathscr{D}%
_{1}^{\varepsilon }\subseteq \mathscr{D}_{0}^{\varepsilon },$ using the
Nested Compact Set Theorem, we obtain an orbit $(z_{n})_{n\in \mathbb{N}}$
of the IFS together with the sequence $k_{n}=i_{1}+\ldots +i_{n}$ as in the
statement of the theorem.
\end{proof}

\subsection{Symbolic dynamics under conditions on connecting sequences\label%
{sec:symb-dyn-covering}}

In this section we give the proof of Theorem \ref
{th:symbolic-dynamics-from-covering}, and later comment how it can be
generalized to higher dimensional $I$.

\begin{proof}[Proof of Theorem~\protect\ref%
{th:symbolic-dynamics-from-covering}]
The proof follows from the same construction as the proof of Theorem~\ref%
{th:symbolic-dynamics-general}.

As in the proof of Theorem \ref{thm:diffusion}, condition \hyperlink{cond:A2.1}{(A2.1)} of Theorem %
\ref{th:symbolic-dynamics-general} follow from \hyperlink{cond:C2.i}{(C2.i}--\hyperlink{cond:C2.iii}{C2.iii)}, combined with
Theorem \ref{th:hor-disc-propagation}. Conditions \hyperlink{cond:A2.2.i}{(A2.2.{\scriptsize I})} and \hyperlink{cond:A2.2.ii}{(A2.2.{\scriptsize II})} of
Theorem \ref{th:symbolic-dynamics-general} follow from \hyperlink{cond:C2.iv}{(C2.iv)}, \hyperlink{cond:C2.v}{(C2.v)},
combined with Theorem \ref{th:hor-disc-propagation}. The shadowing of $I^{n }
$ is therefore obtained identically as in the proof of Theorem \ref%
{th:symbolic-dynamics-general}.

Conditions \hyperlink{cond:C2.i}{(C2.i}--\hyperlink{cond:C2.iii}{C2.iii)} ensure that any $Q^{\varepsilon }$-disc in $%
\mathbf{S}^{\kappa }\cap \{I\in \lbrack 0,1]\}$, for $\kappa \in \{u,d\}$,
can be propagated to $\mathbf{S}^{u}$ and to $\mathbf{S}^{d}$.

One technical detail is to show that we can obtain $Q^{\varepsilon }$-discs
whose $I$-coordinates $(\varepsilon \eta )$-shadow the given action level
sets so that these discs never leave $(\mathbf{S}^{u}\cup \mathbf{S}%
^{d})\cap \{I\in \lbrack 0,1]\}$. More precisely, we need to ensure that we
can propagate these discs through connecting sequences such that the
successive iterates of the discs under $F_{\ell ,\varepsilon }$ return to $(%
\mathbf{S}^{u}\cup \mathbf{S}^{d})$ without leaving $\{I\in \lbrack 0,1]\}$.

Assume that the $Q^{\varepsilon }$-disc $h^{\varepsilon }$ is such that $%
\pi_{I}h^{\varepsilon }\subset \left( I^{n}-\varepsilon \eta
,I^{n}+\varepsilon \eta \right) $. Since $(I_{n})_{n\in \mathbb{N}}$ $\subseteq
[2\varepsilon \eta ,1-2\varepsilon \eta ]$, it follows that
\begin{equation}
\pi _{I}(h^{\varepsilon })\subseteq (\varepsilon \eta ,1-\varepsilon \eta ).
\label{eqn:I_eta_one_minus_eta}
\end{equation}%
Let $\tilde{h}^{\varepsilon }$ be the image of $h^{\varepsilon }$ in $%
\mathbf{S}^{u}\cup \mathbf{S}^{d}$ under some appropriate mapping $F_{\ell
,\varepsilon }$ from $\mathscr{F}_{\varepsilon }$. Since $\eta =2a_{I}+C$,
we have
\begin{eqnarray}
\pi _{I}\tilde{h}^{\varepsilon } &>&\min_{z\in h^{\varepsilon }}\pi
_{I}z-\varepsilon C>\varepsilon \eta -\varepsilon C>0,
\label{eq:hor-disc-above-zero} \\
\pi _{I}\tilde{h}^{\varepsilon } &<&\max_{z\in h^{\varepsilon }}\pi
_{I}z+\varepsilon C<1-\varepsilon \eta +\varepsilon C<1.
\label{eq:hor-disc-below-one}
\end{eqnarray}

If in our construction we have a $Q^{\varepsilon }$-disc in $\mathbf{S}^{u}$
(resp. in $\mathbf{S}^{d}$), and we need to increase its $I$-coordinate from
an $(\varepsilon \eta )$-neighborhood of $I^{n}$ to an $(\varepsilon \eta )$%
-neighborhood of $I^{n+1}$, with $I^{n}<I^{n+1}$, (resp. to decrease its $I$%
-coordinate from an $(\varepsilon \eta )$-\hspace{0.0001cm}neighborhood of $%
I^{n}$ to an $(\varepsilon \eta )$-neighborhood of $I^{n+1}$), then we
propagate the $Q^{\varepsilon }$-disc from $\mathbf{S}^{u}$ to $\mathbf{S}%
^{u}$ (resp. from $\mathbf{S}^{d}$ to $\mathbf{S}^{d}$). By (\ref%
{eqn:I_eta_one_minus_eta}) we can do that without the discs having to leave $%
\{I\in (\varepsilon \eta ,1-\varepsilon \eta )\}$.

If in our construction we have a $Q^{\varepsilon }$-disc in $\mathbf{S}^{u}$
and we need to decrease its $I$-coordinate from an $(\varepsilon \eta )$%
-neighborhood of $I^{n}$ to an $(\varepsilon \eta )$-neighborhood of $%
I^{n+1} $, with $I^{n}>I^{n+1}$, we first propagate the disc to $\mathbf{S}%
^{d}$ without leaving $\{I\in \left[ 0,1\right] \}$, which is ensured by (%
\ref{eq:hor-disc-above-zero}--\ref{eq:hor-disc-below-one}). Then we continue
to propagate the disc from $\mathbf{S}^{d}$ to $\mathbf{S}^{d}$. This means
that %some of
the successive $Q^{\varepsilon }$-discs will be in $\{I\in \left[ 0,1\right]
\}$ and they will subsequently return to $\{I\in (\varepsilon \eta
,1-\varepsilon \eta )\}$.

By mirror arguments, if we have a $Q^{\varepsilon }$-disc in $\mathbf{S}^{d}$
and we need to increase its $I$-coordinate from an $(\varepsilon \eta )$%
-neighborhood of $I^{n}$ to an $(\varepsilon \eta )$-neighborhood of $%
I^{n+1} $, with $I^{n}<I^{n+1}$, we first propagate the disc to $\mathbf{S}%
^{u}$ without leaving $\{I\in \left[ 0,1\right] \}$, which is ensured by (%
\ref{eq:hor-disc-above-zero}--\ref{eq:hor-disc-below-one}). Then we
propagate from $\mathbf{S}^{u}$ to $\mathbf{S}^{u}$. This means that
%some of
the successive $Q^{\varepsilon }$-discs will be in $\{I\in \left[ 0,1\right]
\}$ and they will subsequently return to $\{I\in (\varepsilon \eta
,1-\varepsilon \eta )\}$.
\end{proof}

\correctiona{\begin{remark}

Theorem \ref{th:symbolic-dynamics-general} can be generalized 
to the higher dimensional case where the action angle coordinates are $\left( I_{1},\ldots ,I_{n_{c}},\theta _{1},\ldots,\theta _{n_{c}}\right)$, by taking the strip $\mathbf{S}^{\kappa }=\bar{B}%
^{n_{u}}\times \bar{B}^{n_{s}}\times \mathbb{R}^{n_{c}}\times S_{\theta
}^{\kappa }$ with $S_{\theta }^{\kappa }=\Pi _{i=1}^{n_{c}}[S_{i}^{\kappa
,1},S_{i}^{\kappa ,2}]$, for $\kappa \in \left\{ u,d\right\} $, and assuming
conditions \hyperlink{cond:A2.1}{(A2.1)}, \hyperlink{cond:A2.2.i}{(A2.2.{\scriptsize I})}, \hyperlink{cond:A2.2.ii}{(A2.2.{\scriptsize II})} from Theorem \ref{th:symbolic-dynamics-general}, with the difference that conditions  (\ref{eq:move-ud}--\ref{eq:move-uu}) are needed only for one of the components
$I_{i_{\ast }}$ of $I$. Then sequence of $I_{i_{\ast }}$-level sets $\left(
I_{i_{\ast }}^{n}\right) _{n\in \mathbb{N}}$ with $\left\vert I_{i_{\ast
}}^{n+1}-I_{i_{\ast }}^{n}\right\vert >2\eta \varepsilon $ we can find an
orbit $\left( z_{n}\right) _{n\in \mathbb{N}}$ of the IFS, which $\varepsilon \eta $ shadows the sequence of the IFS, i.e. such that for every
$n\in \mathbb{N}$ we have $\left\vert \pi _{I_{i_{\ast }}}(z_{n})-I_{i_{\ast}}^{n}\right\vert <\varepsilon \eta $.

Generalizing Theorem \ref%
{th:symbolic-dynamics-from-covering} to the setting of higher dimensional $I$ is more subtle. 
%This is not entirely straightforward.
In the case when $n_c=1$, condition \hyperlink{cond:C2}{\textbf{C2}} ensures that $Q^{\varepsilon
}$-discs can be propagated as long as we have $I\in \left[ 0,1\right] $ for
all points on these disc. In higher dimensions we need an analogue of this.
We restrict to the setting of $n_{c}=2$. This is enough to demonstrate the
ideas with which we can achieve this, without getting too technical.

When $n_{c}=2$ we define the strips $\mathbf{S}^{uu},\mathbf{S}^{ud},\mathbf{%
S}^{du},\mathbf{S}^{dd}$ with
\begin{equation*}
\mathbf{S}^{i_{1}i_{2}}=\bar{B}^{u}\times \bar{B}^{s}\times \mathbb{R}%
^{2}\times S_{\theta }^{i_{1}i_{2}},\qquad S_{\theta }^{i_{1}i_{2}}=\Pi
_{k=1}^{2}[\theta _{k}^{i_{1}i_{2},1},\theta _{k}^{i_{1}i_{2},2}],\text{ for
}i_{1},i_{2}\in \left\{ u,d\right\} .
\end{equation*}%
The role of these strips will be that propagating a $Q^{\varepsilon }$-disc
from $\mathbf{S}^{i_{1}i_{2}}$ to $\mathbf{S}^{i_{1}i_{2}}$ will increase
the coordinate $I_{k}$ if $i_{k}=u$ and decrease $I_{k}$ if $i_{k}=d$, for $%
k=1,2$.) We define
\begin{equation*}
L=\bigcup\limits_{i_{1},i_{2},j_{1},j_{2}\in \left\{ u,d\right\}
}L^{i_{1}i_{2},j_{1}j_{2}}
\end{equation*}%
and make the following modification of condition \hyperlink{cond:C2}{\textbf{C2}}.

\medskip

\textbf{Condition C2'.}

\begin{itemize}
\item[(C2'.i)] For each $\ell\in
L^{i_{1}i_{2},j_{1}j_{2}}$, $i_{1},i_{2},j_{1},j_{2}\in \{u,d\},$ and each $%
\varepsilon \in (0,\varepsilon _{0}]$, there is
 a connecting sequence $(N_{\ell,0},\ldots ,N_{\ell,k_\ell})$ from $\mathbf{S}^{i_{1}i_{2}}$ to $\mathbf{S}%
^{j_{1}j_{2}}$;

\item[(C2'.ii)] On the $I,\theta $-coordinate we
have
\begin{equation*}
\bigcup\limits_{\ell\in L^{i_{1}i_{2},j_{1}j_{2}}}\pi _{I,\theta }\left(
N_{\ell,0}\right) =\left[ 0,1\right] ^{2}\times S_{\theta
}^{i_{1}i_{2}},\qquad \text{for }i_{1},i_{2},j_{1},j_{2}\in \{u,d\};
\end{equation*}

\item[(C2'.iii)] Whenever $N_{\ell,0}\cap
N_{\ell^{\prime },0}\neq \emptyset $ for $\ell,\ell^{\prime }\in
L^{i_{1}i_{2},j_{1}j_{2}},$ for every $\left( I^{\ast },\theta ^{\ast
}\right) \in \pi _{I,\theta }(N_{\ell,0}\cap N_{\ell^{\prime},0})$ the multi
dimensional rectangle $\bar{B}^{n_u}\times \bar{B}^{n_s}\times \bar{B}\left(
I^{\ast },\varepsilon _{0}a_{I}\right) \cap \left[ 0,1\right] ^{2}\times
\bar{B}\left( \theta ^{\ast },a_{\theta }\right) \cap S_{\theta
}^{i_{1}i_{2}}$ is contained in $N_{\ell,0}$ or $N_{\ell^{\prime },0}$;

\item[(C2'.iv)] There exists a constant $C>0$,
such that if $z$ passes through a connecting sequence $(N_{\ell ,0},\ldots
,N_{\ell ,k_{\ell }})$ we have
\begin{equation*}
\left\vert \pi _{I}F_{\ell ,\varepsilon }(z)-\pi _{I}(z)\right\vert
<\varepsilon \cdot C;
\end{equation*}

\item[(C2'.v)] There exists a $c>0$ such that for $z
$ passing through a connecting sequence
\begin{align*}
\varepsilon \cdot c& <\pi _{I_{1}}F_{\ell ,\varepsilon }(z)-\pi
_{I_{1}}(z),\quad \text{if }\ell \in L^{uu,uu},L^{ud,ud}, \\
\varepsilon \cdot c& <\pi _{I_{2}}F_{\ell ,\varepsilon }(z)-\pi
_{I_{2}}(z),\quad \text{if }\ell \in L^{uu,uu},L^{du,du}, \\
\varepsilon \cdot c& <\pi _{I_{1}}(z)-\pi _{I_{1}}F_{\ell ,\varepsilon
}(z),\quad \text{if }\ell \in L^{dd,dd},L^{du,du}, \\
\varepsilon \cdot c& <\pi _{I_{2}}(z)-\pi _{I_{2}}F_{\ell ,\varepsilon
}(z),\quad \text{if }\ell \in L^{dd,dd},L^{ud,ud}.
\end{align*}
\end{itemize}

Using analogous arguments to the proof Theorem \ref%
{th:symbolic-dynamics-from-covering}, for $\eta \geq 2a_{I}+C$ we can $%
\varepsilon \eta $ shadow any sequence $I^{n}=\left(
I_{1}^{n},I_{2}^{n}\right) \subset \left[ 2\eta \varepsilon ,1-2\eta
\varepsilon \right] ^{2}$, $n\in \mathbb{N}$, for which $\left\Vert
I^{n+1}-I^{n}\right\Vert _{\max }>2\eta \varepsilon $. The idea is that we
can move in any desired $I_{1}$ or $I_{2}$ direction, depending on the
choice of the strips between which we propagate the $Q^{\varepsilon }$-discs, while keeping them inside $\{I\in \left[ 0,1\right] ^{2}\}$ at the
same time.

\end{remark}} 

%{\correctionr{\begin{remark}
%Condition \textbf{A.1} in Section \ref{sec:diffusion} and  condition \textbf{A.2} in Section \ref{sec:symbolic} are formulated, for simplicity, for
%$\varepsilon\in(0,\varepsilon_0]$. These conditions and the corresponding results have straightforward
%analogues in case when $\varepsilon$ is in one of the subintervals from Remark \ref{rem:epsilon-subdivision}.
%\end{remark}}

%\correction{
%\begin{remark} Just as in Remark \ref{rem:epsilon-subdivision}, we can validate the condition \textbf{A.1} from Section \ref{sec:diffusion} and the condition \textbf{A.2} from Section \ref{sec:symbolic} by subdividing $(0,\varepsilon_0]$ into subintervals, and checking these conditions on each subinterval separately.
%\end{remark}
%}\marginpar{\color{red} 29.}

%TCIDATA{Version=5.00.0.2606}
%TCIDATA{LaTeXparent=0,0,MMFedit.tex}

\section{Hausdorff dimension\label{sec:Hausdorff}}

For $V\subset \mathbb{R}^{n}$ let $\mathrm{\dim }_{H}\left( V\right) $
denote the Hausdorff dimension of $V$ (see \cite{Falconer} for the
definition and basic properties). If $\pi $ stands for the orthogonal
projection on some selection of Cartesian coordinates of $\mathbb{R}^{n}$,
it is a standard result that 
\begin{equation}
\mathrm{\dim }_{H}\left( V\right) \geq \mathrm{\dim }_{H}\left( \pi V\right)
.  \label{eq:Haus-proj-bound}
\end{equation}

\begin{proof}[Proof of Theorem \protect\ref{th:Hausdorff-dim}.]
Let $\varepsilon \in (0,\varepsilon _{0}]$ be fixed and let \hyperlink{cond:C2}{\textbf{C2}} hold.

For any given sequence of action levels $\mathcal{I}=\left( I^{n}\right)
_{n\in \mathbb{N}}\subset \left( 2\varepsilon \eta ,1-2\varepsilon \eta
\right) $ satisfying $\left\vert I^{n+1}-I^{n}\right\vert >2\varepsilon \eta 
$, let $\mathbf{z}=\mathbf{z}(\mathcal{I})=\left( z^{n}(\mathcal{I})\right)
_{n\in \mathbb{N}}$ be an orbit of the IFS that $(\eta\varepsilon)$-shadows
the sequence as in \eqref{eq:energy-shadow}. The existence of such an orbit
follows from Theorem \ref{th:symbolic-dynamics-general}. Denote 
\begin{align*}
V_{\mathcal{I}}&:= \\ & \left\{ z^{0}(\mathcal{I})\in \mathbf{S}^{u}\cup \mathbf{S}%
^{d}:z^{0}(\mathcal{I})\text{ is an initial point of }\mathbf{z}(\mathcal{I}%
),\text{ for which \eqref{eq:energy-shadow} holds}\right\} .
\end{align*}

As in the proof of Theorem \ref{th:symbolic-dynamics-general}, for every $%
Q^{\varepsilon }$-disk $h_{0}^{\varepsilon }$ of the form 
\begin{equation}  \label{eqn:disc_straight}
h_{0}^{\varepsilon }(x)=h_{y^{\ast },I^{\ast },\theta ^{\ast }}^{\varepsilon
}(x)=\left( x,y^{\ast },I^{\ast },\theta ^{\ast }\right),
\end{equation}
with fixed $y^{\ast }\in \bar{B}^{n_{s}}$, $I^{\ast }\in (I^{0}-\varepsilon
\eta ,I^{0}+\varepsilon \eta )$, and $\theta ^{\ast }\in S_{\theta }^{u}\cup
S_{\theta }^{d}$, there exists an initial point $z^{0}(\mathcal{I})$ on that
disc that is in $V_{\mathcal{I}}$. This implies that 
\begin{equation*}
\pi _{y,I,\theta }V_{\mathcal{I}}\supset \bar{B}^{n_{s}}\times
(I^{0}-\varepsilon \eta ,I^{0}+\varepsilon \eta )\times S_{\theta }^{u}\cup
S_{\theta }^{d},
\end{equation*}
hence from \eqref{eq:Haus-proj-bound} we conclude 
\begin{equation*}
\mathrm{\dim }_{H}\left( V_{\mathcal{I}}\right) \geq \mathrm{\dim }%
_{H}\left( \pi _{y,I,\theta }V_{\mathcal{I}}\right) =n_{s}+2.
\end{equation*}

We now assume the special case when $n_{u}=n_{s}=1$ and prove the second
part of Theorem \ref{th:Hausdorff-dim}. i.e., that the Hausdorff dimension
of the set of orbits which follow a prescribed sequence of actions is
strictly greater than $n_{s}+2=3$. This special case corresponds to the
setting of the PER3BP.

We will construct a family of $1$-dimensional Cantor sets inside of $V_{%
\mathcal{I}}$. We will use Lemma \ref{lem:Hausdorff-Cantor} below to show
that each Cantor set has positive Hausdorff measure. Then we will apply
Lemma \ref{lem:Hausdorff-product} below to show that $V_{\mathcal{I}}$ has
Hausdorff dimension strictly greater than $n_{s}+2=3$.

\begin{lemma}
\cite{Falconer}\label{lem:Hausdorff-Cantor} Consider a one dimensional
Cantor set $\mathcal{D}$, which is formed by creating $m_{k}-1\geq 1$ gaps
in each interval at the $k$-th step of the construction. Assume that gaps
created at the $k$-th step are of length at least $\delta _{k}$, with $%
0<\delta _{k+1}<\delta _{k}$ for each $k$. Then%
\begin{equation*}
\dim _{H}\left( \mathcal{D}\right) \geq \lim_{k\rightarrow \infty }\inf_{l>k}%
\frac{\log \left( m_{1}\cdots m_{l-1}\right) }{-\log \left( m_{l}\delta
_{l}\right) }.
\end{equation*}
\end{lemma}

\begin{lemma}
\cite{Falconer}\label{lem:Hausdorff-product} Let $V$ be a Borel subset of $%
\mathbb{R}^{n}$. For $y\in \mathbb{R}^{n-1}$ let $L_{y}\subset \mathbb{R}%
^{n} $ be $L_{y}=\left\{ \left( x,y\right) :x\in \mathbb{R}\right\} $. If $%
n-1\leq s\leq n,$ then%
\begin{equation*}
\int_{\mathbb{R}^{n-1}}\mathcal{H}^{s-(n-1)}\left( V\cap L_{y}\right) dy\leq 
\mathcal{H}^{s}\left( V\right).
\end{equation*}
\end{lemma}

%As before, we consider $\varepsilon \in (0,\varepsilon _{0}]$ fixed.
Let $N$ be an integer satisfying 
\begin{equation}
N>\frac{2}{c\varepsilon },  \label{eqn:N_choice}
\end{equation}%
where $c$ is as in \eqref{eq:move-uu}. This choice ensures that by taking $N$
applications of maps $F_{\ell ,\varepsilon }$ from the IFS, any $%
Q^{\varepsilon }$-disk can be moved both up and down in action by $1$. 
%Since $\left\vertI^{n}-I^{n-1}\right\vert <1,$ by condition (C2.v), from the construction in the proof of Theorem \ref{th:symbolic-dynamics-general} we can choose $k_{n}$ in (\ref{eqn:Hausdorff_set}) such that $k_{n}<N$.

Consider the following compact sets 
\begin{eqnarray*}
A_{uu,uu} &=&\bigcup_{\ell _{1},\ell _{2}\in L^{uu}}F_{\ell _{1},\varepsilon
}^{-1}\circ F_{\ell _{2},\varepsilon }^{-1}(\mathbf{S}^{u}\cap \left\{ I\in %
\left[ 0,1\right] \right\} ), \\
A_{dd,dd} &=&\bigcup_{\ell _{1},\ell _{2}\in L^{dd}}F_{\ell _{1},\varepsilon
}^{-1}\circ F_{\ell _{2},\varepsilon }^{-1}(\mathbf{S}^{d}\cap \left\{ I\in %
\left[ 0,1\right] \right\} ), \\
A_{ud,du} &=&\bigcup_{\ell _{1}\in L^{ud},\ell _{2}\in L^{du}}F_{\ell
_{1},\varepsilon }^{-1}\circ F_{\ell _{2},\varepsilon }^{-1}(\mathbf{S}%
^{u}\cap \left\{ I\in \left[ 0,1\right] \right\} ), \\
A_{du,ud} &=&\bigcup_{\ell _{1}\in L^{du},\ell _{2}\in L^{ud}}F_{\ell
_{1},\varepsilon }^{-1}\circ F_{\ell _{2},\varepsilon }^{-1}(\mathbf{S}%
^{d}\cap \left\{ I\in \left[ 0,1\right] \right\} ).
\end{eqnarray*}%
By condition \hyperlink{cond:C3}{\textbf{C3}} there exists a $\delta $ satisfying%
\begin{equation}
0<\delta =\min \left\{ \mathrm{dist}\left( A_{uu,uu},A_{ud,du}\right) ,%
\mathrm{dist}\left( A_{dd,dd},A_{du,ud}\right) \right\} .  \label{eqn:gap}
\end{equation}%
By the compactness of $\left( \mathbf{S}^{u}\cup \mathbf{S}^{d}\right) \cap
\left\{ I\in \left[ 0,1\right] \right\} $ %, and since $L$ is finite,
there exist $\lambda >1$ such that 
\begin{equation}
\lambda >\sup \{\left\Vert D\left( F_{\ell _{n},\varepsilon }\circ \ldots
\circ F_{\ell _{1},\varepsilon }\right) \left( z\right) \right\Vert \},
\label{eqn:gap_lambda}
\end{equation}%
where the supremum is taken over all \[z\in \left( F_{\ell _{n},\varepsilon
}\circ \ldots \circ F_{\ell _{1},\varepsilon }\right) ^{-1}((\mathbf{S}%
^{u}\cup \mathbf{S}^{d})\cap \left\{ I\in \left[ 0,1\right] \right\} ), \quad \mbox{ for } 
\ell _{i}\in L\mbox{ and }n\leq N.\]

Suppose that $I^{1}>I^{0}$ (the case $I^{1}<I^{0}$ is similar). As in the
proof of Theorem \ref{th:symbolic-dynamics-general}, start with a $%
Q^{\varepsilon }$-disc $h_{0}^{\varepsilon }(x)$ of the form %
\eqref{eqn:disc_straight}, with $\pi _{I}h_{0}^{\varepsilon }\subset
(I^{0}-\eta \varepsilon ,I^{0}+\eta \varepsilon )$. We can \correctiona{propagate }this disc
until we obtain an image disc $h_{1}^{\varepsilon }$ with $\pi
_{I}(h_{1}^{\varepsilon })\subseteq (I^{1}-\eta \varepsilon ,I^{1}+\eta
\varepsilon )$ in two different ways: (i) by propagating the disc only
through $\mathbf{S}^{u}$ up to $(I^{1}-\eta \varepsilon ,I^{1}+\eta
\varepsilon )$, or (ii) by propagating the disc through 
%$\mathbf{S}^{u}$ up to some point, switching to
$\mathbf{S}^{d}$, 
%propagating it through $\mathbf{S}^{d}$ up to some point, and
switching back to $\mathbf{S}^{u}$, and propagating it through $\mathbf{S}%
^{u}$ up to $(I^{1}-\eta \varepsilon ,I^{1}+\eta \varepsilon )$. The choice
in \eqref{eqn:N_choice} ensures that each of these propagations can be
achieved by applying at most $N$ applications of maps $F_{\ell ,\varepsilon
} $ from the IFS. The points in $h_{0}^{\varepsilon }$ that get propagated
in these two ways yield two disjoint sub-discs $\mathcal{D}^{0}$, $\mathcal{D%
}^{1}$ of $h_{0}^{\varepsilon }$. By \eqref{eqn:gap} and %
\eqref{eqn:gap_lambda}, the distance between $\mathcal{D}^{0}$ and $\mathcal{%
D}^{1}$ is at least $\delta _{1}=\lambda ^{-1}\delta $.

%Continuing  this procedure further in order obtain an image disc $h_2$ with $\pi_{I}(h_2)\subseteq  (I^{2}-\eta \varepsilon ,I^{2}+\eta \varepsilon )$, the two ways to propagate yield  two disjoint sub-discs $\mathcal{D}^{00},\mathcal{D}^{01}$ of $\mathcal{D}^{0}$, and two disjoint sub-discs $\mathcal{D}^{10},\mathcal{D}^{11}$ of $\mathcal{D}^{1}$. By \eqref{eqn:gap} and \eqref{eqn:gap_lambda}, the gap between $\mathcal{D}^{00}$ and $\mathcal{D}^{01}$ (and respectively  between $\mathcal{D}^{10}$ and $\mathcal{D}^{11}$) is at least $\delta _{2}=\lambda ^{-2}\delta.$
Continuing this procedure, we obtain recursively a nested sequence of finite
unions of discs, such that the gaps between the discs at step $k$ are at
least $\delta _{k}=\lambda ^{-k}\delta $. The resulting Cantor set $\mathcal{%
D}\left( y^{\ast },I^{\ast },\theta ^{\ast }\right) $ depends on the choice
of the initial $Q^{\varepsilon }$-disc $h_{0}^{\varepsilon }$ given by %
\eqref{eqn:disc_straight}. From Lemma \ref{lem:Hausdorff-Cantor}%
\begin{equation*}
\dim _{H}\left( \mathcal{D}(y^{\ast },I^{\ast },\theta ^{\ast })\right) \geq
\lim_{k\rightarrow \infty }\frac{\log \left( 2^{k}\right) }{-\log \left(
2\lambda ^{-(k+1)}\delta \right) }=\log _{\lambda }2>0.
\end{equation*}

%It is clear that $V_{\mathcal{I}}$ is a Borel set.
%We now argue that $V_{\mathcal{I}}$ is a Borel set. We can represent it by the following construction. Take $A_{0}=\left( \mathbf{S}^{u}\cup \mathbf{S}^{d}\right) \cap \{I\in \left( I^{0}-\varepsilon \eta ,I^{0}+\varepsilon\eta \right) \}$ and inductively define \begin{equation*} A_{n+1}=A_{n}\cap \bigcup_{k\in \mathbb{N}}\bigcup_{\ell _{1},\ldots ,\ell _{k}\in L}\left( F_{\ell _{k},\varepsilon }\circ \ldots \circ F_{\ell _{1},\varepsilon }\right) ^{-1}\left( \left\{ I\in (I^{0}-\varepsilon \eta ,I^{0}+\varepsilon \eta )\right\} \right) . \end{equation*}  Clearly $A_{n}$ are Borel. Since \begin{equation*} V_{\mathcal{I}}=\bigcap_{n\in \mathbb{N}}A_{n},\end{equation*}% we see that is is a Borel set.
By Lemma~\ref{lem:Hausdorff-product} 
\begin{equation*}
\mathcal{H}^{s+2d+\log _{\lambda }2}\left( V_{\mathcal{I}}\right) \geq \int_{%
\bar{B}^{n_{s}}\times \left( I^{0}-\varepsilon \eta ,I^{0}+\varepsilon \eta
\right) \times \left( S_{\theta }^{u}\cup S_{\theta }^{d}\right) }\mathcal{H}%
^{\log _{\lambda }2}\left( \mathcal{D}\left( y,I,\theta \right) \right)
dy\,dI\,d\theta >0,
\end{equation*}%
so $\dim _{H}\left( V\right) \geq n_{s}+2d+\log _{\lambda }2$, proving our
claim. 
\end{proof}

%TCIDATA{Version=5.00.0.2606}
%TCIDATA{LaTeXparent=0,0,MMFedit.tex}

\section{Stochastic behavior}

\label{sec:stochastic}

Our proof will be based on the Donsker's functional central limit theorem.
We start by recalling the result.

Consider a sequence $Y_{n}$ of independent identically distributed random
variables with $\mathbb{E}\left( Y_{n}\right) =0$ and $Var(Y_{n})=1$ and
define 
\begin{equation*}
\begin{split}
S_{n}=&Y_{1}+\ldots +Y_{n}, \\
S\left( t\right) =& S_{\left\lfloor t\right\rfloor }+\left( t-\left\lfloor
t\right\rfloor \right) \left( S_{\left\lfloor t\right\rfloor
+1}-S_{\left\lfloor t\right\rfloor }\right) \qquad \text{ for } t\in
[0,+\infty), \\
S_{n}^{\ast }(t)=&\frac{S\left( nt\right) }{\sqrt{n}}\qquad \text{ for }t\in %
\left[ 0,1\right] .
\end{split}%
\end{equation*}

\begin{theorem}
\cite{Morters}\label{th:functional-CLT}(Donsker's functional central limit
theorem) On the space $C[0,1]$ of continuous functions on the unit interval
with the metric induced by the sup-norm, the sequence $S_{n}^{\ast }$
converges in distribution\footnote{%
See Section \ref{sec:master-stochastic} for more details on the convergence.}
to a standard Brownian motion $W_{t}$.
\end{theorem}

\subsection{Stochastic behavior under conditions on $Q^{\protect\varepsilon }
$-disks}

Let $\Omega _{\varepsilon }\subset \mathbb{R}^{u}\times \mathbb{R}^{s}\times 
\mathbb{R}^{2}\times \mathbb{T}$ be a set of positive Lebesgue measure. (We
add the subscript $\varepsilon $ in $\Omega _{\varepsilon }$ since in the
result stated below the choice of this set will depend on $\varepsilon $.)
On $\Omega _{\varepsilon }$ we define a probability space $\left( \Omega
_{\varepsilon },\mathbb{F}_{\varepsilon },\mathbb{P}_{\varepsilon }\right) $
by taking $\mathbb{F}_{\varepsilon }$ to be the sigma field of Borel sets on 
$\Omega _{\varepsilon }$ and 
\begin{equation}
\mathbb{P}_{\varepsilon }\left( B\right) :=\frac{\left\vert B\right\vert }{%
\left\vert \Omega _{\varepsilon }\right\vert }.  \label{eq:Prob-eps}
\end{equation}

We introduce the following notation. For a point $z$ in the extended phase
space we denote by $\pi_t(z)$ the time-coordinate $t$ of the point $z$,
where $t$ is the covering space $\mathbb{R}$ of $\mathbb{T}^1$. Thus the
`angle' coordinate $\theta$ is given by $\theta=t$ (mod $2\pi$). The
quantity $\pi_t(\Phi^\varepsilon_\tau(z))-\pi_t(z)$ represents the time $%
\tau $ it takes for a point $z$ to flow to $\Phi^\varepsilon_\tau(z)$. For $%
\tilde{h}^{\varepsilon }\subseteq F_{\ell ,\varepsilon }(h^{\varepsilon })$
as in Condition \hyperlink{cond:A2.1}{(A2.1)}, for each pair of points $z\in h^{\varepsilon }$, $%
\tilde{z}=F_{\ell ,\varepsilon }(z)\in \tilde{h}^{\varepsilon}$, there is an
associated time $\pi_t(\tilde{z})-\pi_t(z)$. We denote by $\pi_t(\tilde{h}%
^{\varepsilon})-\pi_t(h^\varepsilon)$ the time interval consisting of all $%
\pi_t(\tilde{z})-\pi_t(z)$ for all $z,\tilde{z}$ as above.

Below we define some uniformity conditions for all pairs of $Q^{\varepsilon
} $-disks $h^{\varepsilon }$, $\tilde{h}^{\varepsilon}$ as in Condition \hyperlink{cond:A2}{\bf A2}.

\textbf{Condition \hypertarget{cond:A3}{A3}.}

\begin{description}
\item[(A3.1)] \hypertarget{cond:A3.1}{}There exists $T_0>0$ such that for every pair of points $z$, $%
\tilde{z}$ as in Condition \hyperlink{cond:A2}{\bf A2} we have 
\begin{equation}  \label{eqn:T_0}
\pi_t(\tilde{z})-\pi_t(z)<T_0.
\end{equation}

\item[(A3.2)] \hypertarget{cond:A3.2}{}
%For every $T_1>0$ there exits $M_1=M_1(T_1)>0$ such that for every pair of  points $z,\tilde{z}  \in \mathbf{S}^{u}\cup \mathbf{S}^{d}$ with $\tilde{z}=F_{\ell _n,\varepsilon }\circ \ldots \circ F_{\ell_1,\varepsilon }\left( z\right) $ for some $n$, we have  \begin{equation} \pi _{I}\left( \tilde{z}\right) \in \left[ \pi _{I}z-\varepsilon M_{1},\pi _{I}z+\varepsilon M_{1}\right] ,\qquad \text{provided }\pi_t(\tilde{z})\in \left[ 0,T_{1}\right].  \label{eq:M1-bound}
There exists $M_{0}>0$ such that for every $z\in \mathbf{S}^{u}\cup \mathbf{S%
}^{d}$ and $t\in [0,T_{0}]$%
\begin{equation}
|\pi _{I}\left( \Phi _{t}^{\varepsilon }(z)\right) -\pi
_{I}(z)|<M_{0}\varepsilon .  \label{eq:M1-bound}
\end{equation}
\end{description}

Condition \hyperlink{cond:A3.1}{(A3.1)} says that the first-return times to $\mathbf{S}^{u}\cup 
\mathbf{S}^{d}$ via any of the maps $F_{\ell,\varepsilon }$ in $\mathscr{F}%
_{\varepsilon }$, for any $\ell\in L$, are uniformly bounded for $z\in 
\mathbf{S}^{u}\cup \mathbf{S}^{d}$. Condition \hyperlink{cond:A3.2}{(A3.2)} says that the change in 
$I$ along the flow $\Phi^\varepsilon_t$ is uniformly bounded for the length
of time it takes a point $z\in \mathbf{S}^{u}\cup \mathbf{S}^{d}$ to return
to $\mathbf{S}^{u}\cup \mathbf{S}^{d}$. 
%This assumption  is satisfied  for instance when for the vector field $J\nabla \left( H_{0}(z)+\varepsilon H_{1}(z;\varepsilon )\right) $ the term $J\nabla H_{0}$ is globally Lipschitz and $J\nabla H_{1}$ is bounded.

We recall that $\mathbf{S}^{u}\cup \mathbf{S}^{d}$ is non-compact. If we
restrict to the dynamics on a compact subset of $\mathbf{S}^{u}\cup \mathbf{S%
}^{d}$, for instance to $(\mathbf{S}^{u}\cup \mathbf{S}^{d})\cap \{I\in[0,1]%
\}$, then both assumptions in Condition \hyperlink{cond:A3}{\bf A3} are automatically satisfied.

The theorem below will be used in the subsequent section to prove Theorem~%
\ref{th:diffusion-process}.

\begin{theorem}
\label{th:diffusion-process-all-strip} Assume that the IFS $\mathscr{F}%
_{\varepsilon }$ satisfies the conditions \hyperlink{cond:A2}{\bf A2} and \hyperlink{cond:A3}{\bf A3}. Let $I=H_{0}-h_{0}$%
, and let $\mu ,\sigma ,X_{0}$ be any fixed real numbers and let $\gamma >3/2
$. Then for sufficiently small $\varepsilon $ there is a set $\Omega
_{\varepsilon }\subset \mathbb{R}^{n_{u}}\times \mathbb{R}^{n_{s}}\times 
\mathbb{R}^{2}\times \mathbb{T}$ of positive Lebesgue measure, such that $%
\pi _{y,I,\theta }\Omega _{\varepsilon }=\pi _{y,I,\theta }(\mathbf{S}%
^{u}\cup \mathbf{S}^{d})\cap \{I\in \lbrack X_{0}-\varepsilon
,X_{0}+\varepsilon ]\}$, and the stochastic processes defined as energy
paths \footnote{%
Alternatively we can define the energy paths as $X_{t}^{\varepsilon
}(z):=I(\Phi _{\gamma \varepsilon ^{-3/2}t}^{\varepsilon }(z))$, by taking $%
\gamma \in \mathbb{R}$ which satisfies $\gamma >T_{0}\sigma /c$, where $c$
and $T_{0}$ are the constants from the conditions \hyperlink{cond:A2}{\bf A2} and \hyperlink{cond:A3}{\bf A3},
respectively. We highlight this alternative in footnotes while giving the
proof.} 
\begin{eqnarray}
X_{t}^{\varepsilon } &:&\Omega _{\varepsilon }\rightarrow \mathbb{R},
\label{eq:Xt-def} \\
X_{t}^{\varepsilon }(z) &:&=I\left( \Phi _{\varepsilon ^{-\gamma
}t}^{\varepsilon }(z)\right) ,\qquad \text{for }t\in \left[ 0,1\right] , 
\notag
\end{eqnarray}%
have the following limit in distribution on $C[0,1]$ 
\begin{equation}
\lim_{\varepsilon \rightarrow 0}X_{t}^{\varepsilon }=X_{0}+\mu t+\sigma
W_{t}.  \label{eq:limit-diffusion-process}
\end{equation}
\end{theorem}

For the proof of Theorem \ref{th:diffusion-process-all-strip} we will need
two technical lemmas.

%The first lemma below says that for every $Q^{\varepsilon }$-disc $% h^{\varepsilon }$ which is $O(\varepsilon )$-close to a given $I$-level set $%I_{0}$ we can find a sub-disc $D$ whose image under the flow $\Phi_{t}^{\varepsilon }$ is $O(\varepsilon )$-close to $I_{0}$, and which is propagated to another $O(\varepsilon )$-close to $I_{0}$ $Q^{\varepsilon }$-disc $h^{\varepsilon \prime }$.

\begin{lemma}
\label{lem:propagation-const} Assume that the IFS $\mathscr{F}_{\varepsilon
} $ satisfies the conditions \hyperlink{cond:A2}{\bf A2} and \hyperlink{cond:A3}{\bf A3}, and let $\eta $ be a constant as
given by Theorem \ref{th:symbolic-dynamics-general}. Then there exist
constants $T>0$ and $M>1$, such that for every $k\in \mathbb{N}$, every $%
\varepsilon \in (0,\varepsilon _{0}]$ and every $Q^{\varepsilon }$-disc $%
h^{\varepsilon } $ in $\mathbf{S}^{u}\cup \mathbf{S}^{d}$ satisfying 
\begin{equation*}
\pi _{I}h^{\varepsilon }\subset \left[ I_{0}-\varepsilon \eta
,I_{0}+\varepsilon \eta \right] ,
\end{equation*}%
there exists a sequence $\ell _{1},\ldots ,\ell _{m}\in L$ ($m$ depends on $%
k,\varepsilon $ and $h^{\varepsilon }$) and a disc $D\subset \bar{B}^{u}$
such that 
\begin{equation*}
h^{\varepsilon \prime }=F_{\ell _{m},\varepsilon }\circ \ldots \circ F_{\ell
_{1},\varepsilon }\circ h^{\varepsilon }(D)
\end{equation*}%
is a $Q^{\varepsilon }$-disc satisfying 
\begin{equation*}
\begin{split}
\pi _{I}h^{\varepsilon \prime }\subset & \left( I_{0}-\varepsilon
M,I_{0}+\varepsilon M\right) ,\text{ and } \\
\pi _{t}h^{\varepsilon \prime }-\pi _{t}h^{\varepsilon }\subset & \left[
kT,(k+1)T\right] .
\end{split}%
\end{equation*}%
Moreover, for every $z\in h^{\varepsilon }\left( D\right) $ and $\tilde{z}%
=F_{l_{m},\varepsilon }\circ \ldots \circ F_{l_{1},\varepsilon }(z)\in
h^{\varepsilon \prime }$ we have%
\begin{equation*}
\pi _{I}\Phi _{t}(z)\in \left( -\varepsilon M,\varepsilon M\right) \qquad 
\text{for every }t\in \left[ 0,\pi _{t}\tilde{z}-\pi _{t}z\right] .
\end{equation*}
\end{lemma}

\begin{figure}[t]
\begin{center}
\includegraphics[width=8.5cm]{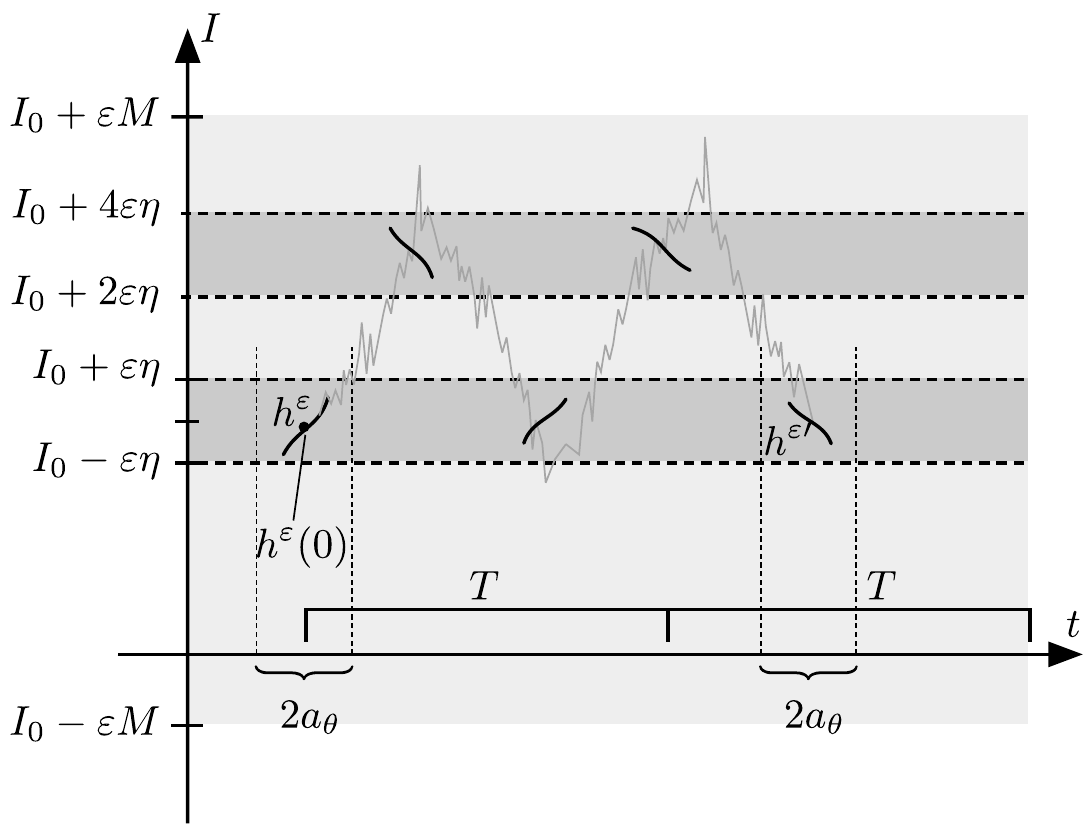}
\end{center}
\caption{The $Q^{\protect\varepsilon}$-disc $h^{\protect\varepsilon}$ is
moved to $h^{\protect\varepsilon \prime}$, by constructing successive $Q^{%
\protect\varepsilon}$-discs which travel up and down between the two regions
in darker grey. Any trajectory which passes through the full excursion stays
within the light grey area. This plot is a sketch of what is happening in
Lemma \protect\ref{lem:propagation-const} for $k=1$.}
\label{fig:discs-travel-const}
\end{figure}

\begin{proof}
The idea behind the proof is depicted in Figure~\ref{fig:discs-travel-const}.

Consider a sequence $\{I^{n}\}_{n\in \mathbb{N}}$ of $I$-level sets whose
values alternate between $I_{0}$ and $I_{0}+3\eta $, i.e., 
\begin{equation}
I_{0},I_{0}+3\eta \varepsilon ,I_{0},I_{0}+3\eta \varepsilon ,\ldots
\label{eq:energy-sequence-for-staying-put}
\end{equation}%
Note that the consecutive values are more than $(2\eta \varepsilon )$ apart.
By Theorem \ref{th:symbolic-dynamics-general} there exists a trajectory
whose $I$-coordinate $(\varepsilon \eta )$-shadows $\{I^{n}\}$. Following
the construction from the proof of Theorem~\ref{th:symbolic-dynamics-general}
we start with a $Q^{\varepsilon }$-disc $h^{\varepsilon }$ whose $I$%
-projection is contained in~$[I_{1}-\eta \varepsilon ,I_{1}+\eta \varepsilon
]$ and take successive iterates of this disc under the IFS $\mathscr{F}%
_{\varepsilon }$.

%When an iterate of the disc reaches $I^{n}$, we choose a sub-disc $h^{\varepsilon }$ whose $I$-projection falls into $\left[I^{n}-\eta \varepsilon ,I^{n}+\eta \varepsilon \right] $, for $n\geq 1$.

%In our case we move with $I$ up and down as we visit successive $I^{\sigma}$.
In order to move from $I^{n}$ to $I^{n+1}$, we make repeated transitions
from $\mathbf{S}^{u}$ to $\mathbf{S}^{u}$, in the case $I^{n}<I^{n+1}$, or
from $\mathbf{S}^{d}$ to $\mathbf{S}^{d}$, in the case $I^{n}>I^{n+1}$. Each
transition corresponds to one application of some return map $%
F_{\ell,\varepsilon }$ to $\mathbf{S}^{u}\cup \mathbf{S}^{d}$. From the
lower bounds from Condition \hyperlink{cond:A2.2.ii}{(A2.2.{\scriptsize II})} there exists an upper bound $K$ for
the number of transitions necessary to move from $I^{n}$ to $I^{n+1}$, where 
$K$ is independent of $\varepsilon $ and of the initial $Q^{\varepsilon }$%
-disc $h^{\varepsilon }$. By Condition \hyperlink{cond:A3.1}{(A3.1)}, the time $t$ for each
transition is upper bounded by $T_{0}$. Hence the time to move from $I^{n}$
to $I^{n+1}$ is upper bounded by $T_{1}:=KT_{0}$, where $T_{1}$ is
independent of $\varepsilon $ and of the initial $Q^{\varepsilon }$-disc $%
h^{\varepsilon }$.

By Condition \hyperlink{cond:A3.2}{(A3.2)} there exists $M_{0}>0$ such that the change in $I$ along
the flow $\Phi^\varepsilon_t$ between successive returns to $\mathbf{S}%
^{u}\cup\mathbf{S}^{d}$ is upper bounded by $M_{0}\varepsilon $. Therefore,
the change in $I$ along the flow as we move from $I^{n}$ to $I^{n+1}$ is
upper bounded by $\varepsilon KM_0$. We note that the change in $I$ by each
return map $F_{\ell,\varepsilon}$ is at most $5\eta\varepsilon$, since each
level set $I^{n}$ is visited within a distance of $\varepsilon \eta $, and
two successive level sets $I^n$ and $I^{n+1}$ are $3\eta\varepsilon$ apart.
Note that by Condition \hyperlink{cond:A3.2}{(A3.2)} we must have $5\varepsilon\eta< \varepsilon KM_0
$.

We let $M$ be an number bigger than $KM_{0}$ and bigger than one. Thus, when
we move $Q^{\varepsilon }$-discs between $I^{n}$ and $I^{n+1}$ the $I$%
-coordinate along the trajectory stays in $\left( I_{0}-\varepsilon
M,I_{0}+\varepsilon M\right) $.

For each $Q^{\varepsilon }$-disc in our construction by (\ref%
{eq:cone-slopes-assumption}) its $t$-projection has diameter less than or
equal to $2a_{\theta }$. Let $T_{2}=2T_{1}=2KT_{0}$ which is an upper bound
for the time needed to move from $I^{n}$ to $I^{n+1}$ and then to $I^{n+2}$. 
%This is possible due to the upper bounds in (\ref{eq:move-a}--\ref{eq:move-b}).
Let $T=T_{2}+4a_{\theta }$.

Our result now follows from the following construction. We start with our $%
Q^{\varepsilon }$-disc $h^{\varepsilon }$. This disc is already in $%
[I^{1}-\varepsilon \eta ,I^{1}+\varepsilon \eta ]=[I_{0}-\varepsilon \eta
,I_{0}+\varepsilon \eta ]$, so we move this disc up in $I$ (using
transitions from $\mathbf{S}^{u}$ to $\mathbf{S}^{u}$) to reach $%
[I^{2}-\varepsilon \eta ,I^{2}+\varepsilon \eta ]=[I_{0}+2\varepsilon \eta
,I_{0}+4\varepsilon \eta ]$, and then down in $I$ (using transitions from $%
\mathbf{S}^{d}$ to $\mathbf{S}^{d}$) to return to $[I^{3}-\varepsilon \eta
,I^{3}+\varepsilon \eta ]=[I_{0}-\varepsilon \eta ,I_{0}+\varepsilon \eta ]$%
. We repeat such trips until we obtain a $Q^{\varepsilon }$-disc $%
h^{\varepsilon \prime }$ with $\pi _{t}h^{\varepsilon \prime }>kT+\pi
_{t}h^{\varepsilon }(0)$, where $k$ is given in the statement of the lemma.
The last transition to reach $h^{\varepsilon \prime }$ takes less time than $%
T_{2}$ and the $t$-projection of every $Q^{\varepsilon }$-disc has diameter
no bigger than $2a_{\theta }$. This implies $\pi _{t}h^{\varepsilon \prime
}<(k+1)T+\pi _{t}h^{\varepsilon }$. From our construction we already know
that the $I$-coordinate along the trajectory of the flow for the points that
start in $h^{\varepsilon }$ and arrive in $h^{\varepsilon \prime }$ is
between $\left( I_{0}-\varepsilon M,I_{0}+\varepsilon M\right) $. We define $%
D$ as the set of points that start in $h^{\varepsilon }$ and arrive in $%
h^{\varepsilon \prime }$ through this construction. The $I$-coordinate along
the trajectory of the flow for the points that start in $h^{\varepsilon }$
and arrive in $h^{\varepsilon \prime }$ is between $\left( I_{0}-\varepsilon
M,I_{0}+\varepsilon M\right) $. This concludes the proof. 
\end{proof}

%The second lemma below says that for every $Q^{\varepsilon }$-disc $%h^{\varepsilon }$ whose time coordinate lies in $t_{0}+\left[ 0,2T\right] $ and which is $O(\varepsilon )$-close to a given $I$-level set $I_{0}$ we can find a sub-disc $D$ whose image by the flow at time $t_{0}+\varepsilon ^{-\gamma +1}$ is $O(\varepsilon )$-close to $I_{0}+s\sqrt{\eps}$, where the constant $s$ is of our choosing. Moreover, the sub-disc $D$ is propagated to another $Q^{\varepsilon }$-disc $h^{\varepsilon \prime }$ which falls within the time interval $t_{0}+\varepsilon ^{-\gamma +1}+[0,2T]$.

\begin{lemma}
\label{lem:discs-travel-up} Assume that the IFS $\mathscr{F}_{\varepsilon }$
satisfies the conditions \hyperlink{cond:A2}{\bf A2} and \hyperlink{cond:A3}{\bf A3}. Let $T$ and $M$ be constants as in Lemma %
\ref{lem:propagation-const}. Let $\gamma >3/2$ and $s\in \mathbb{R}$ be
fixed constants.

Then, for every sufficiently small $\varepsilon >0$, and every $%
Q^{\varepsilon }$-disc $h^{\varepsilon }$ satisfying 
\begin{align*}
\pi _{t}h^{\varepsilon }& \subset t_{0}+\left[ 0,2T\right] , \\
\pi _{I}h^{\varepsilon }& \subset \left( I_{0}-\varepsilon
M,I_{0}+\varepsilon M\right) ,
\end{align*}%
there exists $\ell _{1},\ldots ,\ell _{m}\in L$ and a set $D\subset \bar{B}%
^{u}$ such that%
\begin{equation}
\pi _{I}\{\Phi _{t_{0}+\varepsilon ^{-\gamma +1}-\pi _{t}z}^{\varepsilon
}(z):z\in h^{\varepsilon }\left( D\right) \}\subset I_{0}+s\sqrt{\varepsilon 
}+\left( -\varepsilon M,\varepsilon M\right) ,
\label{eq:flow-in-I-level-at-time}
\end{equation}%
and%
\begin{equation*}
h^{\varepsilon \prime }=F_{\ell _{m},\varepsilon }\circ \ldots \circ F_{\ell
_{1},\varepsilon }\circ h^{\varepsilon }(D)
\end{equation*}
is a $Q^{\varepsilon }$-disc satisfying \footnote{%
If we consider instead the energy paths $X_{t}^{\varepsilon }(z):=I(\Phi
_{\gamma \varepsilon ^{-3/2}t}^{\varepsilon }(z))$, then here we require
that $\pi _{t}h^{\prime }\subset t_{0}+\gamma \varepsilon ^{-1/2}+\left[ 0,2T%
\right]$, by taking $\gamma \in \mathbb{R}$ such that $\gamma >T_{0}s/c$.
This enters in (\ref{eq:place-where-gamma-appears}) during the course of the
proof.} 
\begin{align*}
\pi _{t}h^{\varepsilon \prime }& \subset t_{0}+\varepsilon ^{-\gamma +1}+ 
\left[ 0,2T\right] , \\
\pi _{I}h^{\varepsilon \prime }& \subset I_{0}+s\sqrt{\varepsilon }+\left(
-\varepsilon M,\varepsilon M\right) .
\end{align*}
\end{lemma}

\begin{figure}[t]
\begin{center}
\includegraphics[width=6cm]{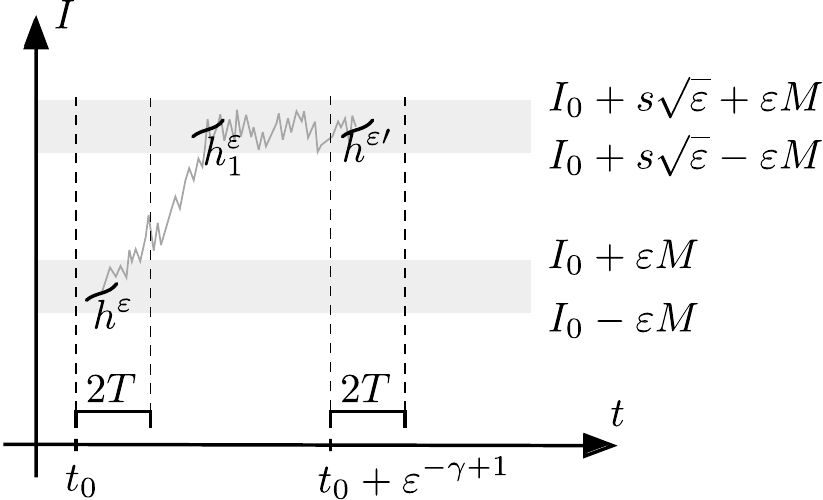}
\end{center}
\caption{A schematic plot for Lemma \protect\ref{lem:discs-travel-up} with $%
c>0$. The disc $h^{\protect\varepsilon}$ is moved to $h^{\protect\varepsilon %
\prime}$, by being moved up and then being propagated within the upper grey
region. }
\label{fig:discs-travel-up}
\end{figure}

\begin{proof}
The idea behind the proof is depicted in Figure~\ref{fig:discs-travel-up}.

Here we conduct the proof for $s>0$. For $s<0$ the proof will be analogous
and we comment on the difference at the end of the proof. If $s=0$ the
result follows directly from Lemma \ref{lem:propagation-const}.

Let $\eta $ be a constant as given by Theorem \ref%
{th:symbolic-dynamics-general}. Consider a sequence $\{I^{n}\}_{n\in \mathbb{%
N}}$ of $I$-level sets whose values alternate between $I_{0}+s\sqrt{\eps}$
and and $I_{0}+s\sqrt{\eps}+3\eta \varepsilon $, i.e., 
\begin{equation}
I^{1}=I_{0}+s\sqrt{\eps},~I^{2}=I_{0}+s\sqrt{\eps}+3\eta \varepsilon
,~I^{3}=I_{0}+s\sqrt{\eps},~I^{4}=I_{0}+s\sqrt{\eps}+3\eta \varepsilon
,~\ldots  \label{eq:energy-sequence-for-going-up}
\end{equation}

By Theorem \ref{th:symbolic-dynamics-from-covering} we know that we can $%
(\varepsilon \eta )$-shadow this sequence. To reach $I^{1}=I_{0}+c\sqrt{\eps}
$ we move $h^{\varepsilon }$ up using transitions from $\mathbf{S}^{u}$ to $%
\mathbf{S}^{u}$. We reach this level by moving $h^{\varepsilon }$ to a $%
Q^{\varepsilon }$-disc which we denote here as $h_{1}^{\varepsilon }$ (see
Figure \ref{fig:discs-travel-up}). By Condition \hyperlink{cond:A2.2.ii}{(A2.2.{\scriptsize II})}, this requires no
more than $s\sqrt{\varepsilon }/\left( c\varepsilon \right) $ applications
of maps $F_{\ell ,\varepsilon }$. Recall that $T_{0}$ is an upper bound on
the time along the flow to follow an application of a map $F_{\ell
,\varepsilon }$. The time needed to reach $h_{1}^{\varepsilon }$ is less
than or equal to $(T_{0}s/c)\varepsilon ^{-1/2}$. Since $\gamma > \frac{3}{2}%
,$ for sufficiently small $\varepsilon $ we have 
\begin{equation}
\varepsilon ^{-\gamma +1}-(T_{0}s/c)\varepsilon ^{-1/2}=\varepsilon
^{-1/2}\left( \varepsilon ^{-\gamma +\frac{3}{2}}-(T_{0}s/c)\right) >0.
\label{eq:place-where-gamma-appears}
\end{equation}%
This means that we reach $h_{1}^{\varepsilon }$ in a time less than $%
\varepsilon ^{-\gamma +1}$. We now apply Lemma \ref{lem:propagation-const}
which ensures that we can stay in $I_{0}+s\sqrt{\varepsilon }+\left(
-\varepsilon M,\varepsilon M\right) $ as we shadow $I^{n}$. By Lemma \ref%
{lem:propagation-const} we can move $h_{1}^{\varepsilon }$ to obtain a $%
Q^{\varepsilon }$-disc $h^{\varepsilon \prime }$ with%
\begin{equation*}
\pi _{t}h^{\varepsilon \prime }\subset \left[ \pi _{t}h_{1}^{\varepsilon
}(0)+kT,\pi _{t}h_{1}^{\varepsilon }(0)+(k+1)T\right] ,
\end{equation*}%
with any $k\in \mathbb{N}$ of our choosing; this interval has length $T$. We
can choose $k$ so that above interval is within the interval $%
t_{0}+\varepsilon ^{-\gamma +1}+\left[ 0,2T\right] $, which has length $2T$.
By Lemma \ref{lem:propagation-const} we also know that the trajectories of
points leading from $h_{1}^{\varepsilon }$ to $h^{\varepsilon \prime }$
remain in $I_{0}+s\sqrt{\varepsilon }+\left( -\varepsilon M,\varepsilon
M\right) $. This in particular implies that at time $t_{0}+\varepsilon
^{-\gamma +1}$ (see Figure \ref{fig:discs-travel-up}) we obtain the
condition (\ref{eq:flow-in-I-level-at-time}). This concludes the proof for $%
s>0$.

For $s<0$ the argument is almost identical, with the only difference that to
reach $I^{1}$ instead of going up we move down using transitions from $%
\mathbf{S}^{d}$ to $\mathbf{S}^{d}$.
\end{proof}

\begin{figure}[tbp]
\begin{center}
\includegraphics[height=4cm]{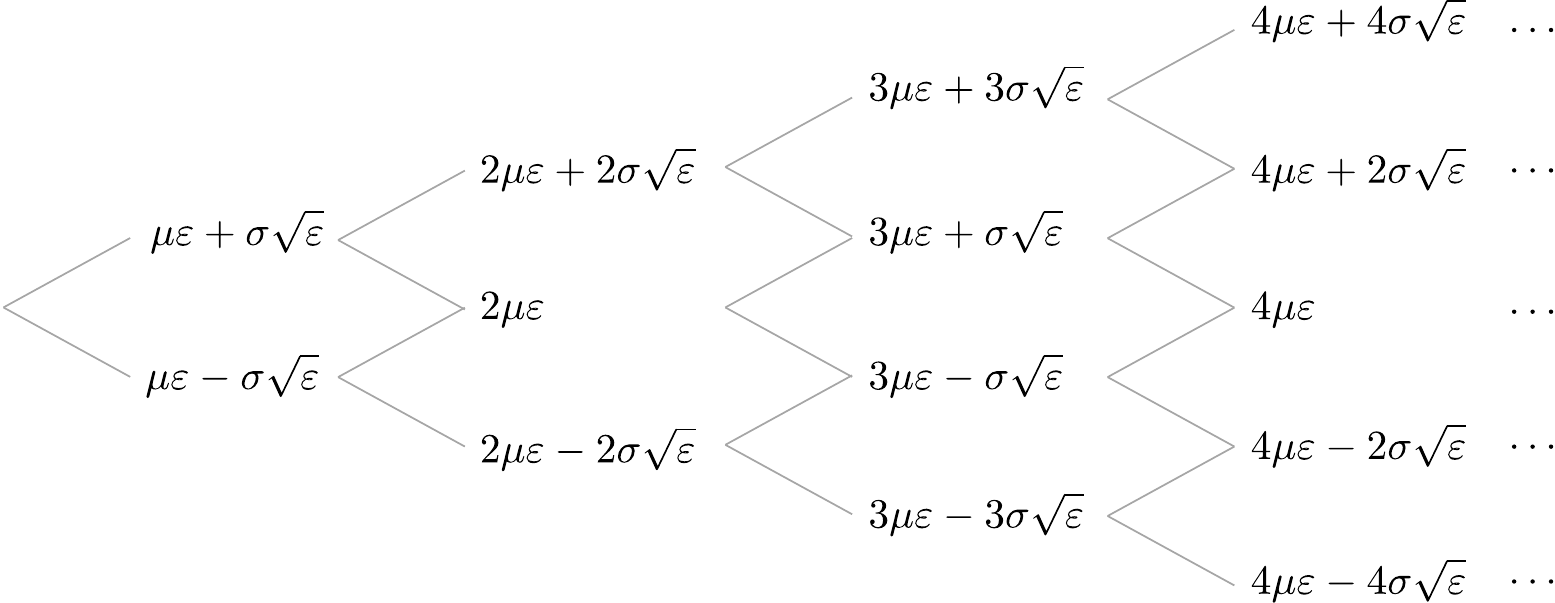}
\end{center}
\caption{The idea behind the proof of Theorem \protect\ref%
{th:diffusion-process-all-strip} is to shadow the changes of $I$ so that
they would follow a random walk from the above tree.}
\label{fig:tree}
\end{figure}

%For our proof we shall use also the following simple technical lemma.

%\begin{lemma}
%\label{lem:technical-convergence} Consider two sequences of
%stochastic processes $Z_{n},S_{n}^{\ast }:\Omega _{n}\rightarrow C\left[ 0,1%
%\right] $ on a sequence of probability spaces $(\Omega _{n},\mathbb{F}_{n},%
%\mathbb{P}_{n})$. Assume that the sequence $S_{n}^{\ast }(t)$ converges to $%
%W_{t}$ in distribution on $C[0,1]$ as $n\rightarrow \infty .$ Moreover
%assume that for a fixed constant $c>0$ holds
%\begin{equation*}
%\left\vert Z_{n}(t)(\omega )\right\vert \leq \frac{1}{n}\left\vert
%S_{n}^{\ast }(t)(\omega )\right\vert +cn^{-1/2},\qquad \text{for all }\omega
%\in \Omega _{n}\text{ and all }t\in \left[ 0,1\right] .
%\end{equation*}%
%Then $S_{n}^{\ast }(t)+Z_{n}(t)$ converges to $W_{t}$ in distribution on $%
%C[0,1]$ as $n\rightarrow \infty $.
%\end{lemma}

%\begin{proof}
%We give the proof in Section \ref{sec:technical-convergence} of
%the Appendix.
%\end{proof}

\begin{proof}[Proof of Theorem \protect\ref{th:diffusion-process-all-strip}]

The intuition behind the proof is to define $\Omega _{\varepsilon }$ as a
set of points in $\mathbf{S}^{u}\cup \mathbf{S}^{d}$ for which the $I$%
-coordinate $O(\varepsilon )$-shadows the random walk in Figure~\ref%
{fig:tree}. 
%This random walk is driven by a sequence of independent identically distributed random variables $Y_{i}^{\varepsilon}$ \begin{equation*} Y_{i}^{\varepsilon}=\left\{ \begin{array}{c} \mu\varepsilon+\sqrt{\varepsilon}\sigma\qquad\text{with prob. }\frac{1}{2}, \\ \mu\varepsilon-\sqrt{\varepsilon}\sigma\qquad\text{with prob. }\frac{1}{2}.% \end{array} \right. \end{equation*} By the functional central limit theorem the sum $\sum_{i=1}^{\left\lfloor \varepsilon ^{-1}\right\rfloor }Y_{i}^{\varepsilon}$ converges on $[0,1]$ to  to the Brownian motion with drift $dB_t=\mu dt+\sigma dW_t$, where $W_t$ is the standard Brownian motion. The convergence is relative to the the Skorohod topology. This is just an informal overview. We now present the details.

Let $M$ be the constant from Lemma \ref{lem:propagation-const} and consider
the following two sets\footnote{%
If we consider instead the energy paths $X_{t}^{\varepsilon }(z):=I(\Phi
_{\gamma \varepsilon ^{-3/2}t}^{\varepsilon }(z))$, as indicated in the
footnote for Theorem \ref{th:diffusion-process-all-strip}, then in the
reminder of the proof we have: in $A_{u}^{\varepsilon }$ and $%
A_{d}^{\varepsilon }$ we have $\gamma \varepsilon ^{-1/2}$ instead of $%
\varepsilon ^{-\gamma +1}$; in $A_{uu}^{\varepsilon }$, $A_{ud}^{\varepsilon
}$, $A_{du}^{\varepsilon }$, $A_{dd}^{\varepsilon }$ we have $2\gamma
\varepsilon ^{-1/2}$ instead of $2\varepsilon ^{-\gamma +1}$; in $A_{\kappa
\omega }^{\varepsilon }$ we have $k\gamma \varepsilon ^{-1/2}$ instead of $%
k\varepsilon ^{-\gamma +1}$, with mirror changes carried throughout the
argument.} 
\begin{align*}
A_{u}^{\varepsilon }& =\left\{ z\in \mathbb{R}^{n_{u}+n_{s}+2}\times \mathbb{%
T}:\pi _{I}\Phi _{\varepsilon ^{-\gamma +1}}^{\varepsilon }\left( z\right)
\in X_{0}+\mu \varepsilon +\sigma \sqrt{\varepsilon }+\left( -\varepsilon
M,\varepsilon M\right) \right\} , \\
A_{d}^{\varepsilon }& =\left\{ z\in \mathbb{R}^{n_{u}+n_{s}+2}\times \mathbb{%
T}:\pi _{I}\Phi _{\varepsilon ^{-\gamma +1}}^{\varepsilon }\left( z\right)
\in X_{0}+\mu \varepsilon -\sigma \sqrt{\varepsilon }+\left( -\varepsilon
M,\varepsilon M\right) \right\} .
\end{align*}%
The subscript $u$ in $A_{u}^{\varepsilon }$ is meant to suggest that for $%
z\in A_{u}^{\varepsilon }$ the trajectory moves `up' in $I$. Similarly, $d$
in $A_{d}^{\varepsilon }$ is meant to suggest that for $z\in
A_{d}^{\varepsilon }$ the trajectory moves `down' in $I$.

By Lemma \ref{lem:discs-travel-up} $A_{u}^{\varepsilon }$ and $%
A_{d}^{\varepsilon }$ are nonempty. (We apply the lemma with {$I_{0}=X_{0}$,}
$s=\mu \varepsilon /\sqrt{\varepsilon }+\sigma $ for $A_{u}^{\varepsilon }$,
and with $s=\mu \varepsilon /\sqrt{\varepsilon }-\sigma $ for $%
A_{d}^{\varepsilon }$. We restrict to $\varepsilon <(\sigma /M)^{2}$ so that 
$A_{u}^{\varepsilon }\cap A_{d}^{\varepsilon }=\emptyset $.) We obtain in
fact more. For every fixed $y_{0}\in \bar{B}^{n_{s}}$, {$I_{0}\in
(X_{0}-\varepsilon M,X_{0}+\varepsilon M)$}, and $\theta _{0}\in S_{\theta
}^{u}\cup S_{\theta }^{d}$, choose an initial $Q^{\varepsilon }$-disc $%
h^{\varepsilon }$ of the form 
\begin{equation}
h^{\varepsilon }(x)=\left( x,y_{0},I_{0},\theta _{0}\right) ,\quad x\in \bar{%
B}^{n_{u}}.  \label{eqn:straight_disc}
\end{equation}%
Applying Lemma \ref{lem:discs-travel-up} yields a point $z\in h^{\varepsilon
}$ so that $z\in A_{u}^{\varepsilon }$. For such a point $z$ we have $\pi
_{y,I,\theta }\left( z\right) =\left( y_{0},I_{0},\theta _{0}\right) $. This
implies that $\pi _{y,\theta }\left( A_{u}^{\varepsilon }\right) =\pi
_{y,\theta }\left( \mathbf{S}^{u}\cup \mathbf{S}^{d}\right) $ and $\pi
_{I}\left( A_{u}^{\varepsilon }\right) =${$(X_{0}-\varepsilon
M,X_{0}+\varepsilon M)$}. By mirror arguments $\pi _{y,\theta }\left(
A_{d}^{\varepsilon }\right) =\pi _{y,\theta }\left( \mathbf{S}^{u}\cup 
\mathbf{S}^{d}\right) $ and $\pi _{I}\left( A_{d}^{\varepsilon }\right) =${$%
(X_{0}-\varepsilon M,X_{0}+\varepsilon M)$}.

We now define the following subsets of $\mathbb{R}^{n_{u}+n_{s}+2}\times 
\mathbb{T}$: 
\begin{align*}
A_{uu}^{\varepsilon }& =A_{u}^{\varepsilon }\cap \left\{ z:\pi _{I}\Phi
_{2\varepsilon ^{-\gamma +1}}^{\varepsilon }\left( z\right) \in X_{0}+2\mu
\varepsilon +2\sigma \sqrt{\varepsilon }+\left( -M\varepsilon ,M\varepsilon
\right) \right\} , \\
A_{du}^{\varepsilon }& =A_{u}^{\varepsilon }\cap \left\{ z:\pi _{I}\Phi
_{2\varepsilon ^{-\gamma +1}}^{\varepsilon }\left( z\right) \in X_{0}+2\mu
\varepsilon +\left( -M\varepsilon ,M\varepsilon \right) \right\} , \\
A_{ud}^{\varepsilon }& =A_{d}^{\varepsilon }\cap \left\{ z:\pi _{I}\Phi
_{2\varepsilon ^{-\gamma +1}}^{\varepsilon }\left( z\right) \in X_{0}+2\mu
\varepsilon +\left( -M\varepsilon ,M\varepsilon \right) \right\} , \\
A_{dd}^{\varepsilon }& =A_{d}^{\varepsilon }\cap \left\{ z:\pi _{I}\Phi
_{2\varepsilon ^{-\gamma +1}}^{\varepsilon }\left( z\right) \in X_{0}+2\mu
\varepsilon -2\sigma \sqrt{\varepsilon }+\left( -M\varepsilon ,M\varepsilon
\right) \right\} .
\end{align*}%
The set $A_{uu}^{\varepsilon }$ is the set of points which first go up to $%
\mu \varepsilon +\sigma \sqrt{\varepsilon }$ at time $\varepsilon ^{-\gamma
+1}$, and then go up again to $2\mu \varepsilon +2\sigma \sqrt{\varepsilon }$
at time $2\varepsilon ^{-\gamma +1}$. This is the reason for the subscript $%
uu$. The set $A_{du}^{\varepsilon }$ is the set of points which first go up $%
\mu \varepsilon +\sigma \sqrt{\varepsilon }$ at time $\varepsilon ^{-\gamma
+1}$, and then go down to $2\mu \varepsilon $ at time $2\varepsilon
^{-\gamma +1}$. This is the reason for the subscript $du$. Similar
descriptions can me made for $A_{ud}^{\varepsilon }$ and $%
A_{dd}^{\varepsilon }$.

We apply Lemma \ref{lem:discs-travel-up} twice for $Q^{\varepsilon }$-discs
of the form \eqref{eqn:straight_disc}, obtaining that the sets $%
A_{uu}^{\varepsilon }$, $A_{ud}^{\varepsilon }$, $A_{du}^{\varepsilon }$, $%
A_{dd}^{\varepsilon }$ are nonempty. Due to \eqref{eqn:straight_disc}, we
have that $\pi _{y,\theta }\left( A_{\kappa \kappa ^{\prime }}^{\varepsilon
}\right) =\pi _{y,\theta }\left( \mathbf{S}^{u}\cup \mathbf{S}^{d}\right) $,
and $\pi _{I}\left( A_{\kappa \kappa ^{\prime }}^{\varepsilon }\right) =${$%
(X_{0}-\varepsilon M,X_{0}+\varepsilon M)$ }for $\kappa ,\kappa ^{\prime
}\in \{u,d\}$.

We continue to subdivide these sets in a similar manner using an inductive
procedure. First we introduce the following notation. For a given sequence
of symbols $\omega =\omega _{k}\ldots \omega _{1},$ where $\omega _{i}\in
\left\{ u,d\right\} $, we denote 
\begin{align}
\left\vert \omega \right\vert & :=k,  \notag \\
U\left( \omega \right) & :=\#\left\{ \omega _{i}:\omega _{i}=u,\,i=1,\ldots
,\left\vert \omega \right\vert \right\} ,  \notag \\
D\left( \omega \right) & :=\#\left\{ \omega _{i}:\omega _{i}=d,\,i=1,\ldots
,\left\vert \omega \right\vert \right\} ,  \notag \\
N\left( \omega \right) & :=U\left( \omega \right) -D\left( \omega \right) .
\label{eq:N-omega-stochastic}
\end{align}%
For each string $\omega $, $\left\vert \omega \right\vert $ represents the
length of $\omega $, $U\left( \omega \right) $ the number of steps `up', $%
D\left( \omega \right) $ the number of steps `down', and $N\left( \omega
\right) $ the `net' number of `up-down' steps.

For a given string $\omega =\omega _{k-1}\ldots \omega _{1}$ and $\kappa \in
\left\{ u,d\right\} $ we now inductively define $A_{\kappa \omega
}^{\varepsilon }\subset \mathbb{R}^{n_{u}+n_{s}+2}\times \mathbb{T}$ as%
\begin{equation}
A_{\kappa \omega }^{\varepsilon }=A_{\omega }^{\varepsilon }\cap \{z:\pi
_{I}\Phi _{k\varepsilon ^{-\gamma +1}}^{\varepsilon }\left( z\right) \in
X_{0}+k\mu \varepsilon +N\left( \kappa \omega \right) \sigma \sqrt{%
\varepsilon }+\left( -M\varepsilon ,M\varepsilon \right) \}.
\label{eq:Ak-omega}
\end{equation}%
By construction, the sets $A_{\omega }^{\varepsilon }$ with $\omega $ of the
same length are disjoint, i.e., if $\left\vert \omega \right\vert
=\left\vert \varpi \right\vert $ and $\omega \neq \varpi $ then $A_{\omega
}^{\varepsilon }\cap A_{\varpi }=\emptyset $. Also because we take the
points in $A_{\omega }^{\varepsilon }$ by inductively applying Lemma \ref%
{lem:discs-travel-up}, for $t\in \left[ \varepsilon k,\varepsilon \left(
k+1\right) \right] $ we obtain (see Figure \ref{fig:discs-travel-up}) 
\begin{equation}
\Phi _{\varepsilon ^{-\gamma }t}^{\varepsilon }(z)\in \Phi _{\varepsilon
^{-\gamma +1}k}^{\varepsilon }(z)+\left[ -2\sigma \sqrt{\varepsilon }%
,2\sigma \sqrt{\varepsilon }\right] .  \label{eq:Phi-t-to-k}
\end{equation}%
%
%
%
%
%
%
%
%
%
%By the continuity of the flow with respect to the initial conditions, if $%
%z\in A_{\omega }^{\varepsilon }$, then there exists an open neighborhood of $%
%z$ which is contained in $A_{\omega }^{\varepsilon }$. We can choose the
%sets $A_{\omega }^{\varepsilon }$ to be of the form
%\begin{equation}  \label{eqn:A_omega_form}
%A_{\omega }^{\varepsilon }=\bigcup_{(y_0, I_0, \theta_0)\in\pi _{y,I,\theta
%}\left(\mathbf{S}^{u}\cup\mathbf{S}^{d}\right)} B\left(x(y_0,
%I_0,\theta_0),\delta(y_0, I_0,\theta_0)\right)\times
%\{y_0\}\times\{I_0\}\times\{\theta_0\},
%\end{equation}
%where $z=(x_0(y_0, I_0,\theta_0),y_0,I_0,z_0)$ and $\delta(y_0,
%I_0,\theta_0)>0$ depend continuously on $(y_0, I_0, \theta_0)$. This means
%that for any $\omega $ we have $\mu(A_{\omega }^{\varepsilon }) >0$, where $%
%\mu$ is the Lebesgue measure.

Let 
\begin{equation}
K_{\varepsilon }:=\left\lfloor \varepsilon ^{-1}\right\rfloor \in \mathbb{N}.
\label{eq:Keps-choice}
\end{equation}%
Let us consider paths of length $\left\vert \omega \right\vert
=K_{\varepsilon }$ and let us define the set $B=\left( \mathbf{S}^{u}\cup 
\mathbf{S}^{d}\right) \cap \left\{ I\in \left[ X_{0}-\varepsilon
,X_{0}+\varepsilon \right] \right\} $. The set $B$ is compact and has finite
Lebesgue measure. We will show that for a fixed $\varepsilon $ we can
restrict the sets $A_{\omega }^{\varepsilon }$ to $\tilde{A}_{\omega
}^{\varepsilon }\subset A_{\omega }^{\varepsilon }\cap B$, so that 
\begin{align}
0<\mu \left( \tilde{A}_{\omega }^{\varepsilon }\right) & =\mu \left( \tilde{A%
}_{\varpi }^{\varepsilon }\right) <\infty \quad \text{for any }\omega
,\varpi \text{ with }\left\vert \omega \right\vert =\left\vert \varpi
\right\vert =K_{\varepsilon },  \label{eq:Aw-same-measure} \\
\pi _{y,I,\theta }\left( \tilde{A}_{\omega }^{\varepsilon }\right) & =\pi
_{y,\theta }\left( \mathbf{S}^{u}\cup \mathbf{S}^{d}\right) \cap \{I\in %
\left[ X_{0}-\varepsilon ,X_{0}+\varepsilon \right] \}.
\label{eq:Aw-same-measure-3}
\end{align}

In order to show (\ref{eq:Aw-same-measure}--\ref{eq:Aw-same-measure-3}) let
us fix $\varepsilon $ and first consider a fixed $\omega $ of length $%
\left\vert \omega \right\vert =K_{\varepsilon }$. Since $M>1$, $\left[
X_{0}-\varepsilon ,X_{0}+\varepsilon \right] \subset \left(
X_{0}-M\varepsilon ,X_{0}+M\varepsilon \right) $, so we know that for every $%
(y_{0},I_{0},\theta _{0})\in \pi _{y,I,\theta }B$ there exists an $x_{0}$
such that $\left( x_{0},y_{0},I_{0},\theta _{0}\right) \in A_{\omega
}^{\varepsilon }$. By the continuity of the solutions of ODEs with respect
to the initial conditions we can choose a small neighborhood of $\left(
x_{0},y_{0},I_{0},\theta _{0}\right) $ which will also be in $A_{\omega
}^{\varepsilon }$. We can take this neighborhood to be of the form $%
B^{n_{u}}\left( x_{0},\delta \right) \times B^{n_{s}+2}\left( \left(
y_{0},I_{0},\theta _{0}\right) ,r\right) \cap B$, where $\delta ,r$ are some
small positive numbers depending on $\omega $ and $y_{0},I_{0},\theta _{0}$.
By the compactness of $\pi _{y,I,\theta }B$ we can choose a finite sequences 
$\left( y_{i},I_{i},\theta _{i}\right) $ and $r_{i}$, for $i=1,\ldots ,n$,
such that 
\begin{equation}
\pi _{y,I,\theta }B=\bigcup\limits_{i=1}^{n}B^{n_{s}+2}\left( \left(
y_{i},I_{i},\theta _{i}\right) ,r_{i}\right) \cap \pi _{y,I,\theta }B.
\label{eq:Aw-projections-condition}
\end{equation}%
Each $\left( y_{i},I_{i},\theta _{i}\right) $ has a $\delta _{i}=\delta
\left( y_{i},I_{i},\theta _{i}\right) $ associated with it. For $\delta
<\min_{i=1,\ldots ,n}\delta _{i}$ we define%
\begin{equation}
\tilde{A}_{\omega }^{\varepsilon }\left( \delta \right) :=\left(
\bigcup\limits_{i=1}^{n}B^{n_{u}}\left( x_{0},\delta \right) \times
B^{n_{s}+2}\left( \left( y_{i},I_{i},\theta _{i}\right) ,r_{i}\right)
\right) \cap B\subset A_{\omega }^{\varepsilon }.
\label{eq:A-tilde-construction}
\end{equation}%
The measure of this set depends continuously on $\delta $. There is a finite
number of $\omega $ of length $K_{\varepsilon }$. This means that we can
choose a finite sequence of $\delta \left( \omega \right) $ such that $%
\tilde{A}_{\omega }^{\varepsilon }:=\tilde{A}_{\omega }^{\varepsilon }\left(
\delta (\omega )\right) $ satisfy (\ref{eq:Aw-same-measure}). Condition (\ref%
{eq:Aw-same-measure-3}) follows from (\ref{eq:Aw-projections-condition}) and
(\ref{eq:A-tilde-construction}).

We now define the set $\Omega _{\varepsilon }$ that appears in the statement
of Theorem~\ref{th:diffusion-process-all-strip}: 
\begin{equation*}
\Omega _{\varepsilon }:=\bigcup_{\left\vert \omega \right\vert
=K_{\varepsilon }}\tilde{A}_{\omega }^{\varepsilon }.
\end{equation*}%
By construction, $\mu (\Omega _{\varepsilon })>0$. Moreover, by (\ref%
{eq:Aw-same-measure-3}), we also have 
\begin{equation*}
\pi _{y,I,\theta }\left( \Omega _{\varepsilon }\right) =\pi _{y,\theta
}\left( \mathbf{S}^{u}\cup \mathbf{S}^{d}\right) \cap \left\{ I\in \left[
X_{0}+\varepsilon ,X_{0}-\varepsilon \right] \right\} .
\end{equation*}

We now define a sequence of random variables $Y_{n}^{\varepsilon }:\Omega
_{\varepsilon }\rightarrow \mathbb{R}$, for $n=1,\ldots ,K_{\varepsilon }$ as%
\begin{equation}
Y_{n}^{\varepsilon }\left( z\right) =\left\{ 
\begin{array}{r}
1\qquad \text{if }z\in \tilde{A}_{\omega }^{\varepsilon }\text{ and }\omega
_{n}=u, \\ 
-1\qquad \text{if }z\in \tilde{A}_{\omega }^{\varepsilon }\text{ and }\omega
_{n}=d.%
\end{array}%
\right.  \label{eq:Yi-for-stochastic}
\end{equation}%
Since the sets $\tilde{A}_{\omega }^{\varepsilon }$ satisfy (\ref%
{eq:Aw-same-measure}), the random variables $Y_{n}^{\varepsilon }$ are
independent and identically distributed with $\mathbb{P}_{\varepsilon
}(Y_{n}^{\varepsilon }=1)=\mathbb{P}_{\varepsilon }(Y_{n}^{\varepsilon }=-1)=%
\frac{1}{2}$, moreover $\mathbb{E}\left( Y_{n}^{\varepsilon }\right) =0,$ $%
Var\left( Y_{n}^{\varepsilon }\right) =1$.

As in the setup for Theorem \ref{th:functional-CLT}, we define $%
S_{n}^{\varepsilon }=Y_{1}^{\varepsilon }+\ldots +Y_{n}^{\varepsilon }$ for $%
n\in \{1,\ldots ,K_{\varepsilon }\}$, define $S^{\varepsilon }\left(
t\right) =S_{\left\lfloor t\right\rfloor }^{\varepsilon }+\left(
t-\left\lfloor t\right\rfloor \right) (S_{\left\lfloor t\right\rfloor
+1}^{\varepsilon }-S_{\left\lfloor t\right\rfloor }^{\varepsilon })$ for $%
t\in \lbrack 0,K_{\varepsilon }]$, and define 
\begin{equation*}
S_{n}^{\varepsilon ,\ast }(t)=\frac{S^{\varepsilon }\left( tn\right) }{\sqrt{%
n}}\qquad \text{for }t\in \left[ 0,1\right] \text{ and }n\in \{1,\ldots
,K_{\varepsilon }\}.
\end{equation*}

Since $Y_{n}^{\varepsilon }$ have distribution independent of $\varepsilon $
(for $n\leq K_{\varepsilon }$),$\ $we see that $S_{n}^{\varepsilon ,\ast
}(t) $ also have distribution independent of $\varepsilon $ (for $n\leq
K_{\varepsilon }$), hence by Theorem \ref{th:functional-CLT}%
\begin{equation}
\lim_{\varepsilon \rightarrow 0}S_{K_{\varepsilon }}^{\varepsilon ,\ast
}(t)=W_{t},  \label{eq:limit-S-star}
\end{equation}%
where the limit is in distribution on $C[0,1].$

Our objective will be to use the process $S_{K_{\varepsilon }}^{\varepsilon
,\ast }(t)$ for the proof of convergence of $X_{t}^{\varepsilon }$. Before
that we need to prepare some auxiliary facts.

Firstly, from (\ref{eq:Yi-for-stochastic}) we see that $S_{n}^{\varepsilon
}(z)$ counts the `net' number of `up-down' steps along the trajectory
starting from $z$, so for $z\in \tilde{A}_{\omega }$ with $\left\vert \omega
\right\vert =n\leq K_{\varepsilon }$ we have (see (\ref%
{eq:N-omega-stochastic})) 
\begin{equation}
N\left( \omega \right) =S_{n}^{\varepsilon }(z).  \label{eq:N-vs-z}
\end{equation}%
For $t\in \left[ 0,1\right] $ let $k\in \mathbb{N}$ be a number such that $%
t\in \left[ \varepsilon k,\varepsilon \left( k+1\right) \right] $. We then
have%
\begin{equation}
k\varepsilon \in \left[ t-\varepsilon ,t+\varepsilon \right] ,
\label{eq:mu_k_eps}
\end{equation}%
and also, by (\ref{eq:Keps-choice}) for such $t$ and $k$ it holds that $k\in %
\left[ tK_{\varepsilon }-1,tK_{\varepsilon }+2\right] ,$ hence 
\begin{equation}
S_{k}^{\varepsilon }(z)\in S^{\varepsilon }\left( tK_{\varepsilon }\right) + 
\left[ -2,2\right] .  \label{eq:Sk-to-SK}
\end{equation}%
Since $K_{\varepsilon }=\left\lfloor \varepsilon ^{-1}\right\rfloor ,$ 
\begin{equation}
\sqrt{\varepsilon }\in \frac{1}{\sqrt{K_{\varepsilon }}}+\left[ \frac{-1}{2}%
K_{\varepsilon }^{-3/2},\frac{1}{2}K_{\varepsilon }^{-3/2}\right] .
\label{eq:Sk-eps}
\end{equation}%
We are now ready for our proof of convergence of $X_{t}^{\varepsilon }$. We
compute
\begin{align*}
&X_{t}^{\varepsilon }(z)=\pi _{I}\Phi _{\varepsilon ^{-\gamma}t}^{\varepsilon }(z)    \\ 
&\in \pi _{I}\Phi _{k\varepsilon ^{-\gamma +1}}^{\varepsilon }(z)+\left[-2\sigma \sqrt{\varepsilon },+2\sigma \sqrt{\varepsilon }\right] \hspace{3.52cm} \text{from (\ref{eq:Phi-t-to-k})} \\ 
& \subset X_{0}+k\mu \varepsilon +S_{k}^{\varepsilon }(z)\sigma \sqrt{\varepsilon }+\left[ -M\varepsilon -2\sigma \sqrt{\varepsilon },M\varepsilon+2\sigma \sqrt{\varepsilon }\right]  \quad \text{from (\ref{eq:Ak-omega}),(\ref{eq:N-vs-z})} \\ 
& \subset X_{0}+\mu t+S_{k}^{\varepsilon }(z)\sigma \sqrt{\varepsilon }+\left[-\left( M+\mu \right) \varepsilon -2\sigma \sqrt{\varepsilon },\left( M+\mu\right) \varepsilon +2\sigma \sqrt{\varepsilon }\right]  \\  
& \hspace{9.38cm} \text{from (\ref{eq:mu_k_eps})} \\ 
&\subset X_{0}+\mu t+S_{k}^{\varepsilon }(z)\sigma \sqrt{\varepsilon }+\left[-3\sigma \sqrt{\varepsilon },3\sigma \sqrt{\varepsilon }\right] \hspace{2.54cm} \text{for small }\varepsilon \\ 
&\subset X_{0}+\mu t+S_{k}^{\varepsilon }(tK_{\varepsilon })(z)\sigma \sqrt{\varepsilon }+\left[ -5\sigma \sqrt{\varepsilon },5\sigma \sqrt{\varepsilon }\right] \hspace{1.7cm} \text{from (\ref{eq:Sk-to-SK})} \\ 
&\subset X_{0}+\mu t+\sigma S_{K_{\varepsilon }}^{\varepsilon ,\ast}(t)(z)+\sigma S_{K_{\varepsilon }}^{\varepsilon ,\ast }(t)(z)\left( \frac{-1}{2K_{\varepsilon }},\frac{1}{2K_{\varepsilon }}\right) +\left[ \frac{-6\sigma }{\sqrt{K_{\varepsilon }}},\frac{6\sigma }{\sqrt{K_{\varepsilon }}}\right] \\  
& \hspace{9.38cm} \text{from (\ref{eq:Sk-eps}).}%
\end{align*}%
We have therefore shown that%
\begin{equation}
X_{t}^{\varepsilon }=X_{0}+\mu t+\sigma \left( S_{K_{\varepsilon
}}^{\varepsilon ,\ast }(t)+Z_{\varepsilon }(t)\right)  \label{eq:X-fina-sum}
\end{equation}%
where 
\begin{equation*}
\left\vert Z_{\varepsilon }(t)\right\vert \leq \frac{1}{K_{\varepsilon }}%
\left\vert S_{K_{\varepsilon }}^{\varepsilon ,\ast }(t)\right\vert +\frac{6}{%
\sqrt{K_{\varepsilon }}}.
\end{equation*}%
Since $Z_{\varepsilon }(t)$ converges to zero in probability and $%
S_{K_{\varepsilon }}^{\varepsilon ,\ast }(t)$ converges in distribution to $%
W_{t}$ we obtain that $S_{K_{\varepsilon }}^{\varepsilon ,\ast
}(t)+Z_{\varepsilon }(t)$ converges in distribution to $W_{t}$. This means
that from (\ref{eq:X-fina-sum}) we obtain (\ref{eq:limit-diffusion-process}%
), which concludes our proof. 
\end{proof}

\begin{remark}
By a modification of the proof of Theorem \ref%
{th:diffusion-process-all-strip}, by shadowing a random walk with time
dependent coefficients, we can obtain convergence to $\mu (t)+\sigma (t)W_t$
for deterministic, continuous functions $\mu (t),\sigma (t)$, $t\in[0,1]$.
\end{remark}

\subsection{Stochastic behavior under conditions on connecting sequences}

%Here we give the proof of Theorem \ref{th:diffusion-process}.

\begin{proof}[Proof of Theorem \protect\ref{th:diffusion-process}]
By Theorem \ref{th:symbolic-dynamics-from-covering} we know that we can $%
\eta \varepsilon $ shadow any sequence of $I^{\sigma }\in \left[ 2\eta
\varepsilon ,1-2\eta \varepsilon \right] $. As long as we remain in $\left\{
I\in \left[ 2\eta \varepsilon ,1-2\eta \varepsilon \right] \right\} $ we can
shadow the energy levels from the binomial process from Figure \ref{fig:tree}%
. Our result follows from the same construction as the proof of Theorem \ref%
{th:diffusion-process-all-strip} with the only difference that once $%
X_{t}^{\varepsilon }$ leaves the set $\left\{ I\in \left[ 0,1\right]
\right\} $ the assumption \hyperlink{cond:C2}{\textbf{C2}} can no longer be used to propagate $%
Q^{\varepsilon }$-discs. This is why we stop the considered processes as
soon as we exit this set.

One issue we can comment on is the existence of $T_{0}$ and $M_{0}$ from
Condition \hyperlink{cond:A3}{\bf A3}, which is part of the assumption of Theorem \ref%
{th:diffusion-process-all-strip}. The $T_{0}$ exists, since any map $%
F_{\ell,\varepsilon }$ associated with the connecting sequence with $\ell
\in L$ has a compact domain. We have a finite number of such maps, hence the
time needed for the flow to pass through these maps is finite.

For $T_{0}>0$ since the set 
\begin{equation*}
A\left( T_{0}\right) :=\left[ 0,T_{1}\right] \times \left[ 0,\varepsilon _{0}%
\right] \times \left( (\mathbf{S}^{u}\cup \mathbf{S}^{d})\cap \left\{ I\in %
\left[ 0,1\right] \right\} \right)
\end{equation*}%
is compact, we can find $M_{0}=M_{0}(T_{0})$ such that 
\begin{equation*}
\pi _{I}\Phi _{t}^{\varepsilon }(z)-\pi _{I}z\in \left[ -M_{0}\varepsilon
,M_{0}\varepsilon \right] \qquad \text{for }\left( t,\varepsilon ,z\right)
\in A(T_{0}).
\end{equation*}%
Such $M_{0}$ will be the bound for (\ref{eq:M1-bound}). 
\end{proof}

%TCIDATA{Version=5.00.0.2606}
%TCIDATA{LaTeXparent=0,0,MMFedit.tex}

%\appendix
\section{Appendix}
\subsection{Proof of Lemma \protect\ref{th:param-dep-cones}}
\begin{proof}
To establish
(\ref{eq:param-dep-cone-prop}) it is enough to show that for all $%
\varepsilon \in E$ and $z_{1},z_{2}\in N$ satisfying $Q_{a}^{\varepsilon
}\left( z_{1}-z_{2}\right) >0$ we have 
\begin{align}
\frac{\left\Vert f_{\kappa }(\varepsilon ,z_{1})-f_{\kappa }(\varepsilon
,z_{2})\right\Vert }{\left\Vert f_{x}(\varepsilon ,z_{1})-f_{x}(\varepsilon
,z_{2})\right\Vert }& <b_{\kappa }\qquad \text{for }\kappa \in \left\{
y,\theta \right\} ,  \label{eq:param-cones-objective1} \\
\frac{\left\Vert f_{I}(\varepsilon ,z_{1})-f_{I}(\varepsilon
,z_{2})\right\Vert }{\left\Vert f_{x}(\varepsilon ,z_{1})-f_{x}(\varepsilon
,z_{2})\right\Vert }& <\varepsilon b_{I}.  \label{eq:param-cones-objective2}
\end{align}

Take $Q_{a}^{\varepsilon }\left( z_{1}-z_{2}\right) >0$. Let $%
z_{1}-z_{2}=\left( x,y,I,\theta \right) $. This means that $\left\Vert
x\right\Vert \neq 0$ and 
\begin{equation}
\left\Vert y\right\Vert /\left\Vert x\right\Vert \leq a_{y},\qquad
\left\Vert I\right\Vert /\left\Vert x\right\Vert \leq \varepsilon
a_{I},\qquad \left\Vert \theta \right\Vert /\left\Vert x\right\Vert \leq
a_{\theta }.  \label{eq:cone-ai-bounds}
\end{equation}%
For $\kappa \in \left\{ x,y,\theta \right\} $ we compute%
\begin{equation}
f_{\kappa }(\varepsilon ,z_{1})-f_{\kappa }(\varepsilon ,z_{2})=\left(
\int_{0}^{1}\frac{\partial f_{\kappa }}{\partial z}\left( \varepsilon
,z_{2}+s\left( z_{1}-z_{2}\right) \right) ds\right) \left(
z_{1}-z_{2}\right) .  \label{eq:cone-der-bounds}
\end{equation}%
Using (\ref{eq:cone-der-bounds}) and (\ref{eq:cone-ai-bounds}), for $\kappa
\in \left\{ y,\theta \right\} $ we have 
\begin{align*}
& \frac{\left\Vert \pi _{\kappa }\left( f_{\kappa }(\varepsilon
,z_{1})-f(\varepsilon ,z_{2})\right) \right\Vert }{\left\Vert \pi _{x}\left(
f(\varepsilon ,z_{1})-f(\varepsilon ,z_{2})\right) \right\Vert } \\
& \leq \frac{\sum_{h\in \left\{ x,y,I,\theta \right\} }\left\Vert \left[ 
\frac{\partial f_{\kappa }}{\partial h}\right] \right\Vert \left\Vert
h\right\Vert }{m\left( \frac{\partial f_{x}}{\partial x}\right) \left\Vert
x\right\Vert -\sum_{h\in \left\{ y,I,\theta \right\} }\left\Vert \left[ 
\frac{\partial f_{\kappa }}{\partial h}\right] \right\Vert \left\Vert
h\right\Vert } \\
& \leq \frac{\left\Vert \left[ \frac{\partial f_{\kappa }}{\partial x}\right]
\right\Vert +\left\Vert \left[ \frac{\partial f_{\kappa }}{\partial y}\right]
\right\Vert a_{y}+\left\Vert \left[ \frac{\partial f_{\kappa }}{\partial I}%
\right] \right\Vert \varepsilon _{0}a_{I}+\left\Vert \left[ \frac{\partial
f_{\kappa }}{\partial \theta }\right] \right\Vert a_{\theta }}{m\left( \frac{%
\partial f_{x}}{\partial x}\right) -\left\Vert \left[ \frac{\partial f_{x}}{%
\partial y}\right] \right\Vert a_{y}-\left\Vert \left[ \frac{\partial f_{x}}{%
\partial I}\right] \right\Vert \varepsilon _{0}a_{I}-\left\Vert \left[ \frac{%
\partial f_{x}}{\partial \theta }\right] \right\Vert a_{\theta }}<b_{\kappa }
\end{align*}

We now turn to establishing (\ref{eq:param-cones-objective2}). We compute%
\begin{align}
& f_{I}(\varepsilon ,z_{1})-f_{I}(\varepsilon ,z_{2})  \notag \\
& =f_{I}(0,z_{1})-f_{I}(0,z_{2})+\varepsilon \int_{0}^{1}\left( \frac{%
\partial f_{I}}{\partial \varepsilon }\left( \varepsilon u,z_{1}\right) -%
\frac{\partial f_{I}}{\partial \varepsilon }\left( \varepsilon
u,z_{2}\right) \right) du  \notag \\
& =\pi _{I}\left( z_{1}-z_{2}\right) +\varepsilon \left(
\int_{0}^{1}\int_{0}^{1}\frac{\partial f_{I}}{\partial \varepsilon \partial z%
}\left( \varepsilon u,z_{2}+s\left( z_{1}-z_{2}\right) \right) duds\right)
\left( z_{1}-z_{2}\right)   \notag
\end{align}%
This gives%
\begin{align*}
& \frac{\left\Vert f_{I}(\varepsilon ,z_{1})-f_{I}(\varepsilon
,z_{2})\right\Vert }{\left\Vert f_{x}(\varepsilon ,z_{1})-f_{x}(\varepsilon
,z_{2})\right\Vert } \\
& \leq \frac{\left\Vert I\right\Vert +\varepsilon \sum_{h\in \left\{
x,y,I,\theta \right\} }\left\Vert \left[ \frac{\partial f_{I}}{\partial
\varepsilon \partial h}\right] \right\Vert \left\Vert h\right\Vert }{m\left( 
\frac{\partial f_{x}}{\partial x}\right) \left\Vert x\right\Vert -\sum_{h\in
\left\{ y,I,\theta \right\} }\left\Vert \left[ \frac{\partial f_{x}}{%
\partial h}\right] \right\Vert \left\Vert h\right\Vert } \\
& \leq \varepsilon \frac{a_{I}+\left\Vert \left[ \frac{\partial f_{I}}{%
\partial \varepsilon \partial x}\right] \right\Vert +\left\Vert \left[ \frac{%
\partial f_{y}}{\partial \varepsilon \partial y}\right] \right\Vert
a_{y}+\left\Vert \left[ \frac{\partial f_{I}}{\partial \varepsilon \partial I%
}\right] \right\Vert a_{I}\varepsilon _{0}+\left\Vert \left[ \frac{\partial
f_{I}}{\partial \varepsilon \partial \theta }\right] \right\Vert a_{\theta }%
}{m\left( \frac{\partial f_{x}}{\partial x}\right) -\sum_{h\in \left\{
y,\theta ,I\right\} }\left\Vert \left[ \frac{\partial f_{x}}{\partial h}%
\right] \right\Vert a_{h}}<\varepsilon b_{I}.
\end{align*}

We have established (\ref{eq:param-cones-objective1}--\ref%
{eq:param-cones-objective2}), which concludes our proof.
\end{proof}

\bibliographystyle{spmpsci}
\bibliography{12-bib}

\end{document}